\documentclass[12pt]
{amsart}
\usepackage{amsfonts,amsmath,amssymb,amsthm,graphicx,afterpage,url}
\usepackage{color}
\input{xy}
\xyoption{all}

\def\Z{{\mathbb Z}} \def\R{{\mathbb R}}  \def\C{{\mathbb C}} \def\Ha{{\mathbb H}}

\def\Cl{\mathop{\fam0 Cl}}
\def\codim{\mathop{\fam0 codim}}

\def\coker{\mathop{\fam0 coker}}

\def\diag{\mathop{\fam0 diag}}

\def\di{\mathop{\fam0 div}}

\def\ex{\mathop{\fam0 ex}}

\def\id{\mathop{\fam0 id}}
\def\inc{\mathop{\fam0 i}}
\def\Int{\mathop{\fam0 Int}}

\def\im{\mathop{\fam0 im}}
\def\lk{\mathop{\fam0 lk}}

\def\mo{\mathop{\fam0 mod}}
\def\pr{\mathop{\fam0 pr}}
\def\rk{\mathop{\fam0 rk}}

\def\SO{SO}
\def\Spin{Spin}

\def\Tors{\mathop{\fam0 Tors}}

\def\ker{\mathop{\fam0 ker}}

\def\rel#1{\allowbreak\mkern8mu{\fam0rel}\,\,#1}
\def\t{\widetilde}
\def\mk{\# }
\newcommand{\capM}[1]{\cap_{#1}}
\newcommand{\de}{\partial}
\newcommand{\nus}{\nu}
\newcommand{\nud}{D\nu}

\newcommand{\ol}{\overline}

\newcommand{\jonly}[1]{}
\newcommand{\aronly}[1]{#1}

\long\def\comment#1\endcomment{}

    \theoremstyle{theorem}
         \newtheorem{Theorem}{Theorem}[section]
         \newtheorem{Lemma}[Theorem]{Lemma}
         \newtheorem{Corollary}[Theorem]{Corollary}
         \newtheorem{Proposition}[Theorem]{Proposition}

\theoremstyle{definition}

\newtheorem{Remark}[Theorem]{Remark}

\newtheorem{Problem}[Theorem]{Problem}

\begin{document}

\comment

\bigskip
Dear Referee,

Thank you for your careful reading of our paper and critical remarks.
Please find the revised version of our paper and our reply to your comments.
The corrections comparatively to the previous version (except the bibliography) are marked in red.

Sincerely Yours, Arkadiy Skopenkov  (on behalf of both authors).

\smallskip
Lemma 2.11, Corollary 2.13, Lemma 3.9, Lemma 3.14, Lemma 3.16, before Theorem 4.10, Theorem 4.11. Corrected.

Proof of Lemma 3.1.
What we need (for the second phrase) is $g(S^4)\cap C_f =\emptyset$ not $g(S^4)\cap\partial C_f=\emptyset$.

Proof of Lemma 3.1. The following phrase is correct the way it is written:
`Since $\nus x=y$ and $y\cap_N z\in\Z$, we obtain
$y\cap_N z=\partial Ay\cap_{\partial C}\nus^!z=Ay\cap_C\widehat Az$.'
Indeed, after Remark 1.2 in `Convention on coefficients' we write
`We identify the coefficient group $\Z_d$ with $H_0(X; \Z_d)$'.

Proof of $(\lambda)$ and $(\varkappa)$ in Lemma 3.2. We have to keep the phrase `represents $\partial AX$ there' as it is because the relation between $\partial AX$ and $fX$ is explained in a separate (the third) bullet point and is more delicate than $\partial AX=fX$.

\bigskip
Dear Alexey,

Thank you for your careful reading of our paper and critical remarks,
both those we privately discussed in April and those sent to us via Editors.
Please find the revised version of our paper and our reply to your comments.
The corrections comparatively to the previous version (except the bibliography) are marked in red.

Sincerely Yours, Arkadiy Skopenkov  (on behalf of both authors).

\smallskip
{\bf Organization of the report.}
We are grateful for starting the report with high estimation of the value of our results.
It would be nice if the important words `...this does not affect the proofs, I believe' would be easier to find in the beginning of the report.
More generally, it would be helpful to our work on improving the exposition in our paper
if critical remarks could be structured as

(i) important remarks justifying referee's recommendation (rejection, major revision, minor revision),

(ii) less important remarks, and

(iii) remarks which the authors could consider on `take it or leave it' basis.


\smallskip
{\bf Language.} The referee mentions what they call "language flaws", specifically they mention missing our useless a/the/commas and strange formulations like "let us".   They also state that this is more or less harmless, but we shall take up their criticism.

As for the placing of commas, this is really a matter of taste.  Similarly, writing a sentence like "Let us prove the claim." is correct English.  We could write something like "We now prove the claim." if the referee feels strongly about this.

Finally, concerning the use of articles, this is a constant source of friction between native speakers of English and Russian.
We cannot say more than that the first author of the paper is a native English speaker and exerted significant efforts to ensure that each sentence in the paper is a correctly formed English sentence.
We are the more grateful for the referee spotting out (in a private letter not in a report) a missing article on page 2, in the statement of the Embedding Problem.

We also corrected `adjectives' to `terms' in footnote 1 according to a suggestion by referee (in a private letter not in a report).

\smallskip
{\bf Organization of the text.} We do not find `chaotic' is justified by specific critical remarks.
Many forward and backward references are due to

* the complexity of the proofs

* our efforts to make the main ideas accessible to non-specialists, by presenting main lemmas
(whose statements are simple and accessible to non-specialists, but whose proof are non-trivial and long) at the beginning of the paper.

We spent several years for writing this paper in a way as much accessible to non-specialists as possible.
The report brings no suggestions how to improve it except the agreement of numbering.

The disagreement of the numberings of subsections on the one hand, and statements on the other, is usual.
The numbering is up to the standards of MMJ (or some interjournal standards), on which we do not have influence.

Also it is not justified that having a separate section for every lemma would make the paper clearer.
We think that having a separate subsection for any of about 50 lemmas would make the text less readable.
E.g. think about having a separate subsections for lemmas 1.4 and 1.5.

The remark concerning `below in \S4.4' is unjustified.
Indeed, `below in \S4.4' means `in the part of \S4.4 below this phrase' not
`in section 4.4 which is a section presented below '

\smallskip
{\bf (1).} Thank you.
We added to \S2.1 the phrase `Throughout this paper `manifold' is shorthand for `compact oriented manifold, possibly with non-empty boundary', and explicitly added `oriented' to the relevant phrase in \S1.3. (We also replaced `compact manifold' and `compact oriented manifold' to `manifold' throughout.)

\smallskip
{\bf (2a).} Thank you, we added  `compatible with the orientation of $Q$'.

\smallskip
{\bf (2b).} In the submitted version we explicitly (not tacitly) adopt the normal point of view:
we write

`A {\em spin structure} (more precisely, {\it stable tangent spin structure}) on $Q$ is a stable tangent framing over the 2-skeleton of $Q$.
Two spin structures on $Q$ are {\it equivalent} if their restrictions to the $1$-skeleton are  homotopic.'

In the revision we corrected the sentences to make the point more explicit.

\smallskip
{\bf (2c, first).} In your remark you explain the difference between a spin structure and a stable spin structure.
This contradicts to `there is no need here to "stabilize" the tangent bundle' (i.e. to use stable spin structure not spin structure).

We edited the corresponding text to make this even more clear.

\smallskip
{\bf (2c, second).} The existence of such a 1--1 correspondence need not be matter-of-fact to a non-specialist, so we have chosen to explicitly state it as (well-known) Lemma 4.2.
Although normal and tangent spin structures are in natural 1--1 correspondence, in the arguments it is helpful
for a non-specialist to see which of the two objects is considered.

\smallskip
{\bf (2d).} Thank you, we explicitly defined `stable vector bundle'.
Thus no confusion with the different notion of algebraic geometry is possible.

\smallskip
{\bf (3).} Thank you, we shortened the phrase.

\smallskip
{\bf (4a).} We did specify what we meant by a stable normal spin structure, see the phrase `a (stable) normal spin structure, i.e. a stable normal framing over the 2-skeleton of some triangulation' of the submitted version.
In the revision we give more detailed definition in \S3.1.

\smallskip
{\bf (4b, 1st sentence).}  In our understanding `due cite' of a lemma is just before its applications
(unless the lemma is as important as to be presented in the introduction).
Since normal bundles and normal spin structures are not used before \S4, we do not want to bother a non-specialist reader by mentioning them in \S3 (and by thus asking to remember this until \S4).
So the lemma is presented in its due cite, unless the referee explains his understanding of `due cite' and
justify that having the lemma in another place  would be more convenient to a reader.

In our understanding `due formulation' of a lemma is the formulation used later.
So the lemma is presented in its due formulation, unless the referee explains his understanding of `due formulation' and justify that having another formulation  would be more convenient to a reader.

\smallskip
{\bf (4b, 2nd sentence).} This remark is superfluous because the sentence before Lemma 4.2 reads
`The following lemma is well-known'.

\smallskip
{\bf About `homological Alexander duality'.} Our use of `homological Alexander duality' is justified in footnote 4. Hence it is wrong that the homological Alexander duality have nothing to do
with Alexander duality (as in the report).



You consider $N$ as a subset of $S^m$, i.e. you fix an embedding $f:N\to S^m$,
while the main point of our paper is classification of  different $f$'s.

The isomorphism $A$ is not the inverse of the connecting homomorphism $\partial$
(as in the report), but is the composition of the inverse with the excision isomorphism.
The latter should not be considered  canonical in this situation (because $f$ is not fixed),
so it should not be confused with the identity.
More importantly, although such an interpretation of $A$ is correct
and was known to us, it provides no simple geometric interpretation
that we can use in calculations required for our proof.
We introduce $A$ exactly for such calculations.
There is analogous remark on $\widehat A$.

In the report you tacitly identify groups $H_{i+1}(S^n;fN)$ and $H_{i+1}(C,\partial C)$
related by a non-canonical (excision) isomorphism, see above.

\endcomment

\newpage
\title{Embeddings of non-simply-connected 4-manifolds in 7-space. I. Classification modulo knots}


\author{D. Crowley and A. Skopenkov}

\thanks{We would like to acknowledge S.~Avvakumov, M.~Kreck, S.~Melikhov, D.~Tonkonog, A. Zhubr and anonymous referee for useful discussions.
We would like to thank the Hausdorff Institute for Mathematics and the University of Bonn for their hospitality and support during the early stages of this project.
A. Skopenkov is supported in part by the Russian Foundation for Basic Research Grants No. 15-01-06302 and 19-01-00169, by Simons-IUM Fellowship and by the D. Zimin Dynasty Foundation.}

\subjclass[2010]{57R40, 57R52 (primary), 57R67, 57Q35, 55R15 (secondary)}

\keywords{Embedding, isotopy, 4-manifolds, surgery obstructions, spin structure}

\date{}

\maketitle

\abstract
We work in the smooth category.
Let $N$ be a closed connected orientable 4-manifold with torsion free $H_1$, where $H_q:=H_q(N; \Z)$.
The main result is {\it a complete readily calculable classification of embeddings $N\to\R^7$},
up to equivalence generated by isotopies and embedded connected sums with embeddings $S^4\to\R^7$.
Such a classification was earlier known only for $H_1=0$ by Bo\'echat-Haefliger-Hudson 1970.
Our classification involves the Bo\'echat-Haefliger invariant $\varkappa(f)\in H_2$,
Seifert bilinear form $\lambda(f):H_3\times H_3\to\Z$ and $\beta$-invariant assuming values in the quotient of $H_1$ defined by values of $\varkappa(f)$ and $\lambda(f)$.
In particular, for $N=S^1\times S^3$ we define geometrically a 1--1 correspondence between the set of
equivalence classes of embeddings and an explicitly defined quotient of
$\Z\oplus\Z$.

Our proof is based on development of Kreck modified surgery approach,
involving some simpler reformulations, and also uses parametric
connected sum.
\endabstract



\tableofcontents


\section{Introduction and main results} \label{s:intro}

In \S\ref{s:intro-kno} we provide a broader context for the results in this paper.
The main results are described in \S\ref{s:intro-kno} and formally stated in \S\ref{s:intro-main}.
Our ideas are described in \S\ref{s:intro-strat} and formally realized in \S\S\ref{s:defandplan}-\ref{s:eta-mk}.
Except for the definitions in bold, `some notation' in bold and Lemma \ref{l:isored}, the material of \S\ref{s:intro-kno} and \S\ref{s:intro-strat} is not formally used in the rest of this paper.

\subsection{The Knotting Problem}\label{s:intro-kno}

Let us start with a citation of \cite{MAE}.

`Three important classical problems in topology are the following, cf. \cite[p. 3]{Ze93}.

{\it The Manifold Problem:} Classify $n$-manifolds.

{\it The Embedding Problem:} Find the least dimension $m$ such that a given manifold admits an embedding into $m$-dimensional Euclidean space $\R^m$.

{\it The Knotting Problem:} Classify embeddings of a given manifold into another given manifold up to isotopy.

The Embedding and Knotting Problems have played an outstanding role in the development of topology. Various methods for the investigation of these problems were created by such classical figures as G. Alexander, H. Hopf, 
L. S. Pontryagin, R. Thom, H. Whitney, M. Atiyah, F. Hirzebruch, R. Penrose, J. H. C. Whitehead, C. Zeeman, W. Browder, J. Levine, S. P. Novikov, A. Haefliger, M. Hirsch, J. F. P. Hudson, M. Irwin and others.'


The Knotting Problem is related to other branches of mathematics, most importantly, to algebraic topology.
For recent surveys see \cite{RS99, Sk08, MAE}; whenever possible we refer to these surveys not to original papers.


This paper is on the Knotting Problem for 4-manifolds.
We consider {\it smooth} manifolds, embeddings and isotopies.
By a classification we mean {\it a complete, readily calculable} classification.\footnote{For a discussion of the terms `smooth' and `readily calculable'
see \jonly{\cite[Remark 2.20]{I},} \aronly{Remark \ref{r:tos1},} \cite[Remark 1.2]{MAE}.}
For $m\ge n+2$ the classifications of embeddings of compact $n$-manifolds into $S^m$ and into $\R^m$ are the same, see details in \cite[Remark 1.3]{MAE}.
In this subsection $P$ is a closed connected $n$-manifold.

\smallskip
{\bf Definition of $E^m(P)$.}
Let $E^m(P)$ be the set of isotopy classes of embeddings $f \colon P \to S^m$, where $S^m \subset \R^{m+1} $ is the unit $m$-sphere.

\smallskip
The Knotting Problem is more accessible for
$$
2m\ge3n+4,
$$
where there are some classical
classifications of embeddings, which are surveyed in \cite[\S2, \S3]{Sk08}, \cite{MAE}.
When $n=4$, classification was obtained for $m\ge 9$ by Whitney--Wu, for $m=8$ by Haefliger--Hirsch,
and for $N=S^4$ and $m=7$ by Haefliger,
giving
$$
|E^m(P)|=1\quad\text{for}\quad m\ge9,\quad E^8(P)=H_1(P;\Z_2),\quad E^7(S^4)\cong\Z_{12}.
$$
Here the equality sign between sets denotes the existence of a `geometrically defined' bijection,
and the isomorphism is a group isomorphism for the group structure defined below.
Any orientable 4-manifold embeds into $\R^7$ (this follows easily from results of Donaldson and Boechat-Haefliger, see detailed references in \cite{MAM}).\footnote{For more information and references see \cite{MAM}.
For classifications of embeddings of $n$-manifolds into $\R^{2n-1}$ when $n\ne4$ see \cite{Ya84, Sa99, Sk10', To10}.}

The Knotting Problem is much harder for $2m<3n+4$.
If $P$ is a closed manifold that is not a disjoint union of homology spheres, then until recently no
classification of embeddings $P\to\R^m$ was known.
This is in spite of the existence of many interesting approaches including methods of Haefliger-Weber, Browder-Wall and Goodwillie-Weiss \cite[\S5]{Sk08}, \cite{Wa70, GW99, CRS04}, see \cite[Remark 1.2]{MAE}.

Recent
classifications for $2m<3n+4$ concern

$\bullet$ embeddings of 3- and 4-dimensional manifolds \cite{Sk08', Sk10, CS11},

$\bullet$ embeddings of $d$-connected $n$-manifolds for $2m\ge3n+3-d$ \cite{Sk02}, and

$\bullet$ embeddings $S^p\times S^q\to S^m$ \cite{CRS07, CRS12, CFS14, Sk15}.

These results are based on three productive approaches.
One of them involves almost embeddings and the $\beta$-invariant of \cite{Sk02, Sk07, Sk14, CRS07, CRS12}
(which, though related to, is different from the $\beta$-invariant in this paper),
another is based on relations between different sets of embeddings \cite{Sk11, Sk15}.
However, these and other approaches are not sufficient to classify the embeddings of $4$-manifolds $N$
into $S^7$, even in the case of $N = S^1\times S^3$.
In this paper we develop the approach which uses the modified surgery of M. Kreck, see \S\ref{s:intro-strat} and~\cite{Sk08', Sk10, CS11}.

This paper is the first and most important one in the program for the classification of (smooth and piecewise linear) embeddings into $S^m$ of {\it non-simply connected} 4-manifolds (under the `torsion free' condition and up to an indeterminancy in certain cases).
See  \cite{II, III}; some parts of those results easily follow from this paper, but we state those parts in \cite{II, III} not here.
In order to explain what is done in the present paper, let us recall some more definitions.


Denote by $[f]$ the isotopy class of an embedding $f \colon P \to S^m$.

\smallskip
{\bf Definitions of the group $E^m(S^n)$ for $m\ge n+3$, of its action on $E^m(P)$ and of $q_{\mk}$, $E^m_{\mk}(P)$.}
Represent elements of $E^m(P)$ and of $E^m(S^n)$ by embeddings $f:P\to S^m$ and $g:S^n\to S^m$ whose images are contained in disjoint balls.
Join the images of $f,g$ by an arc whose interior misses the images.
Let $[f]\#[g]$ be the isotopy class of the {\it embedded connected sum} of $f$ and $g$ along this arc,
cf. \cite[Theorem 1.7]{Ha66},
\cite[\S1]{Av16}.
The isotopy class of the embedded connected sum depends only on the isotopy classes $[f]$ and $[g]$ \cite[\S5]{MAE}.
Hence we can define the operation
\[
\# \colon E^m(P) \times E^m(S^n) \to E^m(P)\quad \text{by}\quad ([f], [g]) \mapsto [f] \#[g].
\]
For $P=S^n$ and $m \geq n+3$ this  defines a group structure on $E^m(S^n)$ \cite{Ha66}.
Clearly $\#$ is an action of $E^m(S^n)$ on the set $E^m(P)$.
We define
\[
E^m_{\mk}(P) := E^m(P)/E^m(S^n).
\]
to be the quotient of this action, i.e. the set of `embeddings modulo knots'.
Let $q_{\mk}:E^m(P)\to E^m_{\mk}(P)$ be the quotient map.


\medskip
A simpler version of the knotting problem is to classify the set $E^m_{\mk}(P)$.
For $n=4$ Bo\'echat and Haefliger classified $E^7_{\mk}(P)$ when $H_1(P; \Z) = 0$ \cite{BH70}.
The action of the knots was investigated in \cite{Sk10} and determined
when $H_1(P; \Z) = 0$ in \cite{CS11}, which also classified $E^7(P)$ in this case.
In general, even the simpler version is hard: If $2m<3n+4$ and $P$ is a closed manifold that is not (homologically)  $[(n{-}2)/2]$-connected, then until recently no
classification of $E^m_{\mk}(P)$ was known.


The main result of this paper is a classification of $E^7_{\mk}(P)$ when $P$ is a closed connected orientable 4-manifold with torsion free $H_1(P; \Z)$; see Theorems \ref{t:s1s3} and \ref{t:gen} below.
This requires finding a complete set of invariants and constructing embeddings realizing particular values of these invariants.
Although the invariants come from modified surgery theory (see \S\ref{s:intro-strat}), we work with them using basic algebraic topology.
Thus we make this paper and its relation to other work (independent of surgery) accessible to non-specialists:
e.g. we emphasize relations between our $\lambda$-invariant and \cite{Sa99, To10}.
Lemma \ref{l:calam} and \S\ref{s:defandplan-kl}, \S\ref{s:defandplan-b} describe the invariants we use.
In \S\ref{s:intro-strat} we explain how the invariants appear in our approach to classification and give an overview of the proof of their completeness.
The beginning of \S\ref{s:intro-main} gives explicit construction of embeddings $S^1\times S^3\to S^7$.
For our general $P$, we use a parametric connected sum operation on embeddings which is described in \S\ref{s:defandplan-pr}.
We create new embeddings $P\to S^7$ from a fixed embedding $f_0:P\to S^7$ using parametric connected sum with embeddings $S^1 \times S^3 \to S^7$.
So embeddings of $S^1 \times S^3$ play a key role in the classification of embeddings for general $P$.


\subsection{Main results}\label{s:intro-main}

We first define a family of embeddings $\tau_\alpha \colon S^1 \times S^3\to S^7$
and a corresponding map
$$
\tau \colon \Z^2 \to E^7(S^1 \times S^3).
$$
Let $V_{4, 2}$ denote the Stiefel manifold of orthonormal $2$-frames in $\R^4$.
Take a smooth map $\alpha:S^3\to V_{4,2}$.
Regarding $V_{4, 2}\subset (\R^4)^2$, write $\alpha(x) = (\alpha_1(x), \alpha_2(x))$.
Define the adjunction map $\R^2 \times S^3 \to \R^4$ by $((s, t), x) \mapsto \alpha_1(x)s + \alpha_2(x)t$.
(Regarding $V_{4, 2}\subset (\R^4)^{\R^2}$, this map is obtained from $\alpha$ by the exponential law $Z^{X\times Y}=(Z^X)^Y$.)
Denote by $\ol\alpha:S^1\times S^3\to S^3$ the restriction of the adjunction map.
We define the embedding $\tau_\alpha$ to be the composition
$$
S^1\times S^3 \xrightarrow{~\ol \alpha\times\pr_2~}
S^3\times S^3 \xrightarrow{~\inc~} S^7,\quad\text{where}\quad \inc(x,y):=(y, x)/\sqrt{2}\quad\text{and}\quad
\pr\phantom{}_2(x, y) = y.
$$
We define the map $\tau$ by $\tau(l, b):=[\tau_{\alpha}]$, where $\alpha\colon S^3 \to V_{4, 2}$ represents
$(l, b) \in \pi_3(V_{4, 2})$ (for the standard identification $\pi_3(V_{4, 2})=\Z^2$ described in \S\ref{s:defandplan-not}).

We define $\tau_{\mk}:=q_{\mk}\tau$.



\begin{Theorem} \label{t:s1s3}
The map
$$
\tau_{\mk}:\Z^2\to E_{\mk}^7(S^1\times S^3)
$$
is a surjection such that
\[
\tau_{\mk}(l,b)=\tau_{\mk}(l',b')\quad\Longleftrightarrow\quad(\ l=l' \quad\text{and}\quad b\equiv b'\mo2l\ ).
\]
\end{Theorem}

\begin{Remark}\label{r:s1s3}
(a) There is a map $\tau:\pi_q(V_{m-q,p+1})\to E^m(S^p\times S^q)$
which is defined analogously to above.
For $m\ge2p+q+3$ the sets $E^m(S^p\times S^q)$ and $E^m_\#(S^p\times S^q)$ admit a group structure such that $\tau$ and $q_{\mk}$ are homomorphisms \cite{Sk15}.
For $p\le q$ and $2m\ge2p+3q+4$ (conjecturally for $2m\ge p+3q+4$) the map $\tau_\#:=q_{\mk}\tau$ is an isomorphism.
Theorem \ref{t:s1s3} shows that the case of embeddings $S^1\times S^3\to S^7$ is different:
there are no group structures on $E^7(S^1\times S^3)$ or on $E^7_\#(S^1\times S^3)$ such that $\tau$ or $\tau_\#$ is a homomorphism
(because by Theorem \ref{t:s1s3} $\tau_\#^{-1}\tau_\#(0,0)$ is infinite while $\tau_\#^{-1}\tau_\#(1,0)$ is finite; analogous argument works for $\tau$ since $E^7(S^4)$ is finite).

(b) The isotopy classes $\tau(1,0)$ and $\tau(0,1)$ are represented by embeddings
$$
S^1\times S^3 \xrightarrow{\pr_2\times T^k} S^3\times S^3 \xrightarrow{~\inc~} S^7,
$$
where the maps $T^k:S^1\times S^3\to S^3$ are defined as follows:

$\bullet$ $T^1(s,y):=sy$, where $S^3$ is identified with the set of unit length quaternions and
$S^1\subset S^3$ with the set of unit length complex numbers;

$\bullet$ $T^2(e^{i\theta},y):=\eta(y)\cos\theta+\sin\theta$, where $\eta:S^3\to S^2$ is the Hopf map, and
$S^2$ here is identified with the 2-sphere formed by unit length quaternions of the form $ai+bj+ck$.

For other constructions see \cite[Examples of knotted tori]{MAM}.
\end{Remark}




Before stating our main results for the general case, we establish some conventions, notation and definitions.

\smallskip
{\bf Convention on coefficients.}
Unless otherwise stated, we omit $\Z$-coefficients from the notation of (co)ho\-mo\-lo\-gy groups.
We identify the coefficient group $\Z_d$ with $H_0(X; \Z_d)$,
the zero-dimensional homology group of a connected oriented manifold $X$.

\smallskip
{\bf Notation for $N$, $f$, characteristic classes and intersections in manifolds.}
Throughout this paper $N$ is a closed connected oriented $4$-manifold and $f \colon N \to S^7$ an embedding.
Let $H_q:=H_q(N)$.
We denote the Poincar\'e dual of a characteristic class by adding a superscript `$*$',
so for example $w_2^*(N) \in H_2(N; \Z_2)$ is the Poincar\'{e} dual of the second Stiefel-Whitney class.
The homology intersection product in an oriented $n$-manifold $M$ is denoted by
$$\capM{M}\colon H_i(M) \times H_j(M) \to H_{n-i-j}(M).$$
The well-known definition of such product is recalled in
\cite{MAI}.
For the intersection {\it powers} we omit subscripts indicating the manifold $M$,
so, for example, $x^2$ denotes $x \capM{M} x$.
Let $\rho_n$ be the reduction modulo $n$.
The intersection $x\capM{M} y$ of a $\Z$-homology class $x$ and a $\Z_n$-homology class $y$ is defined as the
$\Z_n$-homology class $\rho_nx\capM{M} y$.
Let $\sigma(N)$ be the signature of the intersection form $H_2\times H_2\to\Z$.

\smallskip
{\it If $H_1=0$, then the map
$$
\varkappa_\mk:E_{\mk}^7(N)\to H_2^{DIFF}:=\{u\in H_2\ |\ \rho_2u=w_2^*(N),\ u^2=\sigma(N)\}\subset H_2
$$
(which is the Bo\'echat-Haefliger invariant defined below)
is 1--1.} \label{pst:bh} This statement is easily deduced from known results in \jonly{\cite[Remark 2.21.e]{I}.}
\aronly{Remark \ref{r:bil}.e.}
Our second main result is a generalization of this statement to non-simply-connected 4-manifolds.

\smallskip
{\bf Definition of $\di$, $B(H_3)$, $\overline l$ and a symmetric pair.}
For an element $u$ of a free abelian group denote by $\di u$ the divisibility of $u$,
i.e. $\di 0=0$ and $\di u$ is the largest integer which divides $u$ for $u\ne0$.
For an element $u$ of an abelian group $G$ denote by $\di u$ the divisibility of $[u]\in G/\Tors(G)$.

Denote by $B(H_3)$ the space of bilinear forms $H_3\times H_3\to\Z$.
For $l\in B(H_3)$ denote by $\overline l:H_3\to H_1$ the adjoint homomorphism uniquely defined by the property
$l(x,y)=x\capM{N} \overline ly$.
A pair $(u,l)\in H_2\times B(H_3)$ is called {\it symmetric} if
$$l(y,x)=l(x,y)+u\capM{N} x\capM{N} y \quad\text{for all}\quad x,y\in H_3.$$

The maps
$$
\varkappa\colon E^7(N) \to H_2^{DIFF},\quad \lambda\colon E^7(N) \to B(H_3)\quad\text{and}
$$
$$
\beta_{u,l}: (\varkappa\times\lambda)^{-1}(u,l)
\to C_{u,l}:=\coker(2\rho_{\di(u)}\overline l)=\frac{H_1}{2\overline l(H_3)+\di(u)H_1}
$$
required for Theorem \ref{t:gen} below are defined in \S\ref{s:defandplan-kl} and \S\ref{s:defandplan-b}.
Then the maps
$$
\varkappa_{\mk}:E^7_{\mk}(N)\to H_2^{DIFF},\quad \lambda_{\mk}\colon E^7_{\mk}(N) \to B(H_3)\quad\text{and}\quad \beta_{u,l,\mk}:(\varkappa_{\mk}\times\lambda_{\mk})^{-1}(u,l)\to C_{u,l}
$$
of Theorem \ref{t:gen} below are well-defined by $\varkappa=\varkappa_{\mk}q_{\mk}$,
$\lambda=\lambda_{\mk}q_{\mk}$ and $\beta_{u,l}=\beta_{u,l,\mk}(q_{\mk}\times q_{\mk})$
because of the additivity (Lemmas \ref{l:Addl} and \ref{l:Add} below).

In order to avoid double statements of similar properties, we use the following convention:
a statement involving $\varkappa$  holds for both $\varkappa$ and $\varkappa_{\mk}$.
If a statement holds for $\varkappa_{\mk}$ but not for  $\varkappa$, we write $\varkappa_{\mk}$ in the formulation.
In this paper there are no statements which hold for $\varkappa$ but not for  $\varkappa_{\mk}$.
Analogous remark holds for $\lambda$ vs $\lambda_{\mk}$, $\beta_{u,l}$ vs $\beta_{u,l,\mk}$ etc.

\begin{Theorem}\label{t:gen}
Let $N$ be a closed connected orientable 4-manifold with torsion free $H_1$.
Then the product
$$\varkappa_{\mk}\times\lambda_{\mk}:E_{\mk}^7(N)\to H_2^{DIFF}\times B(H_3)$$
has non-empty image consisting of all symmetric pairs,
and for every $(u,l)\in \im (\varkappa_{\mk}\times\lambda_{\mk})$
each map $\beta_{u,l,\mk}$ is 1--1 (see the remark immediately below).
\end{Theorem}


We call geometrically defined maps invariants.
In particular, the maps $\lambda$ and $\varkappa$ are invariants.


\smallskip
{\bf Remark on relative invariants.}
The map $\beta_{u,l}$ is
a {\em relative invariant}.
By this we mean that for $[f_0], [f_1] \in (\lambda\times \varkappa)^{-1}(u, l)$
there is an
invariant $([f_0], [f_1]) \mapsto \beta(f_0, f_1)$ (defined in \S\ref{s:defandplan-b}) and that
$\beta_{u,l}(f) := \beta(f,f')$ for a fixed choice of $[f']\in (\lambda\times \varkappa)^{-1}(u, l)$.
We
suppress the choice of $[f']$ from the notation.


\smallskip
The Seifert bilinear form $\lambda(f)\colon H_3 \times H_3 \to \Z$ (defined in \S\ref{s:defandplan-kl})
measures the linking of $3$-cycles in $N$ under $f$.
For $N=S^1\times S^3$ identify $B(H_3)$ with $\Z$.

\begin{Lemma}[Calculation for $\lambda$; proved in \S\ref{s:defandplan-kl}]\label{l:calam}
(a) For an embedding $f:S^1 \times S^3\to S^7$ we have
$$\lambda(f)=\lk\phantom{}_{S^7}(f|_{(1,0)\times S^3},f|_{(-1,0)\times S^3})\in \Z.$$	

(b) $\lambda(\tau(l, b)) = l$.

(c) We have $\lambda(f)(x,y)=\lk_{S^7}(f|_X,f|_Y)$ if classes $x$ and $y$ are represented by disjoint closed oriented 3-submanifolds (or integer 3-cycles) $X$ and $Y$.
\end{Lemma}


See \jonly{\cite[Remark 2.21]{I}.} \aronly{Remark \ref{r:bil}.}

\subsection{An approach to the Knotting Problem}\label{s:intro-strat}\nopagebreak

The proofs of our main results are based on the ideas we explain below.
These ideas are useful in a wider range of dimensions \cite{Sk08'} and for problems other than classification of embeddings \cite{Kr99}.
In this paper we do not assume the reader is familiar with surgery.
Hence we describe the application of modified surgery in non-specialist terms and make parenthetical remarks for specialists.

\smallskip
{\bf Some notation.}
Take the standard orientation on $\R^m$.
For an oriented manifold with boundary we use the orientation on the boundary whose completion by `the first vector pointing outside' gives the orientation on the manifold.
So an orientation of $S^{m-1}=\de D^m$ is defined.
Denote by

$\bullet$ $N$ a closed connected oriented $4$-manifold and $f \colon N \to S^7$ an embedding;

$\bullet$ $C=C_f$ the closure of the complement in $S^7$ to a sufficiently small
tubular neighborhood of $f(N)$; the orientation on $C$ is inherited from the orientation of $S^7$;

$\bullet$ $\nus=\nus_f:\de C\to N$ the sphere subbundle of the normal vector bundle of $f$:
the total space of $\nus$ is identified with $\partial C$;

In this paper a {\it bundle isomorphism} is an oriented vector bundle
isomorphism identical on the base, or the restriction to the sphere bundle of such.
In this and other notation we sometimes omit the subscript $f$.
We shall also change the subscript `$f_k$' to `$k$'.


\begin{Lemma}\label{l:isored} For a closed connected 4-manifold $N$
two embeddings $f_0,f_1:N\to S^7$ are isotopic if and only if there is an orientation preserving diffeomorphism
$C_0\to C_1$ whose restriction to the boundary $\de C_0\to\de C_1$ is a bundle isomorphism.
\end{Lemma}

Lemma \ref{l:isored} is well-known to the experts (for a proof see, e.g., \cite[Lemma 1.3]{Sk10}) and holds in more general situations.

\begin{Remark} We shall not only decide if there is a diffeomorphism $C_0 \to C_1$ as in Lemma \ref{l:isored}
but we also prove a general `relative diffeomorphism criterion' for certain 7-manifolds with boundary.
This is the Almost Diffeomorphism Theorem \ref{t:aldi}.
It generalizes \cite[Almost Diffeomorphism Theorem 2.8, Diffeomorphism Theorem 4.7]{CS11}.
It is a new non-trivial analogue of \cite[Theorem 3.1]{KS91} and of \cite[Theorem 6]{Kr99}
for 7-manifolds $M$ with non-empty boundary and infinite $H_4(M)$.
\end{Remark}

Lemma \ref{l:isored} reduces the classification of embeddings to a two-step classification problem for their complements.
Firstly, we classify the complements relative to fixed identifications of their boundaries,
and secondly we determine the action of the bundle automorphisms on the relative diffeomorphism classes of the complements.
This is the starting point of both the classical and modified surgery approaches.

To find out if there exists a diffeomorphism $C_0\to C_1$ as in Lemma \ref{l:isored} using classical surgery (see \cite{Wa70}),
we would first need decide if the complements $C_0$ and $C_1$ have the same homotopy type.
If they do, then we take homotopy equivalence $h:C_0\to C_1$ and apply surgery relative to the boundary to the Poincar\'{e} pairs $(C_1, h(\de C_0))$ and $(C_1,\de C_1)$.




In this paper to determine if there is a diffeomorphism $C_0 \to C_1$ as in Lemma \ref{l:isored} we use modified surgery \cite{Kr99}; cf.~\cite[Remark 2.2]{CS11} and the text after it.
For this we fix for $k=0,1$

$\bullet$ the spin structures on $C_k$ which they inherit from $S^7$ and also

$\bullet$ the {\it Seifert classes} which are algebraic analogues of Seifert surfaces; the Seifert classes are the relative homology classes $A_k[N]\in H_5(C_k,\de C_k)\cong\Z$, which are the images of the fundamental class of $N$ under homology Alexander duality (defined in \S\ref{s:prelemmas-not}).

(This data on $C_k$ defines a {\it normal 2-smoothing} of $C_k$ \cite[\S 2]{Kr99},
i.e.~a normal $3$-equivalence $C_k\to B\Spin\times\C P^\infty$.
We use a particular case of the modified surgery approach which corresponds to spin surgery
over the {\it homotopy $2$-type} of the complement.)

The modified surgery approach to the embedding problem requires
that we find a bundle isomorphism $\varphi \colon \partial C_0 \to \de C_1$
preserving both the spin structures and the homology classes $\de A_k[N] \in H_3(\de C_k)$.
We prove that {\it there is always a bundle isomorphism $\varphi:\de C_0\to\de C_1$ preserving spin structure}
(Lemma \ref{l:spbuis}, cf.~\cite[Lemma 2.4]{CS11}).
The first obstruction we encounter to the existence of a diffeomorphism as in Lemma \ref{l:isored} is the difference
$\varphi_*\de A_0[N]-\de A_1[N]\in H_3(\de C_1)$.
The analysis of this obstruction leads to the definition of
$\varkappa$-invariant (see \S\ref{s:defandplan-kl} and \jonly{\cite[Remark 2.22]{I}}\aronly{Remark \ref{r:tobh}})
\[
\varkappa \colon E^7(N)\to H_2.
\]
{\it Assume further that $\varkappa(f_0)=\varkappa (f_1)$.}
We prove that {\it $\varphi_*\de A_0[N]=\de A_1[N]$ for every bundle isomorphism $\varphi:\de C_0\to\partial C_1$}
(Lemma \ref{l:Agr}.a for $q=4$, \cite[Agreement Lemma 2.5]{CS11}).
We then identify the spin boundaries of $(C_0,A_0[N])$ and $(C_1,A_1[N])$
via a bundle isomorphism $\varphi$ which preserves the spin structures.
For any such identification {\it there is a spin bordism $(W,z)$ between $(C_0,A_0[N])$ and $(C_1,A_1[N])$ relative to the boundaries}
(because the complete obstruction to the existence of such a bordism assumes values in $\Omega_7^{Spin}(\C P^\infty)=0$ \cite[Lemma 6.1]{KS91}).
It remains to determine whether we can replace the bordism $(W,z)$ by an $h$-cobordism.
This problem is addressed in \cite[Theorem 3]{Kr99}, where a complete algebraic obstruction
is defined.
Analysis of the obstruction for the bordism $(W,z)$ to have the homology of an $h$-cobordism `outside $H_4(W)$' leads to the definition of $\lambda$-invariant (see \S\ref{s:defandplan-kl} and
\jonly{\cite[Remark 2.22]{I}}\aronly{Remark \ref{r:tobh}})
\[
\lambda \colon E^7(N)\to B(H_3).
\]
{\it Assume further that $\lambda(f_0)=\lambda(f_1)$.}
From the surgery point of view, the $\beta$-invariant (see definition in \S\ref{s:defandplan-b}) arises as
the obstruction for the bordism $(W,z)$ to have the homology of an $h$-cobordism `in the summand of $H_4(W)$ coming from $H_4(\de W)$'
(i.e.~in the {\it singular} part of the intersection form on $H_4(W)$).
This invariant assumes values in a quotient of $H_1$ defined by
$\varkappa(f_0)$ and $\lambda(f_0)$.

{\it Assume further that $\beta(f_0)=\beta(f_1)$.}
We may now assume that the bordism $(W, z)$ `has the homology of an $h$-cobordism' away from the {\it unimodular} part of $H_4(W)$.
Extending arguments from \cite{CS11}, we prove that we can modifiy $f_0$ by connected sum with a knot
$g\colon S^4\to S^7$ so that for some corresponding

$\bullet$ spin bundle isomorphism $\varphi':\de C_{f_0\# g}\to\de C_1$,

$\bullet$ identification of the spin boundaries of the pairs $(C_{f_0\# g},A_{f_0\# g}[N])$ and $(C_1,A_1[N])$,

$\bullet$ spin null bordism $(W',z')$ between the above pairs, relative to the boundaries,

the pair $(W',z')$ is bordant to an $h$-cobordism.
Then by the $h$-cobordism theorem \cite{Mi65} and Lemma \ref{l:isored}, $f_0 \# g$ and $f_1$ are isotopic.

\smallskip
The above discussion outlines the proof that the $\varkappa$-, $\lambda$- and $\beta$-invariants
combine to give a complete systems of invariants for embeddings modulo knots.
This is stated in the MK Isotopy Classification Theorem \ref{l:isomk} and
the behaviour of these invariants under connected sum with knots is described in
the Additivity Lemmas \ref{l:Addl}, \ref{l:Add}.

\smallskip
{\bf Plan of the paper.}
We introduce further notation in \S\ref{s:defandplan-not} and \S\ref{s:prelemmas-not}.
In \S\ref{s:defandplan} we present the important constructions and lemmas
used in the proof of our main results.
The lemmas from \S2 are proven in \S\ref{s:prelemmas} and \S\ref{s:eta-mk}.
The subsection titles in \S3 indicate the most important lemmas proven in that subsection.
A reader who wants to check a particular lemma from \S\ref{s:defandplan} does not need to read all of \S\ref{s:prelemmas} and \S\ref{s:eta-mk}.



\comment



Omit this and the next paragraphs?
More formally, we prove that there is a class $Y\in H_5(C_0\cup_\varphi(-C_1))$ whose intersections with
$C_0$ and $C_1$ are $A_0[N]$ and $A_1[N]$, respectively.
Such a class is called {\it a joint Seifert class} for $f_0,f_1,\varphi$.
Denote $d:=\di(\varkappa(f_0))$.
The $\beta$-invariant arises as the obstruction for the existence of a joint Seifert class $Y$ and a !!!$\pi$-iso!!! $\varphi$ with $\rho_dY^2=0$.

If $(C_0,A_0[N]) \cong (C_1,A_1[N])$, then one can prove that?
$Y^2=0\in H_3(C_0\cup_\varphi(-C_1))$.
There are split surjections $H_3(C_0 \cup_\varphi (-C_1)) \to H_1$, and for a candidate??? $Y$, the $\beta$-invariant is a certain equivalence
class to be the projection of $Y^2 \to H_1$.
Different choices of splittings??? and different choices of
joint Seifert class [of !!!$\pi$-iso!!! $\varphi$???] lead to a well-defined quantity
\[
\beta(f_0, f_1) \in \coker(2\rho_d\overline{\lambda(f_0)})
\]
where we recall that $\overline{\lambda(f_0)}\colon H_3 \to H_1$ is the adjoint of $\lambda(f_0)$.
\remas{In the paper we realize this plan differently (see \S\ref{s:defandplan-b}, Lemma \ref{l:webeta} and the text before it).}

\remas{Upgrade some remarks to Proposition or Corollary?}

Here $\tau(a,b,u,l)=a\#\tau_{\mk}(u,l,b)$. (is well-defined?)

\bigskip
{\bf Main result for $\C P^2$ and $S^2\times S^2$ \cite{Sk10}, \cite{CS11} (delete from \S1?)}

\begin{Theorem} \label{t:cp2} \cite{Sk10}, \cite{CS11}
There are exactly two isotopy classes of embeddings $\C P^2\to\R^7$.
\end{Theorem}

The standard embedding $\C P^2\to\R^7$ is given by
$$(x:y:z)\mapsto(x\overline y, y\overline z, z\overline x,2|x|^2+|y|^2),\quad\text{where}\quad |x|^2+|y|^2+|z|^2=1.$$
The two isotopy classes are represented by the standard embedding and its composition with the symmetry w.r.t. $\R^6\subset\R^7$.

{\bf Addendum.} {\it For every pair of embeddings $f:\C P^2\to\R^7$ and
$g:S^4\to\R^7$ the embedding $f\#g$ is isotopic to $f$.}

\begin{Theorem} \label{t:s2s2} \cite{CS11}
There is a map $\varkappa:E^7(S^2\times S^2)\to\Z\oplus\Z$ whose image is $2\Z\vee2\Z$.
For every integer $u$ there are exactly $\gcd(u,12)$ isotopy classes of
embeddings $f:S^2\times S^2\to\R^7$ with $\varkappa (f)=(2u,0)$, and
the same holds for those with $\varkappa (f)=(0,2u)$.
\end{Theorem}


 {\it Construction of an embedding $f_u$ with $\varkappa (f)=(2u,0)$.}
Take the standard embeddings $2D^5\times S^2\subset\R^7$ (where 2 is
multiplication by 2) and $\partial D^3\subset\partial D^5$.
Take $u$ copies $(1+\frac1n)\partial D^5\times x$ ($n=1,\dots,u$) of the oriented 4-sphere
outside $D^5\times S^2$ `parallel' to $\partial D^5\times x$.
Join these spheres by tubes so that the homotopy class of the resulting
embedding
$$S^4\to S^7-(D^5\times S^2)\simeq S^7-S^2\simeq S^4\quad\text{will be}\quad u\in\pi_4(S^4)\cong\Z.$$
Let $f$ be the connected sum of this embedding with the standard embedding
$\partial D^3\times S^2\subset\R^7$.

\smallskip
{\bf Addendum.} {\it There are embeddings $f_0,f_1:S^2\times S^2\to\R^7$ such that
 for every embedding $g:S^4\to\R^7$

$\bullet$ embedding $f_0\#g$ is isotopic to $f_0$.

$\bullet$ embedding $f_1\#g$ is isotopic to $f_1$ if and only if $g$ is isotopic to the standard embedding.}

\endcomment

\section{Definitions of the invariants and proofs modulo lemmas} \label{s:defandplan}

\subsection{Main notation}\label{s:defandplan-not}

Recall that some notation was introduced in \S\ref{s:intro}.

Throughout this paper `(sub)manifold' is shorthand for `compact oriented (sub)manifold, possibly with non-empty boundary'.

\smallskip
{\bf Some identifications.}

Identify $\pi_n(S^n)$ and $\Z$ by the degree isomorphism.

Identify $S^2$ and $\C P^1$.
Represent $S^3=\{(z_1,z_2)\in\C^2\colon |z_1|^2+|z_2|^2=1\}.$
The Hopf map $\eta:S^3\to S^2$ is defined by $\eta(z_1,z_2)=[z_1:z_2]$.
Identify $\pi_3(S^2)$ and $\Z$ by the Hopf isomorphism (that sends the homotopy class of $\eta$ to 1).

Identify $\R^4$ and the algebra $\mathbb{H}$ of the quaternions.
Identify $\pi_3(V_{4, 2})$ and $\pi_3(S^3) \oplus \pi_3(S^2) = \Z^2$
by the standard isomorphism, which is defined using the
projection $V_{4, 2} \to S^3$ given by $(x, y) \to x$ and the section $S^3 \to V_{4, 2}$ given by $x\mapsto(x,xi)$.
 Identify $\pi_3(\SO_3)$ and $\pi_3(S^2)=\Z$ by the map induced by the action of $\SO_3$ on $S^2$.

\smallskip
{\bf General notation.}

Denote by

$\bullet$ $N$ a closed connected orientable 4-manifold {\it with torsion free $H_1$};

$\bullet$ $f,f_0,f_1:N\to S^7$ embeddings;

$\bullet$ $\underset n\equiv$ a congruence modulo $n$;

$\bullet$ $\pr_k$ the projection of a Cartesian product onto the $k$-th factor;

$\bullet$ $\id X$ the identity map of the set $X$;

$\bullet$ $1_m:=(1,0,\ldots,0)\in S^m$;

$\bullet$ $\Cl X$ the closure of a subset $X$ in the ambient space, which is clear from the context;

$\bullet$ $N_0:=\Cl(N-B^4)$, where $B^4$ is an embedded closed 4-ball in $N$.


For every $q\le m$ identify the space $\R^q$ with the subspace of $\R^m$ given by the equations $x_{q+1}=x_{q+2}=\dots=x_m=0$.
Analogously identify $D^q,S^{q-1}$ with the corresponding subspaces of $D^m,S^{m-1}$.

Define $\R^m_+,\R^m_-\subset\R^m$ and $D^m_+,D^m_-\subset S^m$ by
the equations $x_1\ge0$ and $x_1\le0$, respectively.
Then
$$S^m=D^m_+\bigcup\limits_{\partial D^m_+=\partial D^m_-}D^m_-\quad\text{and}\quad
\partial D^m_+=\partial D^m_-=D^m_+\cap D^m_-=0\times S^{m-1}\ne S^{m-1}.$$
We denote the union of oriented manifolds in the same way as set-theoretic union.
So both formulas $S^4 = D^4_+ \cup (-D^4_-)$ and $S^4 = D^4_+ \cup D^4_-$ are correct, the sign $``\cup"$ means
union of oriented manifolds in the first formula and union of manifolds in the second one.

\smallskip
{\bf Homological notation.}

Denote by $[x]$ the homology class of $x$ or the equivalence class of $x$ which is an element of a quotient group.

We denote the maps induced in homology by the same letters as the inducing maps.
Thus if $g \colon X \to Y$ is a map of spaces, $g \colon H_*(X) \to H_*(Y)$ denotes the induced map on homology.

Homomorphisms between homology groups with $\Z_d$-coefficients are denoted in the same way as those for $\Z$-coefficients.
So the coefficients are to be understood from the context.
When this could lead to confusion, we specify coefficients by indicating the domain and the range of the homomorphism, e.g. $i:H_3(C_0;\Z_d)\to H_3(M_\varphi;\Z_d)$.

We denote by $i_{A,X},j_{A,X},\partial_{A,X}$ or shortly by $i_A,j_A,\partial_A$ or shortly by $i,j,\partial$,
the homomorphisms from the exact sequence of the pair $(X,A)$.
If $A=C_k$ or $A=C_f$, then we shorten the subscript $C_k$ or $C_f$ to just $k$ or just $f$, respectively.
Denote by $\ex \colon H_q(X,A)\to H_q(X-B,A-B)$ the excision isomorphism, where $B$ is a subset of $A$.

For a $p$-manifold $P$ denote $H_q(P,\partial):=H_q(P,\partial P)$.

Let $P$ and $Q$ be $p$- and $q$-manifolds.
Denote by
$$
\text{PD} \colon H^n(P)\to H_{p-n}(P,\de)\quad
\text{and} \quad  \text{PD} \colon H^n(P,\de)\to H_{p-n}(P),
$$
the Poincar\'e duality isomorphisms.
We sometimes identify homology and cohomology groups by Poincar\'e duality.
We choose to work mostly with homology classes, since this has technical advantages for our arguments, see  \cite[Remark 2.3]{CS11}.

For a map $\xi \colon P\to Q$ denote the `preimage' homomorphism by
$$\xi^!:=\text{PD}\circ\xi\circ \text{PD}^{-1}:H_n(Q,\partial)\to H_{p-q+n}(P,\partial),$$
where $\xi$ is the homomorphism induced in cohomology.

We now consider simplicial cycles or cycles represented by maps of manifolds, the intersection of transverse cycles and linking number of disjoint cycles \cite{MAI}.


For set-theoretic intersection we write $X \cap Y$.
(This notation is also used for restriction, see \S\ref{s:prelemmas-not}.)
For the algebraic intersection of chains or integer cycles or oriented manifolds in an ambient manifold $M$ we write $X\capM{M}Y$.
Recall that $X \capM{M} Y=(-1)^{\codim X\codim Y}Y\capM{M} X$, and that if $X,Y$ are cycles, then
$X\capM{M}Y$ depends only on the homology classes represented by $X$ and $Y$.

Let $A$ and $B$ be integer $a$- and $b$-cycles in $\R^m$ having disjoint supports with $a+b+1=m$.
Define {\it the linking number} of $A$ and $B$ by $\lk(A,B):= A \capM{\R^m}\beta$, where $\beta$ is a $(b{+}1)$-cycle in $\R^m$ with $\de\beta=B$.
It is easy to check that $\lk(A,B)=\alpha \capM{\R^m} B$, where $\alpha$ is an $(a{+}1)$-cycle in $\R^m$ with $\de\alpha=A$.
Recall that $\lk(A,B)=(-1)^{(m-a)(m-b)}\lk(B,A)$.

\subsection{Definitions of the $\varkappa$-  and $\lambda$-invariants}\label{s:defandplan-kl}

Let $\zeta \colon N_0 \to \nus^{-1}N_0$ be a `partial' section of $\nus:\de C\to N$.
Consider the following diagram:
$$\xymatrix{
\Z\cong H_4(N_0,\partial) \ar[r]^{\zeta} & H_4(\nus^{-1}N_0,\partial)&
\ar[l]_{\ex} H_4(\partial C,\nus^{-1}B^4) & \ar[l]_(0.375){j_{\partial C}}  H_4(\partial C)  \ar[r]^{i_C} & H_4(C)}. $$
Here $j_{\partial C}$ and $\ex$ are isomorphisms.
The composition $N_0\overset{\zeta}\to\nus^{-1}N_0\overset{\subset}\to\de C$ of $\zeta$
and the inclusion is called a {\bf weakly unlinked section} if
$i_Cj_{\partial C}^{-1}\ex^{-1}\zeta=0 \in H_4(C)$.

We remark that

$\bullet$ a section $\zeta:N_0\to\nus^{-1}N_0$ exists because the Euler class of $\nu$ is zero, since vector bundle associated to $\nu$ is 3-dimensional, and $N_0$ retracts to a 3-polyhedron;

$\bullet$ any section $\zeta:N_0\to\nus^{-1}N_0$ is weakly unlinked for $N=S^1\times S^3$ because
there is an isomorphism $H_4(S^7-f(S^1\times S^3))\cong H_2(S^1\times S^3)=0$.
Cf. Lemma \ref{l:weun}.a.

\begin{Lemma} \label{l:weunl} A weakly unlinked section exists and is unique up to vertical homotopy over
the 2-skeleton of any triangulation of $N$.
\end{Lemma}

\begin{proof} This holds by
\cite[Proposition 1.3]{BH70} because by \cite[Remark 2.4 and footnote 14]{Sk10} our definition of a weakly unlinked section is equivalent to the original definition \cite{BH70}.
Cf. proof of Lemma \ref{l:weun}.b and \cite[the Unlinked Section Lemma (a)]{Sk08'}.
\end{proof}

{\bf Definition of the Bo\'echat-Haefliger invariant $\varkappa:E^7(N)\to H_2$.}
Take a weakly unlinked section $\xi:N_0\to\partial C$.
Recall that $H_2$ is torsion-free.
By Poincar\'e duality $\capM{N}\colon H_2 \times H_2 \to \Z$ is unimodular.
Hence $\varkappa(f)\in H_2$ can be defined by the equation
$$
\varkappa(f)\capM{N}[X]= \lk\phantom{}_{S^7}(fN,\xi X),
$$
for any 2-cycle $X\subset N_0$.

This is well-defined, i.e.~is independent of the choice of $\xi$, by Lemma \ref{l:La}.$\varkappa'$.
This definition is equivalent to those of \cite{BH70, Sk10, CS11} by Lemma \ref{l:La}.$\varkappa'$,e.%
\footnote{Those definitions do not require the assumption that $H_1$ is torsion free.
In those papers the invariant was denoted by $w_f$ \cite{BH70}, or $BH(f)$ \cite{Sk10}, or $\aleph(f)$ \cite{CS11}, instead of $\varkappa(f)$.}
Clearly, the map $\varkappa:E^7(N)\to H_2$ is well-defined by $\varkappa([f]):=\varkappa(f)$.

\smallskip
{\bf Definition of the Seifert form $\lambda \colon E^7(N)\to B(H_3)$.}
Represent classes $x,y\in H_3$ by closed oriented 3-submanifolds (or integer 3-cycles) $X,Y\subset N_0$.
Take a weakly unlinked section $\xi:N_0\to\partial C$.
Define
$$\lambda(f)(x,y):=\lk\phantom{}_{S^7}(fX,\xi Y)\in\Z.$$
This is well-defined; i.e.~is independent of the choice of $\xi$, by Lemma \ref{l:La}.$\lambda'$.
Clearly, the pairing $\lambda(f):H_3\times H_3\to\Z$ is indeed a bilinear form.
Clearly, the map $\lambda:E^7(N)\to B(H_3)$ is well-defined by $\lambda([f]):=\lambda(f)$.
Cf. \cite{Sa99, To10}.

\begin{proof}[Proof of Lemma \ref{l:calam}] Part (c) follows because $\xi Y$ and $fY$ are homologous in $S^7-fX$.
Part (a) follows by (c).
Part (b) follows by (a).
\end{proof}

We give equivalent definitions of $\varkappa$ and $\lambda$ in Lemma \ref{l:La}.

\begin{Lemma}[$\varkappa$-symmetry; proved in \S\ref{s:prelemmas-equ}]\label{l:Lam}
We have
$\lambda(f)(y,x)=\lambda(f)(x,y)-\varkappa (f)\capM{N} x\capM{N} y$.
\end{Lemma}

Cf. \cite[Lemma 2.2]{Sa99}, \cite[Theorem 1.5(2) and Lemmas 1.6 and 2.10]{To10}.

\begin{Lemma}[Additivity of $\lambda$ and $\varkappa$]\label{l:Addl}
For every pair of  embeddings $g:S^4\to S^7$ and $f:N\to S^7$
$$\varkappa(f\#g)=\varkappa(f)\quad\text{and}\quad\quad\lambda(f\#g)=\lambda(f).$$
\end{Lemma}

\begin{proof}
We may assume that $g(S^4)\cap C_f=\emptyset$ and $\nu_f=\nu_{f\#g}$ over $N_0$.
Then additivity for $\lambda$ and $\varkappa$ follows because a weakly unlinked section for $f$
is also a weakly unlinked section for $f\#g$.
\end{proof}


\subsection{Definition of the $\beta$-invariant and the map $\beta_{u,l}$}\label{s:defandplan-b}

Take a small oriented disk $D^3_f\subset\R^7$ whose intersection with $f(N)$ consists of exactly one point
of sign $+1$ and such that $\partial D^3_f\subset\partial C_f$.
Define {\bf the meridian of $f$} by
$$
S^2_f := [\partial D^3_f]\in H_2(C_f).
$$
Then $S^2_f$ is a generator of  $H_2(C_f)$ (by homology Alexander duality, see \S\ref{s:prelemmas-not}).

Recall that $f_0,f_1:N\to S^7$ are embeddings.
For a bundle isomorphism $\varphi:\de C_0\to\de C_1$ define the closed oriented $7$-manifold
$$
M=M_\varphi:=C_0\cup_\varphi(-C_1).
$$
A {\bf joint Seifert class for a bundle isomorphism $\varphi:\de C_0\to\de C_1$} is a class
$$
Y\in H_5(M_\varphi)\quad\text{such that}\quad Y\capM{M_\varphi} [\partial D^3_{f_0}] = 1.
$$
We shall omit the phrase `for a bundle isomorphism $\varphi$' if its choice is clear from the context.


\begin{Lemma}[proved in \S\ref{s:prelemmas-hss}] \label{l:jhss}
If $\varkappa(f_0)=\varkappa(f_1)$ and $\varphi \colon \de C_0\to\de C_1$ is a bundle isomorphism, then
there is a joint Seifert class $Y\in H_5(M_\varphi)$.
\end{Lemma}

We call a bundle isomorphism $\varphi:\de C_0\to\de C_1$ a {\bf $\pi$-isomorphism} if $M_\varphi$ is parallelizable.

\begin{Lemma}[proved in \S\ref{s:lemmas-sf}]\label{l:piso}
If $\varkappa(f_0)=\varkappa(f_1)$ and $\lambda(f_0)=\lambda(f_1)$, then there is a $\pi$-isomorphism
$\varphi:\de C_0\to\de C_1$.
A $\pi$-isomorphism is unique (up to vertical homotopy through linear isomorphisms) over $N_0$.
\end{Lemma}

{\bf Definitions of
$\beta(f_0,f_1)$.}
For $u\in H_2$ and $l\in B(H_3)$ recall that $\overline l \colon H_3 \to H_1$ is the adjoint of $l$ and that
$\rho_{d} \colon H_1 \to H_1(N; \Z_d)$ is reduction modulo $d$.
Assume that $\varkappa(f_0)=\varkappa(f_1)$ and that $\lambda(f_0)=\lambda(f_1)$.
Denote $d:=\di(\varkappa(f_0))$.
By Lemmas \ref{l:jhss} and \ref{l:piso} there is a joint Seifert class $Y\in H_5(M_\varphi)$ and a $\pi$-isomorphism $\varphi:\de C_0\to\de C_1$.
Define
$$
\beta(f_0,f_1):=[(i_{\de C_0,M_\varphi}\nus_0^!)^{-1}\rho_dY^2]\in C_{\varkappa(f_0),\lambda(f_0)}
$$
using the composition
$H_1(N;\Z_d)\xrightarrow{~\nus_0^!~} H_3(\de C_0;\Z_d)\xrightarrow{~i_{\de C_0,M_\varphi}~} H_3(M_\varphi;\Z_d).$

\begin{Lemma}[proved in \S\ref{s:lemmas-jhss}]\label{l:webeta}
The class $\beta(f_0,f_1)$ is well-defined; i.e.

$\bullet$ for every joint Seifert class $Y$ and $\pi$-isomorphism $\varphi$ there is a unique element
$b_{\varphi,Y}\in H_1(N;\Z_d)$ such that $i_{\de C_0,M_\varphi}\nus_0^!b_{\varphi,Y}=\rho_dY^2\in H_3(M_\varphi;\Z_d)$,

$\bullet$ $[b_{\varphi,Y}]\in C_{\varkappa(f_0),\lambda(f_0)}$ is independent of the choice of joint Seifert class $Y$ and $\pi$-isomorphism $\varphi$.
\end{Lemma}

\begin{Lemma}[Calculation of $\beta$; proved in \S\ref{s:lemmas-calc}] \label{l:calc}

(a) $\beta(\tau(0,0),\tau(0,b))=b[S^1\times1_3]\in H_1(S^1\times S^3)$.

(b) $\beta(\tau(l,b'),\tau(l,b))=\rho_{2l}(b-b')[S^1\times1_3]\in H_1(S^1\times S^3;\Z_{2l})$
(cf. Remark before Lemma \ref{l:Bet}).
\end{Lemma}

\begin{Theorem}[Isotopy Classification Modulo Knots; proved in \S\ref{s:lemmas-isotmod}]\label{l:isomk}
If we have $\lambda(f_0)=\lambda(f_1)$, $\varkappa(f_0)=\varkappa (f_1)$ and $\beta(f_0,f_1)=0$,
then there is an embedding $g:S^4\to S^7$ such that $f_0$ is isotopic to $f_1\#g$.
\end{Theorem}

\begin{Lemma}[Additivity of $\beta$; proved in \S\ref{s:lemmas-jhss}]\label{l:Add}
For every pair of  embeddings $g:S^4\to S^7$ and $f:N\to S^7$ we have $\beta(f\#g,f)=0$.
\end{Lemma}

\begin{Lemma}[Transitivity of $\beta$; proved in \S\ref{s:lemmas-jhss}]\label{l:Dif}
For every triple of embeddings $f_0,f_1,f_2:N\to S^7$ having the same values of $\varkappa$- and $\lambda$-invariants
we have $\beta(f_2,f_0)=\beta(f_2,f_1)+\beta(f_1,f_0)$.
\end{Lemma}

{\bf Definitions of the maps}
$$
\beta:[(\varkappa\times\lambda)^{-1}(u,l)]^2\to C_{u,l},\quad
\beta_{\mk}:[(\varkappa_{\mk}\times\lambda_{\mk})^{-1}(u,l)]^2\to  C_{u,l},\quad
\beta_{u,l}:(\varkappa\times\lambda)^{-1}(u,l)\to  C_{u,l}.
$$
Clearly, the map $\beta$ is well-defined by $\beta([f],[g]):=\beta(f,g)$.

The map $\beta_{\mk}$ is well-defined by $\beta=\beta_{\mk}(q_{\mk}\times q_{\mk})$,
according to the additivity and the transitivity of $\beta$ (Lemmas \ref{l:Add} and \ref{l:Dif}).

Take an embedding $f':N\to S^7$ representing an isotopy class in
$(\varkappa\times\lambda)^{-1}(u,l)$.
Let $\beta_{u,l}[f]:=\beta(f,f')$.
The map $\beta_{u,l}$ depends on $f'$ but we do not indicate this in the notation.

\subsection{Parametric connected sum and parametric additivity}\label{s:defandplan-pr}

In our proof of the realizability of the invariants we extensively use the so called {\it parametric connected sum operation} defined below.
We first recall, with minor modifications, some definitions and results of \cite[\S2]{Sk07}, \cite[\S2.1]{Sk15}, \cite{MAP}.

\smallskip
{\bf Definition of a standardized map.}
The base point $\ast$ of $V_{4,2}$ is the standard inclusion $\R^2\to\R^4$.
Take the embedding $\tau_0:S^1\times S^3\to S^7$ for the constant map
$\alpha_0:S^3\to V_{4,2}$ (as defined in \S\ref{s:intro-main}).
Clearly, $\tau_0(S^1\times D^3_{\pm})\subset D^7_{\pm}$.
For an embedding $s:S^1\times D^3_-\to N$, a map $h:N\to S^7$ is called {\it $s$-standardized} if
$$h(N-\im s)\subset\Int D^7_+\quad\text{and}\quad h\circ s=\tau_0|_{S^1\times D^3_-}.$$

\begin{Lemma}\label{l:stand}
For every embedding $s:S^1\times D^3_-\to N$, any $f$ is isotopic to an $s$-standardized embedding $\t f:N\to S^7$.
\end{Lemma}

This is a smooth version of \cite[Standardization Lemma]{Sk07} which is proved analogously in
\jonly{\cite[Remark 2.24.a]{I},} \aronly{Remark \ref{t:par}.a,}
cf. \cite[Standardization Lemma 2.1.a]{Sk15}.


\smallskip
{\bf Definition of parametric connected sum $[f]+_s\tau(l,b)$.}
Take a map $\alpha:S^3\to V_{4,2}$ representing an element $(l,b)\in\Z^2=\pi_3(V_{4,2})$ such that $\alpha(D^3_-)=\ast$.
Then the embedding $\tau_\alpha$ is $i$-standardized for the inclusion $i:S^1\times D^3_-\to S^1\times S^3$.
Take an embedding $s:S^1\times D^3_-\to N$ and an $s$-standardized embedding $\t f:N\to S^7$ isotopic to $f$.
Let $R:\R^m\to\R^m$ be the symmetry of $\R^m$ with respect to the hyperplane given by equations
$x_1=x_2=0$, i.e., $R$ is defined by $R(x_1,x_2,x_3,\dots,x_m):=(-x_1,-x_2,x_3,\dots,x_m)$.
Define the embedding
$$h:N\to S^7 \quad\text{by}\quad
h(a):=\begin{cases} \t f(a) &a\not\in\im s,\\
R\tau_\alpha(x,Ry)          &a=s(x,y). \end{cases}$$
The two formulas agree on $\de\im s$ because $\tau_0(x,y)=R\tau_0(x,Ry)$.
Clearly, $h$ is a smooth embedding, i.e.~it is injective, differentiable and has non-degenerate dervative.

Let $[f]+_s\tau(l,b)\subset E^7(N)$ be the set of isotopy classes of the embeddings $h$,
for all choices of $\t f$ and $\alpha$ as above.
(In fact, $[h]$ clearly does not depend on the choice of $\alpha$, for fixed $l,b,\t f,s$.
Still, $[h]$ may depend on $\t f$, i.e. $+_s$ is not defined at the level of isotopy for embeddings of 4-manifolds into $S^7$, as opposed to
other situations \cite{Sk07, Sk15, MAP}.
Cf.~Corollary \ref{t:rem}.c,d,e.)

\begin{Lemma}[Parametric additivity; proved in \S\ref{s:lemmas-pad}] \label{l:Pad}
For any embedding $s:S^1\times D^3_-\to N$ let $[s]:=[s|_{S^1\times0_3}]\in H_1$.
Then for any embedding $h\in f+_s\tau(l,b)$ and $x,y\in H_3$ we have
$$\varkappa(h)=\varkappa (f),\quad \lambda(h)(x,y) = \lambda(f)(x,y)+l([s]\cap_N x)([s]\cap_N y)\quad\text{and,}$$
$$\text{for} \quad l=0,\quad\beta(f,h)=b[s]\in C_{\varkappa(f),\lambda(f)}.$$
\end{Lemma}


\begin{Corollary} \label{t:rem}
(a) For every $u\in H_2^{DIFF}$ every isotopy class in $\varkappa^{-1}(u)$ can be obtained from a single
isotopy class in $\varkappa^{-1}(u)$ by a finite sequence of parametric connected sum operations.

(b) Denote by $B_0(H_3)$ the group of symmetric bilinear forms $H_3\times H_3\to\Z$.
Then there is a surjection
$$\tau_{\mk}:H_1\times H_2^{DIFF}\times B_0(H_3)\to E_{\mk}^7(N)\quad\text{such that}$$
$$\tau_{\mk}(b,u,l)=\tau_{\mk}(b',u',l')\quad\Leftrightarrow\quad
u=u',\quad l=l'\quad\text{and}\quad b-b'\in C_{u,l+\lambda_{\mk}\tau_{\mk}(0,u,0)}.$$

(c) The set $q_{\mk}([f]+_s\tau(l,b))$ is independent of $s$ for fixed $l,b,\t f$ and $[s|_{S^1\times 0_3}]\in H_1$
(cf.~\cite{Sk14}).

(d) Both sets $[f]+_s\tau(l,b)$ and $q_{\mk}([f]+_s\tau(l,b))$ may consist of more than one element,
i.e.~$q_{\mk}[h]$ from the definition of $[f]+_s\tau(l,b)$ can depend on the choice of $\t f$,
even for $N=S^1\times S^3$.

(e) The set $q_{\mk}(\tau(0,b')+_i\tau(0,b))$ consists of one element, i.e.~$q_{\mk}[h]$ from the definition of
$\tau(0,b')+_i\tau(0,b)$ does not depend on the choice of $\t f$.
\end{Corollary}

\begin{proof}
Parts (a,c) follow from Theorem \ref{t:gen} and the parametric additivity (Lemma \ref{l:Pad}).
Part (b) follows from (a,c) and \jonly{\cite[Remark 2.21.c]{I}} \aronly{Remark \ref{r:bil}.c}.
Part (d) follows from \jonly{\cite[Remark 2.24.b]{I}} \aronly{Remark \ref{t:par}.b}
and the parametric additivity (Lemma \ref{l:Pad}).
Part (e) follows from Theorem \ref{t:s1s3} and the parametric additivity (Lemma \ref{l:Pad}).
\end{proof}

\subsection{Proof of  Theorems \ref{t:s1s3} and \ref{t:gen} assuming lemmas} \label{s:defandplan-s1s3}

\begin{proof}[Proof of  Theorem \ref{t:s1s3}: the surjectivity of $\tau_{\mk}$]
Take any embedding $f:S^1\times S^3\to S^7$.
Identify $B(H_3(S^1\times S^3))$ with $\Z$.
Denote $l:=\lambda(f)\in\Z$.
Take a representative $\alpha:S^3\to V_{4,2}$ of $(l,0)\in\pi_3(V_{4,2})$.
Then $[\tau_\alpha]=\tau(l,0)$.
By the calculation of $\lambda$ (Lemma \ref{l:calam}.b) we have $\lambda(\tau_\alpha)=l$.

We have $\im\overline{\lambda(f)}=l\Z[S^1\times1_3]$ and $\varkappa(f)=\varkappa(\tau_\alpha)=0$.
Hence $C_{\varkappa(f),\lambda(f)}\cong H_1(S^1\times S^3;\Z_{2l})\cong\Z_{2l}$
 and the class $\beta(f,\tau_\alpha)\in C_{\varkappa(f),\lambda(f)}$ is defined.
Take an integer $b$ such that $\beta(f,\tau_\alpha)=-\rho_{2l}b[S^1\times1_3]$.
By the transitivity of $\beta$ (Lemma \ref{l:Dif}) and the calculation of $\beta$ (Lemma \ref{l:calc}.b)
$$\beta([f],\tau(l,b))=
\beta(f,\tau_\alpha)+\beta(\tau(l,0),\tau(l,b)) =\rho_{2l}(-b+b)[S^1\times1_3]=0.$$
Hence by the MK Isotopy Classification Theorem \ref{l:isomk} $q_{\mk}[f]=\tau_{\mk}(l,b)$.
\end{proof}

\begin{proof}[Proof of  Theorem \ref{t:s1s3}: description of preimages of $\tau_{\mk}$]
Denote $\tau:=\tau_{\mk}(l,b)$ and $\tau'=\tau_{\mk}(l',b')$.
By the calculation of $\lambda$ (Lemma \ref{l:calam}.b) we have $\lambda(\tau)=l$ and $\lambda(\tau')=l'$.
So for $l=l'$ by the calculation of $\beta$ (Lemma \ref{l:calc}.b) we have $\beta(\tau',\tau)=\rho_{2l}(b-b')[S^1\times1_3]$.
Hence $$\tau=\tau'\quad\Leftrightarrow\quad l=l'\text{ and }\beta(\tau',\tau)=0\quad\Leftrightarrow\quad
l=l'\text{ and }b\equiv b'\mod2l$$
by the MK Isotopy Classification Theorem \ref{l:isomk}
\end{proof}

\begin{proof}[Proof of Theorem \ref{t:gen}]
The map $\beta_{u,l,\mk}$ is surjective by the parametric additivity (Lemma \ref{l:Pad})
and is injective by the MK Isotopy Classification Theorem \ref{l:isomk}.

In this proof denote $\cdot=\capM{N}$.
By Lemma \ref{l:La}.$\varkappa'$,e our definition of $\varkappa$ is equivalent to that of \cite{BH70}, cf. \cite{Sk10, CS11}.
Hence by \cite{BH70} $\im\varkappa_{\mk}=\im\varkappa=H_2^{DIFF}$.
So it remains to prove that for every $u\in \im\varkappa$
$$\lambda(\varkappa^{-1}(u))=
\{l\in B(H_3)\ :\ l(y,x)=l(x,y)+u\cdot x\cdot y\text{ for every }x,y\in H_3\}.$$
By the $\varkappa$-symmetry (Lemma \ref{l:Lam}) and \jonly{\cite[Remark 2.21.c]{I}} \aronly{Remark \ref{r:bil}.c}
this is implied by the following claim.

{\bf Claim.} {\it For every embedding $f:N\to S^7$ and every symmetric bilinear form $m:H_3\times H_3\to\Z$
there is an embedding $g=g(f,m):N\to S^7$ such that $\varkappa(g)=\varkappa(f)$ and $\lambda(g)=\lambda(f)+m$.}

{\it Proof of claim.}
We can set $g(f,m_1+m_2):=g(g(f,m_1),m_2)$.
Thus it suffices to construct $g(f,m)$ only for basic forms
$$m_p(x,y)=(p\cdot x)(p\cdot y)\quad\text{and}
\quad m_{p,q}(x,y)=(p\cdot x)(q\cdot y)+(p\cdot y)(q\cdot x),\quad\text{where}\quad p,q\in H_1.$$
Take embeddings $s,u,v:S^1\times D^3_-\to N$ whose restrictions to $S^1\times 0$ represent elements $p,q,p+q\in H_1$, respectively.
By the parametric additivity (Lemma \ref{l:Pad}) we can take as $g(f,m_p)$ and $g(f,m_{p,q})$ any elements of
$$[f]+_s\tau(1,0)\quad\mbox{and}
\quad \big(([f]+_v\tau(1,0))+_s\tau(-1,0)\big)+_u\tau(-1,0),\quad\text{respectively},$$
where the latter set is the the set $h_1+_u\tau(-1,0)$ for some $h_1\in h_2+_s\tau(-1,0)$ and for some
$h_2\in [f]+_v\tau(1,0)$.
\end{proof}

\aronly{

\subsection{Appendix: on regular homotopy and the Compression problem}\label{s:defandplan-reho}

\begin{Proposition}[Regular homotopy classification] \label{l:reho}
If $f_0,f_1:N\to S^7$ are embeddings and $\left(\lambda(f_0)-\lambda(f_1)\right)(x,x)\underset2\equiv0$ for
all $x\in H_3$, then $f_0$ and $f_1$ are
regular homotopic.
\end{Proposition}

Define the map $W:E^7(N)\to H_1(N;\Z_2)$ by $\rho_2\lambda(f)(x,x)=W(f)\cap_Nx$ for all $x\in H_3(N;\Z_2)$.
By Proposition \ref{l:reho} $W$ induces an injection on the set of regular homotopy classes of embeddings.
By Theorem \ref{t:gen} $\im W$ consists of those $y\in H_1(N;\Z_2)$ for which there is $u\in H_2^{DIFF}$ and a $u$-symmetric bilinear form $l\in B(H_3)$ such that $\rho_2l(x,x)=y\cap_Nx$ for every $x\in H_3(N;\Z_2)$.
It would be interesting to obtain a more direct description of $\im W$.

\smallskip
{\bf Definition of the Whitney invariant $W_0'$.} (See \cite[\S1]{Sk10'}.)
The {\it singular set} of a smooth map $H:X\to Y$ between manifolds is
$S(H):=\{x\in X\ :\ d_xH\text{ is degenerate}\}$.

Let $f_0,f_1:P\to S^7$ be immersions of a 4-manifold $P$.
Take a general position homotopy $H:P\times I\to S^7\times I$ between $f_0$ and $f_1$.
By general position, $\Cl S(H)$ is a closed 1-submanifold.
Define $W'_0(f_0,f_1):=[\Cl S(H)]\in H_1(P,\partial;\Z_2)$.

(It is well-known that $W'_0(f_0,f_1)$ is indeed independent of $H$ for fixed $f_0$ and $f_1$.)


\begin{Lemma} \label{l:lawh}
Let $f_0,f_1:P\to\R^7$ be embeddings of a 4-manifold $P$ and $X\subset P$
a closed connected 3-submanifold.
Take the normal vector field of $X$ in $P$ defined by the orientations of $X$ and $P$.
Let $X'$ be the shift of $X$ along this vector field.
Then
$$W_0'(f_0,f_1)\cap_P \rho_2[X]=\rho_2[\lk\phantom{}_{\R^7}(f_0X,f_0X')-\lk\phantom{}_{\R^7}(f_1X,f_1X')].$$
\end{Lemma}

\begin{proof}
It suffices to prove this equality for $P=X\times I$, $X=X\times0$ and $X'=X\times 1$.
By the strong Whitney Isotopy Theorem \cite[Theorem 2.2.b]{Sk08} $f_0|_X$ and $f_1|_X$ are isotopic.
Since both sides of the required equality do not change under isotopy of $\id\R^7$,
we may assume that $f_0=f_1$ on $X$.
Take a general position homotopy $H:X\times I\times I\to \R^7\times I$ between $f_0$ and $f_1$
that is fixed on $X$.
The homotopy $H$ gives a homotopy $G:X\times I\to\R^7$ of a normal vector field on $f_0(X)\subset\R^7$
(through normal vector fields which are not assumed to be non-zero).
Since $H|_{X\times t}$ is an embedding, for every $x\in X$ and $t\in I$ the differential $d_{(x,t)}H$
is degenerate if and only if $G(x,t)=0$.
By general position, $G^{-1}(0)=H(X'\times I)\cap f_0(X)$ is a finite number of points.
Then
$$W_0'(f_0,f_1)\cap_P \rho_2[X]=
\rho_2 |H(X'\times I)\cap f_0(X)|=\rho_2[\lk\phantom{}_{\R^7}(f_0X,f_0X')-\lk\phantom{}_{\R^7}(f_1X,f_1X')].$$
\end{proof}

\begin{proof}[Proof of Proposition \ref{l:reho}] The following statements are equivalent:

(i) $f_0$ and $f_1$ are regular homotopic;

(ii) $f_0|_{N_0}$ and $f_1|_{N_0}$ are regular homotopic;

(iii) $W_0'(f_1|_{N_0},f_0|_{N_0})=0$;

(iv) $\lambda(f_0)(x,x)\equiv\lambda(f_1)(x,x)\mod2$ for every $x\in H_3$.

Indeed,

$\bullet$ (i)$\Leftrightarrow$(ii) because by the Smale-Hirsch classification of immersions \cite{Hi60} the complete obstruction to extension of a regular homotopy from $N_0$ to $N$ assumes values in $\pi_4(V_{7,4})=0$ \cite{Pa56}.

$\bullet$ (ii)$\Leftrightarrow$(iii) because by the Smale-Hirsch classification of immersions \cite{Hi60}
the first obstruction to regular homotopy between $f_0|_{N_0}$ and $f_1|_{N_0}$ assumes values in $H_1(N_0,\partial;\pi_3(V_{7,4}))$ and is complete, and because this obstruction clearly coincides with $W_0'(f_1|_{N_0},f_0|_{N_0})$.

$\bullet$ (iii)$\Leftrightarrow$(iv) by Lemma \ref{l:lawh} because by the calculation of $\lambda$
(Lemma \ref{l:calam}.c) $\lambda(f_k)([X],[X])=\lk(f_k|_X,f_k|_{X'})$, so $W_0'(f_0,f_1)\cap_Nx=\rho_2\left(\lambda(f_0)-\lambda(f_1)\right)(x,x)$ for all $x\in H_3(N;\Z_2)$.
\end{proof}

\begin{Problem}[Compression problem] \label{t:com}
For an integer $j\in\{1,2\}$ describe those embeddings $N\to S^7$ which are isotopic to embeddings with image in
$S^{7-j} \subset S^7$.
\end{Problem}

Clearly, $\lambda(f)=\varkappa(f)=0$ for an embedding $f:N\to\R^7$ such that $f(N)\subset\R^6$.

\begin{Proposition}\label{e:comb}
There are embeddings $f_0,f_1:N\to S^7$ such that $f_0(N)\cup f_1(N)\subset S^6$ and $\beta(f_0,f_1)\ne0$.
\end{Proposition}

\begin{proof}
This follows by Lemma \ref{l:cota} below because $\beta(\tau(0,0),\tau(0,1))\ne0$
by the calculation of $\beta$ (Lemma \ref{l:calc}.a).
\end{proof}

\begin{Lemma}\label{l:cota}
There is a representative of $\tau(0,1)$ whose image is in $S^6\subset S^7$.
\end{Lemma}

\begin{proof}[Proof of Lemma \ref{l:cota}]
We use the construction of $\tau(0,1)$ from Remark \ref{r:s1s3}.
Denote by $n:S^2\to TS^3$ a non-zero vector field normal to $S^2\subset S^3$ and
looking to the northern hemisphere of $S^3$.
Then

$\bullet$ for every $x\in S^3$ the image $T^2(S^1\times x)$ is the round circle in $S^3$ passing through $x$ in the direction $n(\eta(x))$, and

$\bullet$ $T^2$ is uniform on this circle.

Consider the normal bundle of $\id S^3\times \eta:S^3\to S^3\times S^2$.
The obstructions to the existence of a non-zero section of this bundle are in $H^{i+1}(S^3,\pi_i(S^1))=0$.
Hence there is such a section $v(x)\in T_{\eta(x)}S^2$, $x\in S^3$.
Define a map $T^3:S^1\times S^3\to S^3$ by setting

$\bullet$ for every $x\in S^3$ the image $T^3(S^1\times x)$ to be the round circle in $S^3$ passing through $x$ in the direction $v(x)$, and

$\bullet$ $T^3$ to be `linear' uniform on this circle.

We have $T^3(S^3)\subset S^2$, hence
$\inc\circ(\pr_2\times T^3)(S^1\times S^3)\subset \inc(S^3\times S^2)\subset S^6\subset S^7$.

Take a linear homotopy $v_t(x):=\dfrac{tn(\eta(x))+(1-t)v(x)}{|tn(\eta(x))+(1-t)v(x)|}\in T_{\eta(x)}S^3$
between non-zero vector fields $n(\eta(x))$ and $v(x)$ on $S^2\subset S^3$.
This homotopy defines a homotopy between $T^2$ and $T^3$ which keeps the image of $S^1\times x$ embedded.
The latter homotopy defines an isotopy from a representative $\inc\circ(\pr_2\times T^2)$ of $\tau(0,1)$ to the embedding
$\inc\circ(\pr_2\times T^3)$ whose image is in $S^6\subset S^7$.
\end{proof}


\subsection{Appendix: some remarks to \S\ref{s:intro} and \S\ref{s:defandplan}}\label{s:defandplan-rem}

\begin{Remark}[to \S\ref{s:intro-kno}]\label{r:tos1}
In this paper `smooth' means `$C^1$-smooth'.
Recall that a smooth embedding is `orthogonal to the boundary'.
For each $C^\infty$-manifold $N$ the forgetful map from the set of $C^\infty$-isotopy classes of $C^\infty$-embeddings  $N\to\R^m$ to $E^m(N)$ is a 1--1 correspondence.
For a (possibly folklore) proof of this result see \cite{Zh16}.
\end{Remark}

\begin{Remark}[to \S\ref{s:intro-main}]\label{r:bil} (a) Theorem \ref{t:s1s3} is not an immediate corollary of Theorem \ref{t:gen}, see Remark \ref{r:t1t2}.

(b) For every $u\in H_2^{DIFF}$ the set $\lambda(\varkappa^{-1}(u))$ consists of those $l\in B(H_3)$ for which $(u,l)$ is symmetric.

(c) For fixed $u$, the set of symmetric pairs $(u, l)$ is in bijection
with the group $B_0(H_3)$.
Indeed $B_0(H_3)$ acts freely and transivitely on this set, for if
$(u,l)$ and $(u,l')$ are symmetric pairs, then $l-l'$ is a symmetric form.

(d) Theorem \ref{t:gen} has a restatement similar to Theorem \ref{t:s1s3}, see Corollary \ref{t:rem}.b.

(e) {\it Deduction of the italicized statement in p. \pageref{pst:bh}.}
Suppose that $H_1 = 0$.
The forgetful map from $E^7_{\mk}(N)$ to the set of PL isotopy classes is injective for $H_1 = 0$ \cite[p. 141]{Bo71},  \cite{Ha68}.
Bo\'echat and Haefliger classified PL embeddings $f \colon N \to S^7$ up to PL isotopy \cite[Theorem 1.6]{BH70}.
They also characterized smoothable PL embeddings  \cite[Theorem 2.1]{BH70}.
All this implies the above result.
An alternative proof of the injectivity of $\varkappa_\mk$ is given in \cite{CS11}.
\end{Remark}

\begin{Remark}[to \S\ref{s:intro-strat}]\label{r:tobh}
Note that $\varkappa$-invariant and $\lambda$-invariants can alternatively be defined using the intersection in the homology of the complements $C_0,C_1$.
However, that definition corresponds to the classical surgery not modified surgery approach.
\end{Remark}

\begin{Remark}[to \S\ref{s:defandplan-b}]\label{r:tob}
The property of $Y$ identified in Lemma \ref{l:desei}.a below provides an equivalent definition of
a joint Seifert class which explains the name and which was used in \cite{Sk10, CS11}, together with the name
`joint homology Seifert surface'.
\end{Remark}

\begin{Remark}[to \S\ref{s:defandplan-pr}]\label{t:par}
(a) {\it Proof of Lemma \ref{l:stand}.}
Define
$$\inc:\sqrt2 D^2\times D^4\to S^7\quad\text{by}\quad \inc(x,y):=(y\sqrt{2-|x|^2},0,0,x)/\sqrt2.$$
(No confusion with the map $\inc$ defined in \S\ref{s:intro-main} will appear.)
Then $\inc=\tau_0$ on $S^1\times D^3$.
For $\gamma\le\sqrt2$ denote $\Delta_\gamma:=\inc(\gamma D^2\times\{-1_3\})\subset \Int D^7_-$.

In this proof we omit the sign $\circ$ for composition.

Any two embeddings $S^1\times D^3\to S^7$ are isotopic.
So we can make an isotopy and assume that $fs=\inc$ on $S^1\times D^3_-$.

Since $7>2\cdot1+3+1$, by general position we may assume that $f(N)\cap \Delta_1=\partial \Delta_1$.
Then there is $\gamma$ slightly greater than 1 such that $f(N)\cap \Delta_\gamma=\partial \Delta_1$.
Take the standard $3$-framing on $\Delta_\gamma$ tangent to $\inc(\gamma D^2\times S^3)$
whose restriction to $\partial\Delta_1$ is the standard normal $3$-framing of $\partial\Delta_1$ in $f(N)$.
Then the standard $2$-framing normal to $\inc(\gamma D^2\times S^3)$ is a $2$-framing on $\partial\Delta_1$ normal to $f(N)$.
Using these framings we construct

$\bullet$ an orientation-preserving embedding $H:D^7_-\to D^7_-$ onto a sufficiently small neighborhood of $\Delta_1$ in $D^7_-$,
and

$\bullet$ an isotopy $h_t$ of $\id(S^1\times S^3)$ shrinking $S^1\times D^3_-$ to a sufficiently small neighborhood of
$S^1\times\{-1_3\}$ in $S^1\times D^3_-$ such that
$$H(\Delta_{\sqrt2})=\Delta_\gamma,\quad H\inc(S^1\times D^3_-)=H(D^7_-)\cap f(N)
\quad\text{and}\quad H\inc=\inc h_1 \quad\text{on}\quad S^1\times D^3_-.$$
The embedding $H$ is isotopic to $\id D^7_-$ by \cite[Theorem 3.2]{Hi76}.
This isotopy extends to an isotopy $H_t$ of $\id S^7$ by the Isotopy Extension Theorem \cite[Theorem 1.3]{Hi76}.
Then $H_t^{-1}fh_t$ is an isotopy of $f$.
Let us prove that the embedding $H_1^{-1}fh_1$ is standardized.

We have $H_1^{-1}fh_1=H_1^{-1}\inc h_1=\inc$ on $S^1\times D^3_-$.
Also if $H_1^{-1}fh_1(N-\im s)\not\subset \Int D^m_+$, then there is $x\in N-\im s$
such that $fh_1(x)\in H(D^7_-)$.
Then $fh_1(x)=H\inc(y)=\inc h_1(y)=fh_1(y)$ for some $y\in S^1\times D^3_-$.
This contradicts the fact that $fh_1$ is an embedding.

(b) Note that $[f]\in [f]+_s\tau(0,0)$ for every $s$.
Also
$\tau(l+l',b+b')\in\tau(l,b)+_{[S^1\times1_3]}\tau(l',b')$, and in particular, $\tau(l,b')\in\tau(l,b)+_{[S^1\times1_3]}\tau(0,b'-b)$.

(c) The calculation of $\lambda$ and $\beta$ (Lemma \ref{l:calam}.b and \ref{l:calc}.b) follows by (b) and the parametric additivity (Lemma \ref{l:Pad}).
However, Lemma \ref{l:Pad} for $\lambda$ and $\beta$ is proved using Lemma \ref{l:calam}.b and the particular case Lemma \ref{l:calc}.a of Lemma \ref{l:calc}.b, respectively.
For this reason, as well as for applications, it is convenient to state Lemma \ref{l:calam}.b and Lemmas \ref{l:calc}.a,b separately.

(d) In Corollary \ref{t:rem}.b it would be interesting to canonically construct at least part of the map $\tau_{\mk}$, and to give an algebraic (possibly non-canonical) construction of an $u$-symmetric form instead of $\lambda\tau_{\mk}(0,u,0)$.

(e) Using parametric connected sum one can define a map
$E^7(S^1\times S^3)^2\to 2^{E^7(S^1\times S^3)}$, and the same statement holds with $E^7$ replaced by $E^7_{\mk}$, cf.   \cite[\S2.1]{Sk15}, \cite{MAP}.
Corollary \ref{t:rem}.de means that this map

$\bullet$ is not single-valued for either $E^7$ or $E^7_{\mk}$,
this is unlike the situation in other dimensions \cite[\S2.1]{Sk15}, \cite{MAP},

$\bullet$ is  single-valued on $\lambda_\#^{-1}(0)\subset E^7_{\mk}(S^1\times S^3)$;
then it defines a group structure on $\lambda_\#^{-1}(0)$ (an unpublished direct proof was sketched by
S.~Avvakumov).

Cf.~\cite{II, III}
for smooth and PL analogues.
\end{Remark}

\begin{Remark}[to \S\ref{s:defandplan-s1s3}]\label{r:t1t2} Theorem \ref{t:s1s3} also follows from the construction of the map $\tau$, the
calculation of $\lambda$ and $\beta$ (Lemmas \ref{l:calam}.b and \ref{l:calc}.b), together with a version of Theorem \ref{t:gen} stating that $\beta_{u,l,\mk}$ is a 1--1 map for {\it every} representative $f'$ of an isotopy class in
$(\varkappa_{\mk}\times\lambda_{\mk})^{-1}(u,l)$ (such a version is essentially proved in the proof of Theorem \ref{t:gen}); we take $f'=\tau(0,l)$.
\end{Remark}

}

\section{Proofs of lemmas}\label{s:prelemmas}

\subsection{More notation}\label{s:prelemmas-not}

Recall that some notation was introduced in \S\S \ref{s:intro-main}, \ref{s:intro-strat} and \ref{s:defandplan-not}.

Denote by $\nud=\nud_f:S^7-\Int C_f\to N$ the oriented normal disk bundle of $f$ (the orientation of $\nud$ is inherited from the orientation of $S^7$ and $N$).

\smallskip
{\bf Definition of homology Alexander duality.}
Consider the following diagram:
$$
\xymatrix{& H_{q-2}(N) \ar[d]^{\nus^!} \ar@{=}[r]^{\text{PD}} \ar[dr]^{\widehat A} & H^{6-q}(N) \ar[d]^{AD}\\
H_{q+1}(C,\de) \ar[r]^{\de_C} \ar@{=}[d]^{\text{PD}} &H_q(\de C) \ar[d]^{\nus} \ar[r]^{i_C} &H_q(C) \\
H^{6-q}(C)  & H_q(N) \ar[l]^{AD} \ar[ul]^{A}}
$$
Here $AD$ is Alexander duality and $A=A_f,\widehat A=\widehat A_f$ are homology Alexander duality isomorphisms.\footnote{We use this name because they are compositions of the `ordinary  Alexander duality' and canonical (Poincar\'e or Poincar\'e-Lefschetz) isomorphisms.
If we identify canonically isomorphic groups, then the `homological Alexander duality' is the same as the `ordinary  Alexander duality'.
Our definition of $A,\widehat A$ reveals that the geometric meanings of the `ordinary Alexander duality'  isomorphisms $H_q(X)\to H^{n-q}(S^n-X)$ and $H^q(X)\to H_{n-q}(S^n-X)$ are different, so the isomorphism are not as analogous to each other as it seems from the notation.}
The lines are exact and the triangles are commutative by the well-known Alexander Duality Lemmas of \cite{Sk08', Sk10}.


Fix an orientation on $N$ and
denote by $[N]\in H_4$ and $[N_0]\in H_4(N_0,\partial)$ the corresponding fundamental classes.
We often use the class $A[N]\in H_5(C,\partial)$ which may be called the {\it homology Seifert surface} of $f$.


\begin{Lemma}[Intersection Alexander duality]\label{l:inaldu}
For every $y\in H_q$ and for every $z\in H_{4-q}$ we have $y\cap_N z=Ay\cap_C\widehat Az$.
\end{Lemma}

\begin{proof} For every $x\in H_r(\partial C)$ we have $\nus(x\cap_{\partial C}\nus^!z)=\nus x\cap_N z$.
Take $x=\partial Ay$.
Since $\nus x=y$ and $y\cap_N z\in\Z$, we obtain
$y\cap_N z=\partial Ay\cap_{\partial C}\nus^!z=Ay\cap_C\widehat Az$.
\end{proof}

{\bf Definition of the restriction homomorphism $r$.}
If $P$ is a (compact oriented) codimension $c$ submanifold of a manifold $Q$ and either $y\in H_k(Q)$ or
$y\in H_k(Q,\de)$, denote
$$r_{P,Q}(y)=r_P(y)=y\cap P:=\text{PD}((\text{PD}y)|_P)\in H_{k-c}(P,\de).$$
If $y$ is represented by a closed submanifold $Y\subset Q$ transverse to $P$, then $y\cap P$ is represented by $Y\cap P$.
Clearly, $y\cap_Q[P]=i_{P,Q}(y\cap P)$.

\smallskip
{\bf Definition of the difference class $d(\xi,\xi')$.}
(This definition is not used until \S\ref{s:lemmas-pad}.)
Let $Q$ be a $q$-manifold, and $\xi,\xi'$ non-zero sections of a $k$-dimensional vector bundle
over $Q$.
We define the {\em difference class}
$$
d(\xi,\xi')\in H_{q-k+1}(Q,\de) = H^{k-1}(Q) = H^{k-1}(Q;\pi_{k-1}(S^{k-1}))
$$
of $\xi$ and $\xi'$ to be the class of the preimage of the zero section under a general position homotopy from $\xi$ to $\xi'$.
This class is the {\em homology primary obstruction} to a vertical homotopy from $\xi$ to $\xi'$, and is equal to the Poincar\'{e} dual of the {\em cohomology primary obstruction} to a vertical homotopy from $\xi$ to $\xi'$, which is defined in  \cite[Theorem 6.4 Ch.~VI]{Wh78}.

Difference classes between other structures, e.g.~spin structures or framings,
on (a part of) a manifold are defined analogously.
(In fact, such structures can be represented as sections of certain bundles.
Then one can use the homological or cohomological definition of the primary obstruction
to vertical homotopy, the two definitions being related by Poincar\'{e} duality.)

\smallskip
 {\bf Definition of cobordism of homology classes together with supporting manifolds.}
(This definition is not used until \S\ref{s:lemmas-calc}.)
Assume that $P$ and $Q$ are closed manifolds and $x_j\in H_{k_j}(P)$, $y_j\in H_{k_j}(Q)$ for $j=1,\ldots,n$.
The tuples $(P,x_1,\ldots,x_n)$ and $(Q,y_1,\ldots,y_n)$ are called {\it cobordant} if there is a manifold $V$ and classes $v_j\in H_{k_j+1}(V,\partial)$ such that
$$\de V=P\sqcup(-Q),\quad \partial v_j\cap P=x_j \quad\mbox{and}\quad \partial v_j\cap Q=y_j
\quad\mbox{for every}\quad j=1,\dots,n.$$


The following definitions of a spin structure on a manifold $Q$ and of the spin characteristic
class $p^*_Q$  will not be used until \S\ref{s:lemmas-sf}.
Take a manifold $Q$ and its triangulation.
We write `skeleta of $Q$' for `skeleta of the triangulation'.

\smallskip
{\bf Definition of a spin structure.}
A {\it stable tangent spin structure} on $Q$ is a stable equivalence class of framing of the tangent bundle of $Q$ over the $1$-skeleton of $Q$, which extends to the $2$-skeleton of $Q$.
Two such framings are {\it stably equivalent} if they are homotopic, perhaps after addition of the trivial bundle with trivial framing to the tangent bundle of $Q$.
For brevity we omit `stable tangent'.
(This will not lead to a confusion because our manifolds $Q$ will have dimensions at least 3, so stable tangent spin structures are in 1--1 correspondence with tangent spin structures in the ordinary sense. See also Lemma \ref{l:krst}.)

The trivial spin structure on $S^{7}$ is the one induced from the spin structure on $D^8$ compatible with the orientation. (Note that this is the only spin structure on $S^7$.)

If $P \subset Q$ is a compatibly triangulated codimension zero submanifold, then a spin structure on $Q$ {\em induces} a spin structure on $P$ by restricting the framing over the 1-skeleton of $Q$ to a framing over the $1$-skeleton of $P$.
If $Q$ has boundary $\de Q$, then a spin structure on $Q$ induces a spin structure on $\de Q$.

If $F \colon Q \to P$ is a diffeomorphism with differential $dF$ and $s$ is a spin structure on $Q$, then the
{\it induced} spin structure $F_*s$ on $P$ is obtained by applying $dF$ to the vector fields over the
$1$-skeleton of $Q$ which define $s$.

\smallskip
{\bf Remark on spin structures via maps to $B\Spin$.}
Take a manifold $Q$.
A {\it stable oriented vector bundle} over $Q$ is a sequence $\xi_j$, $j\ge q$, of $j$-dimensional  oriented vector bundles over $Q$, together with isomorhisms $\xi_j\oplus n\to\xi_{n+j}$ of $(n+j)$-dimensional oriented vector bundles for any $n$ (here $q$'s can be different for different bundles, and $n$ is the trivial bundle of dimension $n$) \cite[\S18.10]{KL05}.\footnote{In this paper (like in textbooks) it is convenient to identify isomorphic bundles.
Let $n$ be the trivial bundle of dimension $n$.
Consider the equivalence on the set of oriented vector bundles over  $Q$
(of all different dimensions) generated by identifying bundles $x$ and $x\oplus1$.
Then the equivalence classes are in 1-1 correspondence with stable oriented vector bundles over $Q$
(indeed, to an equivalence class $(x,x\oplus1,x\oplus2,\ldots)$ there corresponds the stable oriented vector bundle $(x,x\oplus1,x\oplus2,\ldots), (\id,\id,\id,\ldots)$).}
Analogously one defines a stable spin vector bundle.
Let $B\SO$ and $B\Spin$ be the classifying spaces for stable oriented and spin vector bundles respectively.
Recall that $B\Spin=B\SO\!\left<4\right>$ is the (unique up to homotopy) 3-connected space for which there exists a fibration $\gamma:B\Spin\to B\SO$ inducing an isomorphism on $\pi_i$ for every $i\ge4$.

Let $\tau'_Q$ be the oriented tangent bundle of $Q$.
Let $\tau_Q$ be the stable oriented vector bundle which is the sequence $\tau'_Q\oplus n$, $n\ge0$, with the identical isomorhisms.
We use the same notation $\tau_Q$ for the classifying map $Q \to B\SO$ of $\tau_Q$, which is
the composition of the canonical map $B\SO_q \to B\SO$ with the classifying map of $\tau'_Q$.
A {\it spin lift} of $\tau_Q$ is a map $\overline{\tau}_Q \colon Q \to B\Spin$ such that
$\tau_Q = \gamma \circ \overline{\tau}_Q$.
Obstruction theory shows that a spin structure on $Q$ may be regarded as a homotopy class of a spin lift $\overline{\tau}_Q$ of the map $\tau_Q$, up to homotopy through spin lifts of $\tau_Q$.






\smallskip
{\bf Definition of $p_Q^*$ for a spin $q$-manifold $Q$.}
It is well-known that there is a generator $p\in H^4(B\Spin)\cong\Z$ such that $2p$
is the pull back in $H^4(B\Spin)$ of the universal first Pontryagin class $p_1\in H^4(B\SO)$
\cite[\S3, proof of Lemma 2.11]{CS11}.
Take a map $\overline{\tau}_Q \colon Q\to B\Spin$ corresponding to a spin structure on $Q$ and define
$$p_Q^*:=\text{PD}\overline{\tau}_Q p\in H_{q-4}(Q,\partial).$$
We remark that  $p^*_Q$ does not depend of the choice of spin structure on $Q$ \cite[page~170]{CCV08}
(for simply-connected $Q$ this is obvious).

\subsection{Lemmas on the $\varkappa$- and $\lambda$-invariants (\ref{l:La} and \ref{l:Lam})}
\label{s:prelemmas-equ}

In this subsection

$\bullet$ the larger intersection symbol $\bigcap$ denotes the intersection of homology classes in $\nus^{-1}N_0$;

$\bullet$ we identify $H_q(N_0,\de)$ with $H_q$ by the isomorphism $r_{N_0,N}$ for each $q\in\{1,2,3\}$;

$\bullet$ for a section $\xi \colon N_0\to\nus^{-1}N_0$ we use without mention that
$\xi[N_0]\bigcap\nus^!y=\xi y$ for each $q\in\{1,2,3\}$ and $y\in H_q$;

$\bullet$ we shorten $\lambda(f)$, $\overline{\lambda(f)}$ and
$\varkappa(f)$ to $\lambda$, $\overline\lambda$ and $\varkappa$ respectively;

$\bullet$ we define $\overline\varkappa:H_2\to\Z$ by $\overline\varkappa(y):=\varkappa\cap_N y$;

$\bullet$ before we prove that $\lambda$ and $\varkappa$ are independent of $\xi$ (Lemma \ref{l:La}.$\lambda'$,$\varkappa'$) we denote them by $\lambda_\xi$ and $\varkappa_\xi$ respectively.

\begin{Lemma}\label{l:La}
Let $\xi:N_0\to\nus^{-1}N_0$ be a section such that $i_C\xi$ is weakly unlinked.

(a) $i_C(\xi[N_0]\bigcap x)=A[N]\cap_Ci_Cx$ for each $q\in\{1,2,3\}$ and $x\in H_q(\nus^{-1}N_0)$.

($\lambda$)
$\overline{\lambda_\xi}=\widehat A^{-1}i_C\xi$ on $H_3$.

($\varkappa$)
$\overline{\varkappa_\xi}=\widehat A^{-1}i_C\xi$ on $H_2$.

($\lambda'$) $\overline\lambda_\xi(y)=\widehat A^{-1}(A[N]\cap_C\widehat Ay)$ for every $y\in H_3$.

($\varkappa'$) $\varkappa_\xi=A^{-1}(A[N]\cap_CA[N])$.

(e) $\varkappa=e^*(\xi^\perp)$, where $\xi^\perp$ is the normal bundle of $\xi:N_0\to\partial C$
(or, equivalently, the orthogonal complement to $\xi$ in $\nud|_{N_0}$).
\end{Lemma}

\begin{proof}[Proof of (a)]
By \cite[Section Lemma 2.5.a]{Sk10} $A[N]\cap \nus^{-1}N_0 = \xi [N_0]$.
Hence we have the equalities
$A[N]\cap_Ci_Cx =i_C\left((A[N]\cap \nus^{-1}N_0)\bigcap x\right)= i_C(\xi[N_0]\bigcap x)$.
\end{proof}

\begin{proof}[Proof of ($\lambda$) and ($\varkappa$)]
The formulas follow because for every closed oriented $q$-submanifold $X\subset N$, $q\in\{3,4\}$,
$$
\lk\phantom{}_{S^7}(fX,\xi Y)\overset{(1)}= \lk\phantom{}_{S^7}(\partial AX,\xi' Y)\overset{(2)}=
A[X]\cap_Ci_C\xi y\overset{(3)}= [X]\cap_N\widehat A^{-1}i_C\xi y$$
where

$\bullet$ $AX$ is any $(q{+}1)$-chain in $C$ whose boundary is in  $\partial C$ and represents $\partial A[X]$ there;

$\bullet$ $\xi'Y$ is a small shift of $\xi Y$ into the interior of $C$;

$\bullet$ (1) holds because $\nu\partial A[X]=[X]$, so
$fX$ is homologous to $\partial AX$ in
$S^7-\Int C$, and then because $\xi Y$ is homologous to $\xi'Y$ in $C$;

$\bullet$ (2) holds by definition of the linking coefficient;

$\bullet$ (3) holds by intersection Alexander duality (Lemma \ref{l:inaldu}).
\end{proof}

\begin{proof}[Proof of ($\lambda'$) and ($\varkappa'$)]
By (a)
for $x=\nus^!y$
we have $A[N]\cap_C\widehat Ay = i_C(\xi[N_0]\bigcap\nus^!y) = i_C\xi y\in H_q(C)$ for each $q\in\{2,3\}$.
So ($\lambda$) implies ($\lambda'$).
Also ($\varkappa$) implies that for every $y\in H_2$ we have
$$\varkappa\cap_N y= \widehat A^{-1}(A[N]\cap_C\widehat Ay)=
A[N]\cap_C(A[N]\cap_C\widehat Ay)= A^{-1}(A[N]\cap_CA[N])\cap_Ny.$$
Here the second and the third equalities follow by intersection Alexander duality (Lemma \ref{l:inaldu}).
This proves ($\varkappa'$).
\end{proof}

\begin{proof}[Proof of (e)]
Part (e) follows because for every $y\in H_2$
$$
\varkappa\cap_N y\overset{(1)}=A[N]\cap_Ci_C\xi y\overset{(2)}=\xi[N_0]\bigcap \xi y=
\xi[N_0]\bigcap\xi[N_0]\bigcap\nus^!y=e^*(\xi^\perp)\cap_N y.
$$
Here (1) holds by ($\varkappa$) and intersection Alexander duality (Lemma \ref{l:inaldu}), (2) holds by (a)
and the other two equalities are obvious.
\end{proof}

\begin{proof}[Proof of $\varkappa$-symmetry (Lemma \ref{l:Lam})]
Let $\xi \colon N_0 \to \de C$ be a weakly unlinked section obtained from a section $\zeta\colon N_0\to\nus^{-1}N_0$ by composing with the inclusion $\nus^{-1}N_0\to\de C$.
Let $-\xi \colon N_0 \to \de C$ be the weakly unlinked section
obtained by composing $-\zeta$ with the inclusion $N_0 \to \de C$.
If the homology classes $x,y \in H_3$ are represented by closed oriented 3-submanifolds (or integer 3-cycles) $X,Y\subset N_0$, then
$$
\lambda(y,x)=\lk\phantom{}_{S^7}(fY,\xi X)=\lk\phantom{}_{S^7}(\xi X,fY)=\lk\phantom{}_{S^7}(fX,-\xi Y).
$$
Hence
$$
\lambda(x,y)-\lambda(y,x)\overset{(1)}=
fX\cap_{S^7}Y_\xi\overset{(2)}=
\xi X\cap_{S^7}Y_\xi'\overset{(3)}=
\xi x\cap_{\partial C}\xi y\overset{(4)}=$$
$$=\xi[N_0]\bigcap\nus^!x\bigcap\xi[N_0]\bigcap\nus^!y\overset{(5)}=
\xi[N_0]\bigcap\xi[N_0]\bigcap\nus^!(x\cap_N y)\overset{(6)}=$$
$$=\xi[N_0]\bigcap\xi(x\cap_N y)\overset{(7)}=A[N]\cap_C  i_C\xi(x\cap_N y)\overset{(8)}=
 \varkappa\cap_N x\cap_N y,$$
where

$\bullet$ $Y_\xi\subset S^7$ is the 4-submanifold (with boundary) that is the union over $a\in Y$ of segments joining $\xi a$ to $(-\xi)a$
(or $Y_\xi$ is the corresponding integer 4-chain); we have $Y_\xi\cong Y\times I$;

$\bullet$ the algebraic intersection of submanifolds (or the cycle and the chain) in $S^7$ is defined because the first one does not intersect the boundary of the second one;

$\bullet$ (1) holds by definition of the linking coefficient and the above formula for $\lambda(y,x)$;

$\bullet$ $Y_\xi'$ is obtained from $Y_\xi$ by a small shift along $\xi-f$ considered as vector field on $fN$;

$\bullet$ (2), (4), (6) are clear;

$\bullet$ (3) holds because $Y_\xi'\cap\partial C=\xi Y$;

$\bullet$ (5) holds because $\dim\nus^{-1}N_0-\dim\xi[N_0]$ is even, so we can exchange
the order of terms in the cap product without changing the sign,
and because $\nus^!x\bigcap\nus^!y=\nus^!(x\cap_N y)$;

$\bullet$ (7) holds by Lemma \ref{l:La}.a;

$\bullet$ (8) holds by Lemma \ref{l:La}.$\varkappa$ and intersection Alexander duality (Lemma \ref{l:inaldu}).
\end{proof}

\subsection{Parametric additivity of $\varkappa$ and $\lambda$ (Lemma \ref{l:Pad})}\label{s:lemmas-pad}

Let $V$ be a 4-manifold with non-empty boundary.
Recall that an embedding $v:V\to D^7_+$ is called {\it proper}, if $f^{-1}\partial D^7_+=\partial V$.
Denote by

$\bullet$ $C=C_v$ the closure of the complement in $D^7_+$ to a tubular neighborhood of $v(V)$;

$\bullet$ $\nus=\nus_v:\Cl(\de C_v-\de D^7_+)\to V$ the restriction of the oriented normal vector bundle of $v$.

A section $\zeta:V\to \nus^{-1}V=\Cl(\de C_v-\de D^7_+)$ of $\nus$ is called {\it weakly unlinked} if
$$
i_{\de C,C}\zeta[V]=0\in H_4(C,\de D^7_+\cap\de C).
$$
We remark that if we would take $\de V=\emptyset$ in this definition, we would not obtain the definition of a weakly unlinked section for a closed manifold.

\begin{Lemma} \label{l:weun}
(a) Any section is weakly unlinked for any proper embedding $v:D^1\times S^3\to D^7_+$.

(b) For any proper embedding $v:V\to D^7_+$ of a connected 4-manifold $V$ with non-empty boundary
and torsion free $H_1(V,\de)$, a weakly unlinked section exists and is unique up to vertical homotopy
over any 2-skeleton of $V$.

(c) Let $g \colon N\to S^7$ and $s \colon D^1\times S^3\to N$ be embeddings such that $g|_{g^{-1}(D^7_\pm)}$ is a proper embedding into $D^7_\pm$ and $g^{-1}(D^7_-)=\im s$.
Then any weakly unlinked section for the abbreviation of $g$,%
\footnote{For a map $f:X\rightarrow Y$ and $A\subset X$, $f(A)\subset B\subset Y$, the {\it abbreviation}  $g:A\rightarrow B$ of $f$ is defined by  $g(x):=f(x)$.}
$N-\Int(\im s)\to D^7_+$, extends over $N_0:=N-\Int s(D^1\times D^3_-)$ to a weakly unlinked section for $g$.
\end{Lemma}

\begin{proof}[Proof of (a)]
Part (a) follows because
$$
H_4(C_v,\de D^7_+\cap\de C_v)\cong H_4(D^7-\inc(D^1\times S^3),\de D^7_+-\inc(S^0\times S^3))
\cong H_2(D^1\times S^3,\de)=0
$$
by Alexander duality, the homology exact sequence of a pair and the 5-lemma.
(Cf.~the proof of additivity of $\beta$, Lemma \ref{l:Add}, in \S\ref{s:lemmas-jhss}.)
\end{proof}

\begin{proof}[Proof of (b)]
(The proof is analogous to the `absolute' case \cite[Proposition 1.3]{BH70}.)
For sections $\xi$ and $\xi'$  of the normal bundle of $v$ the {\em difference class}
$d(\xi,\xi')\in H_2(V,\de)$ is defined in \S\ref{s:prelemmas-not}.
Alexander duality $\widehat A\colon H_q(V,\de)\to H_{q+2}(C_v,\de D^7_+\cap\de C)$ is defined analogously to the absolute case.
Then $d(\xi,\xi')=\pm\widehat A^{-1}(\xi-\xi')[V,\de V]$ analogously to \cite[Lemme 1.2]{BH70}.
This implies the uniqueness of a weakly unlinked section for $v$.

Let us prove the existence of a weakly unlinked section for $v$.
The normal Euler class $\overline e(v)$ assumes values in $H^3(V)\cong H_1(V,\de)$.
Since the normal bundle of $v$ is odd-dimensional, $2\overline e(v)=0$.%
\footnote{Alternatively, since $\overline e=0$ for embeddings of closed manifolds,
$2\overline e(v)=\overline e(2v)=0$ for the `double' $2v$ of $v$.}
Since $H_1(V,\de)$ is torsion free, $\overline e(v)=0$.
Since $V$ is connected and has non-empty boundary, it retracts to a 3-dimensional subpolyhedron.
Hence there is a section $\zeta\colon V \to\nus^{-1}V$ of $\nus$.
Denote by $U$ a closed neighbourhood of a 2-skeleton in $V$.
Construct a section $\xi'\colon U\to \nus^{-1}V$ such that
$d(\xi',\zeta|_U)=\mp\widehat A^{-1}\zeta[V]\cap U\in H_2(U,\de)$.
By \cite[Theorem 37.4]{St99} there is an extension $\xi$ of $\xi'$ to $V$ such that
$d(\xi,\zeta)=\mp \widehat A^{-1}\zeta[V]$.
Since $d(\xi,\zeta)=\pm\widehat A^{-1}(\xi-\zeta)[V]$, the extension $\xi$ is weakly unlinked.
\end{proof}

\begin{proof}[Proof of (c)]
Since $H^q(D^1\times D^3_-,\de D^1\times D^3_-;\pi_{q-1}(S^2))=0$ for every $q$,
obstruction theory entails that there is a section
$\xi\colon N_0\to\partial C_g$ extending a given weakly unlinked section for the abbreviation of $g$.
Define $x \in H_4(C_g)$ by $x:=i_{C_g}j_{\de C_g}^{-1}\ex^{-1}\xi[N_0,\de]$,
as in the definition of a weakly unlinked section for closed manifolds (\S\ref{s:defandplan-kl}).
We have
$$
x\cap D^7_+=\xi[g^{-1}(D^7_+)]=0\in H_4(C_g\cap D^7_+,C_g\cap \de D^7_+).
$$
Consider the following part of the homology exact sequence of the pair $(C_f, C_f \cap \de D^7_-)$:
$$
\xymatrix{H_4(C_g\cap D^7_-) \ar[r] \ar@{=}[d]^{\widehat A_g} & H_4(C_g) \ar[r]^{j}
& H_4(C_g,C_g\cap D^7_-) \ar@{=}[d]^{\ex_+}\\
H_2(D^1\times S^3)=0 & & H_4(C_g\cap D^7_+,C_g\cap \de D^7_+)}
$$
Since $0=x\cap D^7_+=\ex_+jx$, we have $x=0$, i.e.~$\xi$ is weakly unlinked for $g$.
\end{proof}

{\it Some notation for the proof of parametric additivity.}
Recall that $s\colon S^1\times D^3\to N$ is an embedding realizing $[s]\in H_1$ and
$R$ is the symmetry of $S^m$ with respect to the subspace defined by $x_1=x_2=0$.
Let $N_+:=\Cl(N-\im s)$ and $N_-:=\im s\subset N$.
For a cycle $X$ representing $[X]\in H_l$ and in general position to $N_+\cap N_-$ denote $X_\pm:=X\cap N_\pm$ a relative cycle representing
$[X]\cap N_\pm\in H_l(N_\pm,\partial)$.

By Lemma \ref{l:stand} we may assume that $f$ and an embedding $\tau_\alpha$ representing $\tau(l,b)$ are $s$-standardized and $i$-standardized, respectively.
Take the embedding $h$ given in the definition of parametric connected sum in \S\ref{s:defandplan-pr}.
Then both $f$ and $h$ satisfy the assumptions of Lemma \ref{l:weun}.c.
We have $f=h$ on $N_+$.
Using Lemma \ref{l:weun}.abc we can form a weakly unlinked section $\xi_h$ for $h$ as follows: we take the union of

$\bullet$ a weakly unlinked section for the abbreviation $N_+\to D^7_+$ of $h$ with

$\bullet$ the restriction to $s(D^1_+\times D^3_-)$ of a weakly unlinked section for the abbreviation $N_-\to D^7_-$ of $h$.

A weakly unlinked section $\xi_f$ for $f$ can be constructed analogously: simply replace $h$ by $f$.

\begin{proof}[Completion of the proof of parametric additivity for $\varkappa$]
For every $x\in H_2$ take an integer 2-cycle (or closed oriented 2-submanifold) $X\subset N$ representing $x$.
By general position we may assume that $X\subset N_+$.
There is an integer 3-chain $X'$ in $D^7_+$ such that $\partial X'=\xi_h X$.
So parametric additivity for $\varkappa$ holds because
$$
\varkappa(h)\cap_N x=\lk\phantom{}_{S^7}(hN,\xi_h X) = hN\cap_{S^7} X'\overset{(3)}=
fN\cap_{S^7} X'\overset{(4)}=\varkappa(f)\cap_N x,$$
where

$\bullet$ the equality (3) follows because $h=f$ on $N_+$ and $hN_-,fN_-\subset D^7_-$;

$\bullet$ the equality (4) is proved in the same ways as the first two equalities,
with $h$ replaced by $f$.
\end{proof}


\begin{proof}[Completion of the proof of parametric additivity for $\lambda$]
for every $x,y\in H_3$ take integer 3-cycles (or closed oriented 3-submanifolds) $X,Y\subset N$ representing $x,y$.
There are integer 4-chains $Y'_\pm$ in $D^7_{\pm}$ such that
$\partial(Y'_++Y'_-)=\xi_h Y$ and $\partial Y'_\pm\cap hN=\emptyset$.
We have
$$\lambda(h)(x,y)\overset{(a)}=\lk\phantom{}_{S^7}(hX,\xi_hY)=hX_+\cap_{S^7} Y'_+ + hX_-\cap_{S^7} Y'_-.$$
So parametric additivity for $\lambda$ follows because
$$
hX_+\cap_{S^7} Y'_+\overset{(*)}=fX_+\cap_{S^7} Y'_+\overset{(**)}=\lambda(f)(x,y)\quad\text{and}
$$
$$
hX_-\cap_{S^7} Y'_-\overset{(1)}= \tau_\alpha s^{-1}X_-\cap_{S^7} RY'_-\overset{(2)}=
(\lambda\tau_\alpha)(x_s,y_s)\overset{(3)}= l([s]\cap_N x)([s]\cap_N y).
$$
Here

$\bullet$ equality (*) holds because because $h=f$ on $N_+$ and $hN_-,fN_-\subset D^7_-$;

$\bullet$ equality (**) holds by equality (a) for $h$ replaced by $f$, because $fs=\tau_0|_{S^1\times D^3_-}$,
so $fX_-\cap_{S^7} Y'_-=0$ analogously to the calculation of $\lambda$ (Lemma \ref{l:calam}.c);

$\bullet$ equality (1) holds because $R$ preserves the orientation;

$\bullet$ $x_s:=([s]\cap_N x)[1_1\times S^3]\in H_3(S^1\times S^3)$ and analogously define $y_s$;

$\bullet$ equality (2) is proved below;

$\bullet$ equality (3) holds by the calculation of $\lambda$ (Lemma \ref{l:calam}.b).

To prove equality (2), we first apply the analogue of equality (**) for $f$ replaced by $\tau_\alpha$.
Observe that $R\xi s$ is a weakly unlinked section for $\tau_\alpha$.
This shows that the left hand side of equality (2) equals to the value of $\lambda \tau_\alpha$ on certain
homology classes in $H_3(S^1 \times S^3)$.
We have the equality $[s^{-1}X_-]=([s]\cap_N x)[1_1\times D^3_+]=x_s\cap(1_1\times D^3_+)$ and the same with $X,x$ replaced by $Y,y$.
Hence these homology classes are $x_s$ and $y_s$.
\end{proof}

\subsection{Agreement of Seifert classes (Lemmas \ref{l:jhss} and \ref{l:Agr})}\label{s:prelemmas-hss}

\begin{Lemma}\label{l:phi-unl}
If $\varkappa(f_0)=\varkappa(f_1)$ and $\lambda(f_0)=\lambda(f_1)$, \ $\varphi:\de C_0\to\de C_1$ is a bundle isomorphism and $\xi:N_0\to\de C_0$ a weakly unlinked section for $f_0$,
then $\varphi\xi$ is a weakly unlinked section for $f_1$.
\end{Lemma}

\begin{proof}
The proof we give follows the same line of reasoning as \cite[\S3, end of proof of Lemma 2.5]{CS11}.
By Lemma \ref{l:weunl} there exists a weakly unlinked section $\xi_1$ for $f_1$.
By Lemma \ref{l:La}.e
$$e^*((\varphi\xi)^\perp)=e^*(\xi^\perp)=\varkappa (f_0)=\varkappa (f_1)=e^*(\xi_1^\perp).$$
For any pair of sections $\zeta,\eta:N_0\to\partial C_1$ we have
$$ e^*(\zeta^\perp)-e^*(\eta^\perp)=\pm2d(\zeta,\eta)=\pm2i_Cj_{\partial C}^{-1}\ex\phantom{}^{-1}(\zeta-\eta),$$
where $d(\zeta,\eta)\in H_2(N_0;\pi_2(S^2))$ is the difference class \cite[Lemme 1.7]{BH70},
defined in \S\ref{s:prelemmas-not} above.
We apply this for $\zeta=\xi_1$ and $\eta=\varphi\xi$.
Since $H_2(N)$ has no 2-torsion, we obtain the equation
$i_Cj_{\partial C}^{-1}\ex^{-1}\varphi\xi=i_Cj_{\partial C}^{-1}\ex^{-1}\xi_1=0$; i.e.~$\varphi\xi$ is weakly unlinked.
\end{proof}

{\bf Definition of a joint Seifert class.}
A {\it joint Seifert class for $x\in H_q$ and a bundle isomorphism $\varphi:\de C_0\to\de C_1$}
is an element
$$X\in H_{q+1}(M_\varphi)\quad\text{such that}\quad X\cap C_k=A_kx\in H_{q+1}(C_k,\partial)\quad\text{for each}\quad k=0,1.$$
When the bundle isomorphism $\varphi$ is clear from the context,
we shall simply call $X$ a joint Seifert class for $x \in H_q$.
Note that a joint Seifert class, as defined in \S\ref{s:defandplan-b}, is
a joint Seifert class for $[N] \in H_4$ by Lemma \ref{l:desei}.a below.

\begin{Lemma}[Agreement of Seifert classes] \label{l:Agr}
Assume that $\varkappa(f_0)=\varkappa(f_1)$, that $\lambda(f_0)=\lambda(f_1)$,%
\footnote{Of these assumptions we need none for (a,b) and $q\not\in\{2,3,4\}$,
and only $\varkappa(f_0)=\varkappa(f_1)$ for (a,b) and $q\in\{2,4\}$.}
and that $\varphi:\partial C_0\to\partial C_1$ a bundle isomorphism.
Assume that the coefficients are $\Z$ or $\Z_d$ for some $d$; they are omitted from the notation.
Let $\de_k:=\de_{\de C_k,C_k}$.

(a) $\varphi\partial_0 A_0=\partial_1 A_1:H_q\to H_q(\de C_1)$.

(b) $\partial:H_{q+1}(M_\varphi,C_0)\to H_q(C_0)$ is zero for each $q$.

(c) For every $x\in H_q$ there is a joint Seifert class for $x$.
\end{Lemma}


\begin{proof}
Part (c) follows by (a) and the Mayer-Vietoris sequence for $M_\varphi$:
$$H_5(M_\varphi)
\xrightarrow{~r_0\oplus r_1~} H_5(C_0,\partial)\oplus H_5(C_1,\partial)
\xrightarrow{~\partial_0-\partial_1~} H_4(\partial C_0).$$

{\it For $q\ge5$} part (a) is trivial and part (b) follows because $H_{q+1}(M_\varphi,C_0)\cong H_q=0$.

{\it For $q=1$} part (b) is trivial and part (a) follows because $\varphi\partial_0 A_0=\partial_1 A_1=\nus_0^{-1}$.

{\it Now assume that $q\in\{2,3,4\}$.}
Let $\xi_0:N_0\to\partial C_0$ be a weakly unlinked section for $f_0$.
Since $\varkappa (f_0)=\varkappa (f_1)$, $\lambda(f_0)=\lambda(f_1)$ and $H_2$ has no 2-torsion,
by Lemma \ref{l:phi-unl} $\xi_1:=\varphi\xi_0$ is a weakly unlinked section for $f_1$.
In this proof $k\in\{0,1\}$.
The map
$$
\xi_k\oplus\nus^!_k \colon H_q\oplus H_{q-2}\to H_q(\partial C_k)
$$
is an isomorphism for each $q\in\{2,3\}$.
The map
$$
(j_k^{-1}\ex\phantom{}^{-1}_k\xi_k)\oplus\nus^!_k \colon H_4(N_0,\de)\oplus H_2\to H_4(\partial C_k)
$$
is an isomorphism.
Let $i_k:=i_{\de C_k,C_k}$.
We have $i_k\nus^!_k=\widehat A_k$.
By Lemma \ref{l:La}.$\lambda$,$\varkappa$ and \cite[Lemma 2.5.b]{Sk10} we have
$$i_k\xi_k=\widehat A_k\overline{\varkappa(f_k)}\quad\text{on}\quad H_2,\quad i_k\xi_k=\widehat A_k\overline{\lambda(f_k)}\quad\text{on}\quad H_3
 \quad\text{and}\quad i_kj_k^{-1}\ex\phantom{}^{-1}_k\xi_k[N_0]=i_k\de_kA_k[N]=0.
$$
Hence $\varphi\ker i_0=\ker i_1=\im\partial_1$.
Then the following commutative diagram
$$
\xymatrix{
& H_{q+1}(M_\varphi, C_0) \ar@{=}^{\ex}[d] \ar[r]^{\partial} & H_q(C_0) \\
H_q \ar[r]^{A_1} & H_{q+1}(C_1,\partial) \ar[r]^{\partial_1}
& H_q(\partial C_1) \ar[u]^{i_0\varphi^{-1}} \ar[r]^{i_1} & H_q(C_1), }
$$
shows that $i_0\varphi^{-1}\partial_1=0$, which implies (b).

Since $\nus_1\varphi\partial_0A_0=\nus_0\partial_0A_0=\id H_q=\nus_1\partial_1A_1$, we have
$\varphi\partial_0 A_0-\partial_1 A_1=\nus_1^!y$ for some map $y:H_q\to H_{q-2}$.
Applying $i_1$ to both sides and using that $\varphi\partial_0 A_0x\in\varphi\ker i_0=\ker i_1$
we obtain $0=\widehat A_1y$.
Hence $y=0$, i.e.~(a) holds.
\end{proof}


\begin{proof}[Proof of Lemma \ref{l:jhss}]
By Lemma \ref{l:Agr}.b, $i_{C_0, M_\varphi}$ is injective.
Since $H_2(M_\varphi,C_0) \cong H_2(C_1,\de)\cong H_1$ is torsion free, $i_{C_0, M_\varphi}$ is split injective.
As $S^2_{f_0} \in H_2(C_0) \cong \Z$ is primitive, $i_{C_0,M_\varphi} S^2_{f_0}\in H_2(M_\varphi)$ is primitive.
So by Poincar\'e duality there is a joint Seifert class $Y\in H_5(M_\varphi)$.
\end{proof}

\subsection{Spin bundle isomorphisms (Lemma \ref{l:spbuis})}


\newcommand{\spn}{\mathrm{sp}} 

For $k=0,1$ let $\spn_k$ be the (stable tangent) spin structure on $\de C_k$ induced from the trivial spin structure on $S^7$.

A bundle isomorphism $\varphi\colon \de C_0\to\de C_1$ is called {\bf spin} if it carries
$\spn_0$ to $\spn_1$.

\begin{Lemma}[Spin Lemma]\label{l:spbuis}
(a) For a bundle isomorphism $\varphi\colon  \de C_0 \to \de C_1$ the manifold $M_\varphi$ is spin if and only if $\varphi$ is spin.
Moreover, for every spin bundle isomorphism $\varphi \colon \de C_0 \to \de C_1$, there is a unique spin structure on
$M_\varphi$ whose restrictions to $C_0,C_1$ are induced from $S^7$.

(b) A spin bundle isomorphism exists and is unique (up to homotopy) over the 2-skeleton of any triangulation of $N$.
\end{Lemma}

{\bf Remark.}
(a) The `if' part of the Spin Lemma \ref{l:spbuis}.a can be proved as follows:
If $\varphi$ is spin, then the spin structures on $C_0,C_1$ coming from $S^7$ agree
up to homotopy on the boundaries.
Hence they can be glued together to give a spin structure on $M_\varphi$.

Below we give another proof together with the proof of the `only if' and the `moreover' parts.

(b) In order to illustrate the main idea of the Spin Lemma \ref{l:spbuis}.b let us sketch of a proof for
$N=S^1\times S^3$.
(The sketch is not formally used in the proof.)
Take a smooth map $\alpha \colon S^1\to SO_3$ representing the generator of $\pi_1(SO_3)$.
Identify $\partial C_f$ and $S^1\times S^3\times S^2$.
Define a bundle automorphism $f_\alpha$ of $S^1\times S^3\times S^2$ by
the formula $f_\alpha(x,y,z) = (x,y,\alpha(x)z)$.
The manifold $\de C_f$ has precisely two stable tangent spin structures
and the self-bundle-isomorphism $f_\alpha$ acts by exchanging these.
This implies the existence.  The uniqueness follows from the fact
that every spin bundle isomorphism is isotopic to $f_\alpha$ or the identity.

\smallskip
{\bf Definition of the difference class $d(\spn,\spn')$.}
Let $Q$ be a $q$-manifold.
For spin structures $\spn$ and $\spn'$ on  $Q$ their {\it difference} in
$$
H_{q-1}(Q,\de;\Z_2)=H^1(Q;\Z_2)=H^1(Q;\pi_1(SO))
$$
is the primary obstruction to homotopy from $\spn$ to  $\spn'$, cf.~\S\ref{s:prelemmas-not}.
(This is the homology class represented by the degeneracy set of a general position homotopy, through ordered $(q{-}1)$-sets of vectors, from $\spn$ to  $\spn'$.)

The following facts about spin structures are well known, follow by elementary obstruction theory, and will be used without mention:

$\bullet$ if the difference of $\spn$ and $\spn'$ is zero, then $\spn$ and $\spn'$ are equivalent, and

$\bullet$ for a $q$-manifold $Q$ the difference with a fixed spin structure is a 1--1 correspondence between $H_{q-1}(Q,\de;\Z_2)=H^1(Q;\Z_2)$ and spin structures on $Q$ up to equivalence.


For spin structures $\spn$ and $\spn'$ on $\de C_1$ let
$$
d(\spn,\spn')\in H_3(N; \Z_2)
$$
be the  preimage of the difference class in $H_5(\de C_1;\Z_2)$ under the isomorphism $\nus^!$.

\begin{proof}[Proof of the Spin Lemma \ref{l:spbuis}.a]
We have
$$\varphi\text{ is spin}\quad\Leftrightarrow\quad
d(\varphi_*\spn_0,\spn_1)=0\quad\Leftrightarrow\quad
w_2^*(M_\varphi)=0\quad\Leftrightarrow\quad
M_\varphi\text{ is spin}.$$
Here the second equivalence holds because

$\bullet$ $w_2^*(M_\varphi)=i_{C_0,M_\varphi}\widehat A_0d(\varphi_*\spn_0,\spn_1)$
by the naturality of the primary obstruction (the details are analogous to Lemma \ref{l:pondif} below), and

$\bullet$ $H_6(M_\varphi,C_0)\cong H_5=0$, so $i_{C_0,M_\varphi}$ is injective
(cf.~the Agreement Lemma \ref{l:Agr}.b for $s=5$).

Let us prove the `moreover' part.

Existence follows by the proof of the `if' part above.

Let us prove uniqueness.
The Mayer-Vietoris sequence for $M_\varphi=C_0\cup(-C_1)$ gives that the sum of the restriction homomorphisms
$H^1(M_\varphi;\Z_2) \to H^1(C_0;\Z_2) \oplus H^1(C_1;\Z_2)$ is injective.
So a spin structure on $M_\varphi$ is determined up to equivalence by its restrictions to $C_0$ and $C_1$.
Hence the required spin structure is unique.
\end{proof}

For the proof of the Spin Lemma \ref{l:spbuis}.b we need the following definitions and lemmas.

For bundle isomorphisms $\varphi,\psi\colon \de C_0\to \de C_1$ their {\it difference}
$$
d(\varphi,\psi) \in H^1(N;\pi_1(SO_3))=H_3(N; \Z_2)
$$
is the primary obstruction to homotopy of bundle isomorphisms between $\varphi$ and $\psi$.

\begin{Lemma}\label{l:spdif}
(a) For every bundle isomorphism $\varphi\colon \de C_0\to\de C_1$ and $v\in H_3(N;\Z_2)$ there is a bundle isomorphism
$\varphi_v\colon \de C_0\to\de C_1$ such that $d(\varphi_v,\varphi)=v$.

(b) For every pair of bundle isomorphisms $\varphi,\psi\colon D_0\to D_1$ and every spin structure $\spn$ on $\de C_0$ we have $d(\psi_*\spn,\varphi_*\spn)=d(\psi,\varphi)$.
\end{Lemma}

\begin{proof}[Proof of (a)]
Since $H_3$ has no torsion, $v=\rho_2\overline v$ for some $\overline v\in H_3$.
Since $H_3 \cong H^1(N) \cong [N, S^1]$, the class $\overline v$ is represented by an oriented 3-submanifold
$P\subset N$ that is  the preimage of a regular value of a map $N \to S^1$ representing $\overline v$.
Denote by $V\times D^1$ a tubular neighborhood of $V$ in $N$.
We have that $e(\nu_0|_V)=e(\nu_0)\cap V=0$ and that $\nu_0|_V$ is stably equivalent to the stable normal bundle of a parallelizable manifold.
Hence $\nu_0|_V$ is trivial.
Take a trivialization of $\nu_0$ over $V\times D^1$, i.e.~identify $\nus_0^{-1}(V\times D^1)$ and
$V\times D^1\times S^2$.
Take a smooth map $\alpha\colon D^1\to SO_3$ which maps a neighbourhood of the boundary to the identity
and which, modulo the boundary, represents the generator of $\pi_1(SO_3)\cong\Z_2$.
Then define
$$\varphi_1(x):=
\begin{cases}\varphi(x) & x\not\in\nus_0^{-1}(V\times D^1),\\
\varphi(a, t, \alpha(t)z) &x=(a,t,z)\in V\times D^1\times S^2=\nus_0^{-1}(V\times D^1).
\end{cases}$$
By construction $d(\varphi_1, \varphi) = v$.
So part (a) follows by taking $\psi:=\varphi_1$.%
\footnote{By (b) the equivalence class of $\varphi_1$ depends only on $v$ not on $V$ and on the trivialization.}
\end{proof}

\begin{proof}[Proof of (b)]
Carry out the construction of (a) for $\varphi$ and $v:=d(\psi,\varphi)$.
Now (b) follows because
$$d(\psi_*\spn,\varphi_*\spn)=d(\varphi_{1*}\spn,\varphi_*\spn)=vd=v,$$
where

$\bullet$ the first equality follows because $\psi$ is equivalent to $\varphi_1$ over the 1-skeleton;

$\bullet$ $a\in V$, $b\in S^2$ are any points and $d\in H^1(a\times D^1\times b,\de;\Z_2)\cong\Z_2$ is
the relative difference class of the spin structures
$\varphi_{1*}\spn|_{a\times D^1\times b}$, $\varphi_*\spn|_{a\times D^1\times b}$ which coincide on the boundary;

$\bullet$ the second equality follows by the naturality of the primary obstruction;

$\bullet$ the last equality is proved as follows.
Take a smooth map $\alpha\colon D^1\to SO_3$ which maps a neighbourhood of the boundary to the identity
and which, modulo the boundary, represents the generator of $\pi_1(SO_3)\cong\Z_2$.
Since the standard inclusion  $SO_2 \to SO_3$ induces an epimorphism $\pi_1(SO_2)\to\pi_1(SO_3)$,%
\footnote{There is an alternative proof not using this fact, cf.~proof of Lemma \ref{l:strdif} below.}
we may assume that $\alpha(t) \in SO_2$ for all $t \in D^1$;
i.e.~that $\alpha(t)b=b$ for all $t\in D^1$.
So $\varphi_v=\varphi$ over $V\times D^1\times b$, thus $d=1$ by definition of a spin structure.
\end{proof}

\begin{proof}[Proof of the Spin Lemma \ref{l:spbuis}.b]
The result \cite[Lemma 2.4]{CS11} asserts that there is a bundle isomorphism $\varphi'\colon \de C_0\to \de C_1$.
Hence using Lemma \ref{l:spdif} we can modifiy $\varphi'$ over the $1$-skeleton of $N$ to obtain
a spin bundle isomorphism $\varphi$.
Applying Lemma \ref{l:spdif} again, and using the fact that $\pi_2(SO_3) = 0$, we obtain that $\varphi$ is unique
up to vertical homotopy over the $2$-skeleton of $N$.
\end{proof}

\subsection{String bundle isomorphisms (Lemmas \ref{l:piso} and \ref{l:frbuis})}\label{s:lemmas-sf}

\newcommand{\str}{\mathrm{st}}


For $k=0,1$ let $D_k:=S^7-\Int C_k$ and let $\str_k$ be the homotopy class of the restriction to $D_k$ of the stable tangent framing on $S^7$.




The following String Lemma \ref{l:frbuis} can be regarded as a `complex' version of the Spin Lemma \ref{l:spbuis}.

\begin{Lemma}[String Lemma]\label{l:frbuis}
Assume that $\varkappa(f_0)=\varkappa(f_1)$ and $\lambda(f_0)=\lambda(f_1)$.
A bundle isomorphism $\varphi\colon \de C_0\to \de C_1$ is a $\pi$-isomorphism if and only if its extension $\Phi \colon D_0\to D_1$
carries $\str_0$ to $\str_1$.
\end{Lemma}

{\bf Remarks.} (a) The `if' direction of the String Lemma \ref{l:frbuis} is simple:
When $\Phi_*\str_0=\str_1$,
the stable tangent framings on $C_0,C_1$ coming from $S^7$ agree up to homotopy after identifying $\de C_0$ and $\de C_1$ by
$\varphi$.
So these framings can be glued together to give a stable tangent framing on $M_\varphi$.

(b) For a $\pi$-isomorphism $\varphi \colon \de C_0 \to \de C_1$, a stable tangent framing on $M_\varphi$ whose restrictions to $C_0, C_1$ are induced from from $S^7$ is not unique.

\smallskip
For the proof of Lemmas \ref{l:frbuis} and \ref{l:piso} we need the following definitions and lemmas.

\begin{Lemma} \label{l:p1}
For every map $\alpha\colon S^3\to\SO_3$ denote by $\xi(\alpha)$ the oriented 3-dimensional vector bundle over $S^4$

$\bullet$ whose total space is $(\R^3\times D^4_+) \cup_{\phi(\alpha)} (\R^3\times (-D^4_-))$, where $\phi(\alpha)(v, x) = (\alpha(x)v,x)$,

$\bullet$ and whose projection maps $(v,x)$ to $x$.

For a map $\alpha_1\colon S^3\to\SO_3$ representing $1\in\pi_3(\SO_3)=\Z$ we have
$p_1^*(\xi(\alpha_1))=4\in H_0(S^4)=\Z$.
\end{Lemma}

\begin{proof} The lemma is well-known; we present the proof for completeness.
We start with the identity $p_1^*:=p_1^*(\xi(\alpha_1))= \pm4 \in H_0(S^4)=\Z$ \cite{Mi56}.
To determine the sign in this equation let $S\xi(\alpha_1)$ be the total space of the oriented $S^2$-bundle associated to $\xi(\alpha_1)$ and $z\in H_4(S\xi(\alpha_1))\cong\Z$ a generator.
By \cite[Theorem 5]{Wa66} $p_1^*(S\xi(\alpha_1))\cap z \equiv 4 z^3 \mod 24$. This and $p_1^*=\pm4$ imply that
$p_1^*(S\xi(\alpha_1)) = 4z^2$, consequently $p_1^*=+4$.
\end{proof}

For stable tangent framings $\str$ and $\str'$ on $D_1$ which are homotopic on the 2-skeleton of $D_1$ their {\it difference}
$$
d(\str,\str')\in H^3(D_1;\pi_3(SO))=H^3(N)=H_1,
$$
is the primary obstruction to vertical homotopy between them, cf.~\S\ref{s:prelemmas-not}.
Here we use the zero section $N\to D^1$ to identify the cohomology groups.

A bundle isomorphism $\Phi\colon D_0\to D_1$ is called {\it spin} if its restriction to the boundary is spin.
Since the restriction induces an isomorphism $H^1(D_k; \Z_2)\to H^1(C_k;\Z_2)$, this is equivalent to carrying the
spin structure on $D_0$ induced from $S^7$ to the spin structure on $D_1$ induced from $S^7$.

Thus if $\Phi\colon D_0\to D_1$ is a spin bundle isomorphism, then the difference class
$d(\Phi_*\str_0,\str_1) \in H_1$ is well-defined by the uniqueness statements in the Spin Lemma \ref{l:spbuis}.a,b.

\begin{Lemma} \label{l:pondif}
For a spin bundle isomorphism $\Phi\colon D_0\to D_1$ we have
$p_{M_\varphi}^*=i_{C_0,M_\varphi}\widehat A_0d(\Phi_*\str_0,\str_1)$.
\end{Lemma}

\begin{proof}
Let

$\bullet$ $\delta_\varphi \in H^4(D_0\times I,D_0\times\de I)$ be the primary obstruction to the extension of
$\str_0|_{D_0\times0}\cup\str_1|_{D_0\times1}$ to a stable tangent framing of $D_0\times I$, and

$\bullet$ $\gamma_\varphi \in H^4(\de C_0\times I,\de)$ be the primary obstruction
to the extension of
$\str_0|_{\de C_0=\de C_0\times0}\cup  \str_1|_{\de C_1=\de C_0\times1}$ to a stable tangent framing of
$\de C_0\times I$.

We consider the following commutative diagram:
$$\xymatrix{
\delta_\varphi\in H^4(D_0\times I,D_0\times\de I) \ar[d]^{\text{restriction}}
& H^4(\Sigma(D_0\times I)) \ar[l]_(0.375){\cong} & H^3(D_0) \ar[l]_(0.35){\cong}
& H_1 \ar[l]_(0.375){\cong} \ar[dl]_{\nus^!} \ar[d]^{i_{C_0,M_\varphi}\widehat A_0} & \ni d(\Phi_*\str_0,\str_1) \\
\gamma_\varphi\in H^4(\de C_0\times I,\de)\ar[r]^(0.55){\cong} & H_3(\de C_0\times I)\ar[r]^(0.55){\cong}
& H_3(\de C_0)\ar[r]^(0.525){i_{\de C_0,M_\varphi}} & H_3(M_\varphi) & \ni p_{M_\varphi}^*
}.$$
The lemma follows because by the naturality of the primary obstruction

$\bullet$ the image of $d(\Phi_*\str_0,\str_1)$ under the first line of isomorphisms is $\delta_\varphi$;

$\bullet$ the restriction of $\delta_\varphi$ is $\gamma_\varphi$;

$\bullet$ the image of $\gamma_\varphi$ under the second line homomorphisms is  $p_{M_\varphi}^*$.

The latter statement follows because
$$M_\varphi\cong
C_0\bigcup\limits_{\de C_0=\de C_0\times0}\de C_0\times I\bigcup\limits_{\varphi\colon \de C_0\times1\to\de C_1}(-C_1)$$
and $p_{M_\varphi}^*$ is the primary obstruction to extending the spin structure on $M_\varphi$ to a stable tangent framing on $M_\varphi$.
(To see the latter, observe that class $p_{M_\varphi}^*$ is the primary obstruction to lifting $\overline\tau_{M_\varphi}\colon M_\varphi\to BSpin$ to $\gamma^*EO$, where  $\gamma^*EO\to BSpin$ is
the pullback to $BSpin$ of the universal $O$-bundle along the canonical map $\gamma \colon BSpin \to BO$.)
\end{proof}

\begin{proof}[Proof of the String Lemma \ref{l:frbuis}]
By the Spin Lemma \ref{l:spbuis}.a, it suffices to prove the result for spin $\varphi$.
For every spin bundle isomorphism $\varphi$ with extension $\Phi\colon D_0\to D_1$ we have
$$
\Phi_*\str_0=\str_1\quad\Leftrightarrow\quad
d(\Phi_*\str_0,\str_1)=0\quad\Leftrightarrow\quad
p_{M_\varphi}^*=0\quad\Leftrightarrow\quad
M_\varphi\text{ is parallelizable}.
$$
Here

$\bullet$ the first equivalence is the completeness of the obstruction;

$\bullet$ the second equivalence holds by Lemma \ref{l:pondif} because
$\varkappa(f_0)=\varkappa(f_1)$, $\lambda(f_0)=\lambda(f_1)$ and $H_2$ has no torsion,
so $i_{C_0,M_\varphi}$ is injective by the Agreement Lemma \ref{l:Agr}.b for $s=3$;

$\bullet$ the last equivalence holds because $\pi_l(SO_7)=0$ for $l=4,5,6$.
\end{proof}

By the Spin Lemma \ref{l:spbuis}.b every two spin bundle isomorphisms $\Phi,\Psi\colon D_0\to D_1$ are homotopic over a neighborhood of the 2-skeleton of some triangulation of $N$.
Hence the primary (and only) obstruction,
$$
d(\Phi,\Psi) \in H^3(N;\pi_3(SO_3))=H_1,
$$
to homotopy of spin bundle isomorphisms from $\Phi$ to $\Psi$ is well-defined.

\begin{Lemma}\label{l:strdif}
(a) For every $v\in H_3$ and spin bundle isomorphism $\Phi\colon D_0\to D_1$,
there is a spin bundle isomorphism $\Phi_v\colon D_0\to D_1$ such that $d(\Phi_v,\Phi)=v$.

(b) For every pair of spin bundle isomorphisms $\Phi,\Psi\colon D_0\to D_1$ and
for every stable tangent framing $\str$ on $D_0$, we have $d(\Psi_*\str,\Phi_*\str)=2d(\Psi,\Phi)$.
\end{Lemma}

\begin{proof}[Proof of (a)]
Take an oriented circle $V\subset N$ representing $v$.
Denote by $V\times D^3$ a tubular neighborhood of $V$ in $N$.
The restriction $\nud_0|_V$ is an oriented bundle over a circle and so is trivial.
Take a trivialization of $\nud_0$ over $V\times D^3$, i.e.~identify $(\nud_0)^{-1}(V\times D^3)$ and
$V\times D^3\times B^3$.
Take a smooth map $\alpha\colon D^3\to SO_3$ which
maps the boundary to the identity and represents (modulo the boundary) the generator $1 \in \pi_3(SO_3)\cong\Z$.
We define $\Phi_1$ by
$$\Phi_1(x):=\begin{cases}
\Phi(x) & x\not\in (\nud_0)^{-1}(V\times D^3),\\
\Phi((a,t,\alpha(t)z)) &x=(a,t,z)\in V\times D^3\times B^3=(\nud_0)^{-1}(V\times D^3).
\end{cases}$$
By construction $d(\Phi_1, \Phi) = v$.
\end{proof}

\begin{proof}[Proof of (b)]
Make the construction of (a) for $v:=d(\Psi,\Phi)$.
Now (b) follows because
$$d(\Psi_*\str,\Phi_*\str)=d(\Phi_{1*}\str,\Phi_*\str)=vd=2v,$$
where

$\bullet$ first first equation follows because $\Psi$ is equivalent over $N_0$ to $\Phi_1$;

$\bullet$ $a\in V$ is any point and $d\in H^3(a\times D^3\times 0,\de)\cong\Z$ is the relative difference class of stable tangent framings $\Phi_{1*}\str\,|_{a\times D^3\times0}$, $\Phi_*\str\,|_{a\times D^3\times0}$ which coincide
on the boundary;

$\bullet$ the second equation follows by the naturality of the primary obstruction;

$\bullet$ the last equality is proved as follows.
The relative difference class of tangent framings $\Phi_{1*}\str'|_{a\times D^3\times0}$,
$\Phi_*\str\,|_{a\times D^3\times0}$ coinciding on the boundary is 1.
Since the stabilization homomorphism $\pi_3(SO_3)\to \pi_3(SO)$ is identified with multiplication by 2, we have $d=2$.
\end{proof}

\begin{Lemma} \label{l:sp}
If $M$ is a closed spin 7-manifold, then $p_M^*$ is divisible by 2.
\end{Lemma}

\begin{proof}
We have $\rho_2(p_M)=w_4(M)=v_4(M)=0$.
Here

$\bullet$ the first equality is proved in \cite[\S3, Proof of Lemma 2.11.b]{CS11};

$\bullet$ the second equality holds because $M$ is spin;

$\bullet$ the third equality holds because $Sq^4\colon H^3(M;\Z_2)\to H^7(M;\Z_2)$ is trivial.
\end{proof}

\begin{proof}[Proof of Lemma \ref{l:piso}]
We use the String Lemma \ref{l:frbuis}.
Since $H_1$ has no 2-torsion, the uniqueness follows by Lemma \ref{l:strdif}.b.
Let us prove the existence.

By the Spin Lemma \ref{l:spbuis}.b there is a spin bundle isomorphism $\varphi \colon \de C_0 \to \de C_1$.
Let $\Phi$ be the extension of $\varphi$.
By Lemmas \ref{l:pondif} and \ref{l:sp},
$i_{C_0M_\varphi}\widehat A_0d(\Phi_*\str_0,\str_1)=p_{M_\varphi}^*$ is even.
Since $\varkappa(f_0)=\varkappa(f_1)$, $\lambda(f_0)=\lambda(f_1)$ and $H_2$ has no torsion,
$i_{C_0,M_\varphi}$ is injective by the Agreement Lemma \ref{l:Agr}.b for $s=3$.
Since we assume that $H_3(M_\varphi,C_0)\cong H_3(C_1,\de)\cong H_2$ is torsion free, it follows that
$d(\Phi_*\str_0,\str_1)$ is also even; i.e., $d(\Phi_*\str_0,\str_1)=2v$ for some $v\in H_1$.
By Lemma \ref{l:strdif}.a there is a spin bundle isomorphism $\Psi\colon D_0\to D_1$ such that $d(\Phi,\Psi)=-v$.
Then by Lemma \ref{l:strdif}.b
$$d(\Psi_*\str_0,\str_1)=d(\Phi_*\str_0,\str_1)+d(\Phi_*\str_0,\Psi_*\str_0)=2v-2v=0.$$
Then by the String Lemma \ref{l:frbuis} $\Psi$ is a $\pi$-isomorphism.
\end{proof}

\newcommand{\xra}{\xrightarrow}

\subsection{Joint Seifert classes (Lemmas \ref{l:webeta}, \ref{l:Add}, \ref{l:Dif} and \ref{l:desei})} \label{s:lemmas-jhss}

\begin{Lemma}[Description of joint Seifert classes] \label{l:desei}
Let $\varphi \colon \de C_0\to \de C_1$ be a bundle isomorphism and $i:=i_{C_0,M_\varphi}$.%

(a) A class $Y\in H_5(M_\varphi)$ is a joint Seifert class if and only if $Y\cap C_k=A_k[N]$ for each $k=0,1$.

(b) Let $Y\in H_5(M_\varphi)$ be a joint Seifert class.
A class $Y'\in H_5(M_\varphi)$ is a joint Seifert class if and only if
$$Y'=Y_y:=Y+i\widehat A_0y\quad\text{for some}\quad y\in H_3.$$

(c)  Let $Y\in H_5(M_\varphi)$ be a joint Seifert class.
Then $Y_y^2-Y^2=2i\widehat A_0\overline{\lambda(f_0)}(y)$ for every $y\in H_3$.

(d) If $p\in H_4(M_\varphi)$ and $q\in H_3(C_0)$,
then $p\cap_{M_\varphi} iq=A_0^{-1}(p\cap C_0)\cap_N\widehat A_0^{-1}q$.
\end{Lemma}

\begin{proof}[Proof of (a)] The `if' part follows because
$Y\cap_{M_\varphi} iS^2_{f_0}=(Y\cap C_0)\cap_{C_0}S^2_{f_0}=A_0[N]=1$.

Let us prove the `only if' part.
Since $H_1(C_k)=0$ and $H_2(C_k)\cong\Z$, we have $H_5(C_k,\de)\cong\Z$.
Since $(Y\cap C_k)\cap S^2_{f_k}=1$, the class $Y\cap C_k$ equals the generator $A_k[N]$ of $H_5(C_k,\de)$.
\end{proof}

\begin{proof}[Proof of (b)]
Look at the segment of (the Poincar\'e-Lefschetz dual to) the Mayer-Vietoris sequence:
$$H_5(\partial C_0)\xra{~i_{\de C_0,M_\varphi}~} H_5(M_\varphi)
\xra{~r_0\oplus r_1~}
H_5(C_0,\partial)\oplus H_5(C_1,\partial)\xra{~\partial_0-\partial_1~} H_4(\partial C_0).$$
The `only if' part follows because
$(Y'-Y)\cap S^2_{f_k}=0$ and $\nus_0^!\colon H_3(N)\to H_5(\de C_0)$ is an isomorphism, so
$Y'-Y\in\ker(r_0\oplus r_1)=\im i_{\de C_0,M_\varphi}=\im(i\widehat A_0)$.

The `if' part follows analogously because $iS^2_{f_k}\cap \im i_{\de C_0,M_\varphi}=0$.
\end{proof}

\begin{proof}[Proof of (c)]
We have
$$Y_y^2-Y^2=2Y\cap_{M_\varphi} i\widehat A_0y=2i((Y\cap C_0)\cap_{C_0}\widehat A_0y) \overset{(3)}=
2i(A_0[N]\cap_{C_0}\widehat A_0y) \overset{(4)}= 2i\widehat A_0\overline{\lambda(f_0)}(y).$$
Here (3) holds by (a) and (4) holds by Lemma \ref{l:La}.$\lambda'$.
\end{proof}

\begin{proof}[Proof of (d)]
We have $p\cap_{M_\varphi} iq = (p\cap C_0)\cap_{C_0}q = A_0^{-1}(p\cap C_0)\cap_N\widehat A_0^{-1}q.$
Here the last equality holds by intersection Alexander duality (Lemma \ref{l:inaldu}).
\end{proof}

\begin{proof}[Proof of Lemma \ref{l:webeta}]
Consider the following diagram:
$$\xymatrix{ H_4(M_\varphi,C_0;\Z_d) \ar[r]^(0.55){\partial} & H_3(C_0;\Z_d) \ar[r]^(0.475){i} & H_3(M_\varphi;\Z_d) \ar[r]^(0.45){j} &
H_3(M_\varphi, C_0;\Z_d) \ar[r]^{\ex} & H_3(C_1, \partial;\Z_d)}.$$
By Lemmas \ref{l:desei}.a and  \ref{l:La}.$\varkappa'$ we have $\ex j_{C_0,M_\varphi}Y^2=Y^2\cap C_1=(A_1[N])^2=A_1u$ with $\Z$-coefficients (these maps $\ex$ and $j_{C_0,M_\varphi}$ are not to be confused with the above $\Z_d$-homomorphisms $\ex$ and $j$ which are used elsewhere in this proof).
Hence $j\rho_dY^2=0$.
By the Agreement Lemma \ref{l:Agr}.b for $s=3$ $\partial=0$.
So $i$ is an isomorphism onto $\ker j$.
Now the existence and the uniqueness of $b_{\varphi,Y}$ follow because
$\widehat A_0=i_{\de C_0,C_0}\nus^!_0$ is an isomorphism.

The independence of $\beta(f_0,f_1)$ from $Y$ for fixed $\varphi$ follows by Lemma \ref{l:desei}.b,c because
a change of $Y$ by $Y_y$ leads to a change of $b_{\varphi,Y}=\widehat A_0^{-1}i_M^{-1}\rho_dY^2$ by $2\rho_d\overline{\lambda(f_0)}(y)$.

The independence of $\beta(f_0,f_1)$ from $\varphi$ is implied by the uniqueness of Lemma \ref{l:piso}, the independence
of $Y$ for fixed $\varphi$ and the following Lemma \ref{l:jhssn0}  applied to $f_0=f_1$ and $d=\di(\varkappa(f_0))$
(then $C_0=C_1$ but $\varphi_0\ne\varphi_1$ is possible).
\end{proof}

\begin{Lemma}\label{l:jhssn0}
Assume that $C_1\supset C_0$, $H_5(C_1,C_0)=0$,
and $\varphi_k\colon \de C_f\to \de C_k$, $k=0,1$, are bundle isomorphisms coinciding over $N_0$.
Then for every joint Seifert class $Y_0\in H_5(M_{\varphi_0})$ there is a joint Seifert class
$Y_1\in H_5(M_{\varphi_1})$ such that
$i_{C_f,M_{\varphi_1}}^{-1}\rho_dY_1^2\subset i_{C_f,M_{\varphi_0}}^{-1}\rho_dY_0^2\subset H_3(C_f;\Z_d)$ for every  $d$.
\end{Lemma}

\begin{proof}
Denote
$$M_k:=M_{\varphi_k}\quad\text{and}\quad
\overline M_k:=C_0\bigcup\limits_{\varphi_k|_{N_0}\colon \nus_f^{-1}(N_0)\to \nus_k^{-1}(N_0)}(-C_1)\quad\mbox{so that}
\quad M_k=\overline M_k\cup_{S^2\times\partial B^5}S^2\times B^5.$$
Since $C_1\supset C_0$, we have $\overline M_1\supset\overline M_0$.
Consider the following diagram, where $r=r_{\overline M_0,\overline M_1}$ and the coefficients are $\Z$ or $\Z_d$:
$$\xymatrix{
H_q(M_1) \ar[r]^{r_{M_1}} & H_q(\overline M_1,\de) \ar[r]^r & H_q(\overline M_0,\de) & \ar[l]_{r_{M_0}} H_q(M_0)\\
& H_q(\overline M_1) \ar[u]^{j_{\de\overline M_1,\overline M_1}} & \ar[l]_{i_{\overline M_0,\overline M_1}} H_q(\overline M_0)
\ar[u]^{j_{\de\overline M_0,\overline M_0}} & \\
&  H_q(C_f) \ar[uul]^{i_{C_f,M_1}} \ar[u]_{i_{\overline M_1}} \ar[ur]^{i_{\overline M_0}} \ar@(r,d)[rruu]_{i_{C_f,M_0}} }.$$
From the (Poincar\'e dual of the cohomology) exact sequence of the pair $(M_1,\overline M_1)$ we obtain that
$r_{M_1}$ is an epimorphism for $q=5$.
Analogously $r_{M_0}$ is a monomorphism for $q=3$.
By excision $H_5(\overline M_1,\overline M_0)\cong H_5(C_1,C_0)=0$.
Hence from the homology and the cohomology exact sequences of the pair $(\overline M_1,\overline M_0)$ we obtain that
$r$ is an epimorphism for $q=5$.
So we can take $Y_1\in r_{M_1}^{-1}r^{-1}r_{M_0}Y_0\in H_5(M_1)$.
Clearly, $Y_1\cap S^2_f=1$, i.e., $Y_1$ is a joint Seifert class for $\varphi_1$.
The required inclusion follows from
$\rho_dY_1^2\in r_{M_1}^{-1}r^{-1}r_{M_0}\rho_dY_0^2\in H_3(\overline M_1,\de)$
and  the commutativity of the diagram because $r_{M_0}$ is a monomorphism for $q=3$.
\end{proof}

\begin{proof}[Proof of the additivity of $\beta$ (Lemma \ref{l:Add})]
Denote $h:=f\#g$.
We have
$$\beta(h,f)=\beta(f,f)=0.$$
Let us prove the second equality. (It also follows by the transitivity of $\beta$, Lemma \ref{l:Dif}.)
The manifold $M_{\id\de C_f}=\de(C_f\times I)$ is parallelizable.
Take $Y:=\de(A_f[N]\times I)\in H_5(M_{\id\de C_f})$.
Clearly, $Y$ is a joint Seifert class for $\id\de C_f$.
By Lemma \ref{l:La}.$\varkappa'$ we have $A_f[N]^2=A_f\varkappa(f)\in H_3(C_f,\de)$.
Then $Y^2\in d(\varkappa(f))H_3(M_{\id\de C_f})$.
Hence $\beta(f,f)=0$.

Let us prove the first equality.
We may assume that $g(S^4)\cap C_f=\emptyset$, $\nu_f=\nu_h$ over $N_0$ and $C_h\supset C_f$.
Since $\pi_4(V_{7,4})=0$ \cite{Pa56}, all embeddings $S^4\to S^7$ are regular homotopic \cite{Sm59}.
Hence $f$ and $h$ are regular homotopic identically on $N_0$.
A regular homotopy between them extends to a regular homotopy of a tubular neighborhood of $fN$ in $S^7$, identical over $\nu_f^{-1}N_0$.
The bundle isomorphism $\varphi\colon \partial C_f\to\partial C_h$ defined by this regular homotopy is
identical over $N_0$.
Extend $\varphi$ to a bundle isomorphism $S^7-\Int C_f\to S^7-\Int C_h$ identical over $N_0$.
Then
by the String Lemma \ref{l:frbuis} $\varphi$ is a $\pi$-isomorphism.
Now the first equality holds by Lemma \ref{l:webeta} and Lemma \ref{l:jhssn0} for $f_0=f$, $\varphi_0=\id\de C_f$, $Y_0$ any Seifert class, $f_1=h$,  $\varphi_1=\varphi$ and $d=d(\varkappa(f))$.
The assumptions of Lemma \ref{l:jhssn0} are fulfilled because
$$H_5(C_h,C_f)\overset{\ex}\cong H_5(B^4\times D^3-h(B^4),B^4\times\partial D^3)\cong
H_5(B^4\times D^3-B^4\times 0,B^4\times\partial D^3)=0.$$
Here the second isomorphism holds by the 5-lemma
because of the Alexander duality isomorphism
$H_q(B^4\times D^3-h(B^4))\cong H_{6-q}(B^4,\partial)\cong H_q(B^4\times D^3-B^4\times0)$.
\end{proof}



\begin{proof}[Proof of the transitivity of $\beta$ (Lemma \ref{l:Dif})]
By Lemma \ref{l:piso} we have that there are $\pi$-isomorphisms
$\varphi_{01}\colon \partial C_0\to\partial C_1$ and $\varphi_{21}\colon \partial C_2\to\partial C_1$.
Let
$$
V:=M_{\varphi_{01}}\times[-1,0]\bigcup\limits_{-C_1\times 0=-C_1\times0}  M_{\varphi_{21}}\times[0,1].
$$
Denote $\varphi_{kl}:=\varphi_{lk}^{-1}$.
Observe that $M_{\varphi_{kl}}=-M_{\varphi_{lk}}$.
Then $\partial V=M_{\varphi_{10}}\sqcup M_{\varphi_{21}}\sqcup M_{\varphi_{02}}$.

In this paragraph $k\in\{10,21\}$.
Take any $x\in H_3$.
By Lemmas \ref{l:jhss} and \ref{l:Agr}.c there are joint Seifert classes $Y_{4,k}\in H_5(M_{\varphi_k})$
and $Y_{3,k}\in H_4(M_{\varphi_k})$, $Y_{3,k}$ for $x$.
Denote $I_{10}=[-1,0]$ and $I_{21}=[0,1]$.
Applying the Mayer-Vietoris sequences for $V$ we see that there is a class
$\overline Y_q\in H_{q+2}(V,\partial)$ such that
$$
\overline Y_q\cap (M_{\varphi_k}\times I_k)=Y_{q,k}\times I_k\in H_{q+2}(M_{\varphi_k}\times I_k,\partial)
\quad\mbox{for each}\quad k\in\{10,21\},\quad q\in\{3,4\}.
$$
Then for each $q\in\{3,4\}$ the class $Y_{q,20}:=\de\overline Y_q\cap M_{\varphi_{20}}\in H_{q+1}(M_{\varphi_{20}})$ is a joint Seifert class for $f_2$, $f_0$ and  $\varphi_{20}$, where $Y_{3,20}$ corresponds to $x$.
So the triple $(V,\overline Y_4,\overline Y_3)$ is a cobordism  between
$$(M_{\varphi_{20}},Y_{4,20},Y_{3,20})\quad\text{and}\quad
(M_{\varphi_{10}},Y_{4,10},Y_{3,10})\sqcup(M_{\varphi_{21}},Y_{4,21},Y_{4,21}).$$
Since $Y^2_{4,rl}\cap Y_{3,rl}$ is a characteristic number of such triples,
$$Y_{3,20}\cap_{M_{\varphi_{20}}} Y_{4,20}^2 =
Y_{3,10}\cap_{M_{\varphi_{10}}} Y_{4,10}^2+Y_{3,21}\cap_{M_{\varphi_{10}}} Y_{4,21}^2\in\Z.$$
Denote $d:=\di(\varkappa(f_0))$.
By Lemma \ref{l:webeta} there are
$b_{rl}:=\widehat A_l^{-1}i_{C_l,M_{\varphi_{rl}}}^{-1}\rho_dY^2_{4,rl}\in H_1(N;\Z_d)$.
Then by Lemma \ref{l:desei}.d $Y_{3,rl}\cap_{M_{\varphi_{rl}}} Y_{4,rl}^2=x\cap_N b_{rl}\in\Z_d$.
Hence $x\cap_N(b_{20}-b_{10}-b_{21})=0\in\Z_d$.
Since this holds for every $x\in H_3$ and $H_1$ is torsion free, $b_{20}=b_{10}+b_{21}\in H_1(N;\Z_d).$
By the String Lemma \ref{l:frbuis} the composition $\varphi_{20}\colon =\varphi_{01}^{-1}\varphi_{21}$ is a $\pi$-isomorphism.
Hence taking the quotients modulo $2\overline{\lambda(f_0)}(H_3)$ of both sides we obtain
$\beta(f_2,f_0)=\beta(f_1,f_0)+\beta(f_2,f_1)$.
\end{proof}


\subsection{Calculations of the $\beta$-invariant (Lemmas \ref{l:calc} and \ref{l:Pad} for $\beta$)} \label{s:lemmas-calc}

\begin{Lemma}[proved below in \S\ref{s:lemmas-calc}]\label{l:Bet}
Let $f_0,f_1\colon S^1\times S^3\to S^7$ be embeddings such that $\lambda(f_0)=\lambda(f_1)=0$.
Then
\footnote{There is a unique $x\in H_3(M_\varphi)$ such that $4x=p_1^*(M_\varphi)$.
This follows by the proof of the lemma (or because $p_1(M_\varphi)$ is divisible by 4 by Lemma \ref{l:sp} and  $H_3(M_\varphi)\cong\Z$; one proves the latter using the agreement of Seifert classes, Lemma \ref{l:Agr}.b, and the exact sequence of pair $(M_\varphi,C_0)$).}
$$
i_{C_0,M_\varphi}\widehat A_0\beta(f_0,f_1)=Y^2-\frac14 p_1^*(M_\varphi)\in H_3(M_\varphi)
$$
for any bundle isomorphism $\varphi\colon \de C_0\to\de C_1$ such that $M_\varphi$ is spin
and any joint Seifert class $Y\in H_5(M_\varphi)$.
\end{Lemma}

\begin{Lemma}[proved below in \S\ref{s:lemmas-calc}]\label{l:cobordbeta}
Let $f_0,f_1\colon S^1\times S^3\to S^7$ be embeddings and suppose that $\varphi,\varphi'\colon \de C_0\to \de C_1$ are bundle isomorphisms
such that $\varphi=\varphi'$ over $S^1\times D^3_-$, and over $S^1\times D^3_+$ the bundle isomorphism $\varphi$ is obtained from $\varphi'$
by twisting with the $+1\in\pi_3(SO_3)=\Z$ (i.e.~for the extension $\Phi:D_0\to D_1$, $N=S^1\times S^3$ and $V:=S^1\times 1_3$
we have $\varphi'=\Phi_1|_{\de C_0}$ in the notation of the proof of Lemma \ref{l:strdif}.a.).
Then the triple
$$
(M_{\varphi'},Y_4',Y_3')\quad\text{is cobordant to}\quad
(M_\varphi,Y_4,Y_3)\sqcup(\C P^3\times S^1,[\C P^2\times S^1],[\C P^2\times1_1])
$$
for some joint Seifert classes $Y_q\in H_{q+1}(M_\varphi)$ and $Y_q'\in H_{q+1}(M_{\varphi'})$, $q=3,4$,
where $Y_3$ and $Y_3'$ are those for $1_1\times S^3$.
\end{Lemma}

Lemmas \ref{l:cobordbeta}, \ref{l:cobordeta} and \cite[Cobordism Lemma]{Sk08} are analogous.

\begin{proof}[Proof of Lemma \ref{l:Bet}]
By Lemma \ref{l:piso} there is a $\pi$-isomorphism $\de C_0\to\de C_1$.
If $\varphi:\de C_0\to\de C_1$ is a $\pi$-isomorphism, then $M_\varphi$ is parallelizable, so $p_1(M_\varphi)=0$,
hence the required equality holds by definition of $\beta(f_0,f_1)$.



The proof of Lemma \ref{l:piso} shows that {\it every spin bundle isomorphism $\varphi:\de C_0\to\de C_1$
can be modified by a sequence of twistings with $\pm1\in\pi_3(SO_3)=\Z$ to obtain a $\pi$-isomorphism $\varphi'$.}
Hence it suffices to prove that
{\it if $\varphi,\varphi'$ are spin bundle isomorphisms and $\varphi$ is obtained from $\varphi'$ by twisting with $\pm1\in\pi_3(SO_3)=\Z$, then
$i_{C_0,M_\varphi}^{-1}(Y^2-\frac14 p_1^*(M_\varphi))=i_{C_0,M_{\varphi'}}^{-1}((Y')^2-\frac14 p_1^*(M_{\varphi'}))$
for some joint Seifert classes $Y\in H_5(M_\varphi)$ and $Y'\in H_5(M_{\varphi'})$.}

Let us prove this assertion.
Take any joint Seifert classes $Y_3\in H_4(M_\varphi)$ and $Y_3'\in H_4(M_{\varphi'})$ for $1_1\times S^3$.
The map $(i_{C_0,M_\varphi}\widehat A_0)^{-1}:H_3(M_\varphi)\to H_1(S^1\times S^3)=\Z$
is an isomorphism coinciding with $x\mapsto x\cap_{M_\varphi}Y_3$.
Analogous assertion holds with $\varphi,Y,Y_3$ replaced by $\varphi',Y', Y_3'$.
Now the assertion for `$+1$-modification' follows because by Lemma \ref{l:cobordbeta}
$$
(Y_4')^2\cap_{M_{\varphi'}} Y_3'-Y_4^2\cap_{M_\varphi} Y_3=
[\C P^2\times S^1]^2\cap_{\C P^3\times S^1}[\C P^2\times1_1]=
[\C P^1\times S^1]\cap_{\C P^3\times S^1}[\C P^2\times1_1]=1
$$
$$
\mbox{and}\qquad
p_1^*(M_{\varphi'})\cap_{M_{\varphi'}} Y_3'-p_1^*(M_\varphi)\cap_{M_\varphi} Y_3=
p_1^*(\C P^3\times S^1)\cap_{\C P^3\times S^1}[\C P^2\times1_1]=4.
$$
The assertion for `$-1$-modification' is analogous.
\end{proof}

\comment


The following sentence seems incorrect:
Analogously to Lemma 3.11.a it is proved that for every two spin bundle isomorphisms ?C0 ? ?C1 one of them can be obtained from (a spin bundle isomorphism equivalent to) the other by a sequence of twistings with +1 ? p3(SO3) = Z.

If we have two spin bundle isomorphisms with d(\varphi_0, \varphi_1) \in H_1 then a priori there is a further obstruction to their being homotopic.  This is of course an element of

\pi_4(SO_3) = \Z/2.

When I first wrote this argument, I used that the map from the group of bundle automorphisms \aut(\de C_0 \to S^1 \times S^3) to the integers, given by taking the characteristic number

Y^2-\frac14p_1^*(M_\varphi)

is a homomorphism.  Then of course the torsion group Z/2 cannot contribute to the characteristic number.  Without this point of view, we have to find a way around this final obstruction for the proof to be correct.

\begin{proof}[Proof of Lemma \ref{l:Bet}]
We are assuming that $M_\varphi$ is spin.
If $M_\varphi$ is parallelizable, then $p_1(M_\varphi)=0$, so the required equality holds by definition of $\beta(f_0,f_1)$.
By Lemma -ref, there is a string bundle isomorphism $\varphi''$.
Let $\Phi$ and $\Phi''$ be the extensions of $\varphi$ and $\varphi'$ to the disc bundle
If $d(\Phi,\Phi'') = 0$, the by the proof of Lemma -ref,
$M_\varphi''$ is a $\pi$-manifold and so the lemma holds for $M_\varphi''$.

Now $\varphi'$ is obtained from $\varphi$ by twisting with   $+1\in\pi_3(SO_3)=\Z$.
By Lemma \ref{l:cobordbeta}
$$
(Y_4')^2\cap_{M_{\varphi'}} Y_3'-Y_4^2\cap_{M_\varphi} Y_3=
[\C P^2\times S^1]^2\cap_{\C P^3\times S^1}[\C P^2\times*]=
[\C P^1\times S^1]\cap_{\C P^3\times S^1}[\C P^2\times*]=1
$$
$$
\mbox{and}\qquad
p_1^*(M_{\varphi'})\cap_{M_{\varphi'}} Y_3'-p_1^*(M_\varphi)\cap_{M_\varphi} Y_3=
p_1^*(\C P^3\times S^1)\cap_{\C P^3\times S^1}[\C P^2\times*]=4.
$$
Thus $Y^2{-}\frac14p_1^*(M_\varphi) = Y^2{-}\frac14p_1^*(M_\varphi')$.
Let $\Phi'$ be the extension of $\varphi'$ to the disc bundle.
By construction and the proof of Lemma -ref, $d(\Phi, \Phi')$ generates $H_1(S^1 \times S^3)$.

Now for any three bundle isomorphisms $\Phi_a, \Phi_b$ and
$\Phi_c$ the difference element satisfies the relation
\[ d(\Phi_a, \Phi_c) = d(\Phi_a, \Phi_b) + d(\Phi_b, \Phi_c) .\]
Let $d(\Phi, \Phi'') = d[\ast \times S^3] \in H_1(S^1 \times S^3; \pi_3(SO_3)) = H_1$.
If $d > 0$, then by inductively modifying $\varphi$ as above, we will obtain a bundle
isomorphism $\varphi'''$ with $Y^2{-}\frac14p_1^*(M_\varphi) = Y^2{-}\frac14p_1^*(M_\varphi''')$
and $d(\Phi''', \Phi'') = 0$.  Hence $\varphi'''$ is string and the lemma holds.
If $d<0$ then the argument doesn't work: we need a version of Lemma 3.16 with -1.
\end{proof}

\endcomment

\begin{proof}[Proof of Lemma \ref{l:cobordbeta}]
Take a map $\alpha\colon S^3\to SO_3$ representing $+1\in\pi_3(SO_3)$ and such that $\alpha|_{D^3_-}=\id S^2$.
Identify $\partial C_0$ with $S^2\times S^1\times S^3$ by any bundle isomorphism.
Identify $\partial C_1$ with $\partial C_0=S^2\times S^1\times S^3$ by $\varphi$.
Define a self-diffeomorphism
$$
\overline\alpha\quad\text{of}\quad
(S^2\times S^1\times S^3\times I)-\left(S^2\times S^1\times \Int D^3_+\times\left[\frac13,\frac23\right]\right)
$$
$$
\text{by}\quad \overline\alpha(x):=\begin{cases}
(\alpha(b)a,z,b,t) & x=(a,z,b,t)\in S^2\times S^1\times D^3_+\times[\frac23,1]\\
x &\text{otherwise.}\end{cases}
$$
Let
$$
V:=(C_0\times I)\bigcup\limits_{\overline\alpha}\ (C_1\times I)\quad\text{and}\quad
\Sigma:=S^2\times D^3\times\left[\frac13,\frac23\right].
$$
Then $V$ is a cobordism between $M_{\varphi'}$ and $M_\varphi\sqcup E_\alpha\times S^1$, where
$$
\widehat\alpha\colon \de\Sigma\to\de\Sigma\quad\text{maps $(a,b,t)$ to}\quad
\begin{cases}(a,b,t) &t<2 \\ (\alpha(b)a,b,2)& t=2\end{cases}
$$
$$
\text{and}\quad E_\alpha:=\Sigma\cup_{\widehat\alpha}\Sigma\overset{(1)}\cong
\frac{S^2\times D^4}{\{(a,b)\sim(\alpha(b)a,R (b))\}_{(a,b)\in S^2\times D^3_+}}\overset{(2)}\cong\C P^3.
$$
Here $R \colon D^3_+\to D^3_-$ is the reflection with respect to $\partial D^3_+=\partial D^3_-$.
The diffeomorphism (2) is well-known and is proved using a retraction to the dual $(\C P^1)^*\subset\C P^3$
of the complement to a tubular neighborhood of $\C P^1\subset\C P^3$.%
\footnote{Alternatively, (2) holds because $E_\alpha$ is an $S^2$-bundle over $S^4$
with characteristic map representing $+1\in\pi_3(SO_3)$.}

By the Agreement Lemma \ref{l:Agr}.a $\varphi\partial A_0=\partial A_1$ and $\varphi'\partial A_0=\partial A_1$.
Hence using the (Poincar\'e dual of the cohomology) Mayer-Vietoris sequence for $V$ we see that for each $q\in\{3,4\}$ there are
$$\overline Y_q\in H_{q+2}(V,\partial)\quad\text{such that}\quad
\overline Y_4\cap (C_k\times I)=A_k[S^1\times S^3]\times I\in H_6(C_k\times I,\partial)$$
$$\text{and}\quad\overline Y_3\cap(C_k\times I)=A_k[1_1\times S^3]\times I\in H_5(C_k\times I,\partial)
\quad\text{for each}\quad k=0,1.$$
Then
$Y_q:=\de\overline Y_q\cap M_\varphi$ and $Y_q':=\de\overline Y_q'\cap M_{\varphi'}$ are joint Seifert classes; $Y_4$ is for $[S^1\times S^3]$ and $Y_3$ is for $[1_1\times S^3]$.
Hence $(V,\overline Y_4,\overline Y_3)$ is a cobordism between
$$(M_{\varphi'},Y_4',Y_3')\quad\text{and}\quad (M_\varphi,Y_4,Y_3)\sqcup
(S^1\times E_\alpha,\partial\overline Y_4\cap(S^1\times E_\alpha),\partial\overline Y_3\cap(S^1\times E_\alpha)).$$
We have $\partial \overline Y_4\cap(S^1\times E_\alpha)=[S^1]\otimes y_4$ for a certain $y_4\in H_4(E_\alpha)$.
Since
$$\overline Y_4\cap (C_0\times I)=A_0[S^1\times S^3]\times I,\quad\text{we have}
\quad y_4\cap\Sigma=\left[*\times D^3_-\times\left[\frac13,\frac23\right]\right]\in H_4(\Sigma,\de).$$
Hence under (1) $y_4$ goes to a class whose intersection with $[S^2\times 0]$ in the quotient manifold is $+1$.
Therefore under the composition of (1) and (2) $y_4$ goes to
a class whose intersection with $[\C P^1]$ in $\C P^3$ is $+1$, i.e.~to $[\C P^2]$.

We have $\partial \overline Y_3\cap(S^1\times E_\alpha)=*\otimes y_3$ for a certain $y_3\in H_4(E_\alpha)$.
Analogously to the above under the composition of (1) and (2) $y_3$ goes to $[\C P^2]$.
\end{proof}

Take a map $\alpha\colon S^3\to\pi_3(V_{4,2})$ representing the element $(0,b)$.
In this subsection we abbreviate the subscript $\tau_k$ to $k$, e.g. $\nu_\alpha=\nu_{\tau_\alpha}$.

\begin{proof}[Proof of the calculation of $\beta$ (Lemma \ref{l:calc}.a)]
Take a normal vector field $e_1$ on $S^3\subset S^7$ tangent to $D^4\supset S^3$ and pointing outside $D^4$.
Take the standard framing $S^3\times D^3\to S^7$ of the orthogonal complement to $e_1$ in the normal bundle of $S^3\subset S^7$.
Take a normal vector field $e_2$ on $S^3\subset S^7$ orthogonal to $e_1$ and representing an element $b\in\pi_3(S^2)$ w.r.t.~the standard framing.
Since the `action' map $\pi_3(SO_3)\to\pi_3(S^2)$ is an isomorphism, $e_2$
can be completed  to a framing $e_2,e_3,e_4$ of the orthogonal complement to $e_1$,
and this framing is unique up to homotopy.
Recall that $\im\tau_\alpha$ is formed by ends of $\varepsilon$-length vectors normal to $S^3\subset S^7$ in the subbundle spanned by $e_1$ and $e_2$ of the normal bundle to $S^3$ in $S^7$ (for some small $\varepsilon$).

Recall that $\nu_\alpha=\nu_{\tau_\alpha}$ is the normal bundle of $\tau_\alpha$.
Take a section $\xi_\alpha$ of $\nu_\alpha$ in the 2-plane subbundle, spanned by $e_1$ and $e_2$, of the normal bundle to $S^3\subset S^7$, so that for every point $(x,y)\in S^1\times S^3$ the vector $\xi_\alpha(x,y)$ looks into the
2-disk bounded by $\tau_\alpha(S^1\times y)$ but not outside.
Then $\xi_\alpha,e_3,e_4$ is a framing of $\nu_\alpha$.

For every $x\in S^3$ the circle $\xi_\alpha(S^1\times x)\subset S^3\times x$  bounds a 2-disk in $x\times D^4$.
The union of such 2-disks along $x\in S^3$ is a $5$-manifold $W_\alpha \cong D^2 \times S^3$
with boundary $\xi_\alpha(S^1 \times S^3)$.
Choose the orientation on $W_\alpha$ so that $\de W_\alpha=\xi_\alpha(S^1\times S^3)$.
Since $A_\alpha^{-1} = \nus_\alpha\de$ and $\nus_\alpha \xi_\alpha = \id N$,
the manifold $W_\alpha$ represents the relative homology class
$[W_\alpha]=A_\alpha[S^1\times S^3]\in H_5(C_\alpha,\partial)$.

Let
$$D^4_\alpha:=(1-\varepsilon)D^4\subset S^7\quad\text{and}\quad S^3_1:=1_1\times S^3.$$
We have $\de D^4_\alpha=\xi_\alpha(S^3_1)$.
Hence $[D^4_\alpha]=A_\alpha[S^3_1]\in H_4(C_\alpha,\de)$ for a certain orientation on $D^4_\alpha$.

Make analogous construction for $\alpha$ replaced by the constant map $\alpha_0$.
We obtain the standard embedding $\tau_0=\tau_{\alpha_0}$, a section $\xi_0$, a framing,
$$[W_0]=A_0[S^1\times S^3]\in H_5(C_0,\de)\quad\text{and}\quad [D_0^4]=A_0[S^3_1]\in H_4(C_0,\de).$$
Take a bundle isomorphism $\varphi\colon \de C_0\to\de C_\alpha$ defined by the constructed framings.
We may assume that $f_0=f_\alpha$ over a neighborhood of $S^1\times1_3$.
Hence $\varphi$ carries the spin structure $\spn_\alpha$ on $\de C_\alpha$ coming from $S^7$ to
the spin structure $\spn_0$ on $\de C_0$ coming from $S^7$.
Therefore $M_\varphi$ is spin.
Since $\varphi\xi_0=\xi_\alpha$, we have
$$
\varphi\de W_0=\varphi\xi_0(S^1\times S^3)=\xi_\alpha(S^1\times S^3)=\de W_\alpha\quad\text{and}
\quad \varphi\de D^4_0=\varphi\xi_0(S^3_1)=\xi_\alpha(S^3_1)=\de D^4_\alpha.
$$
Hence $W_0\cup_{\de}(-W_\alpha)$ and $\Sigma^4:=D^4_0\cup_{\de}(-D^4_\alpha)$ with their natural orientations
represent joint Seifert classes $Y_4\in H_5(M_\varphi)$ and $[\Sigma^4]\in H_4(M_\varphi)$, where $[\Sigma^4]$ corresponds to $[S^3_1]$.
We have
$$
W_0\cap D^4_0=W_\alpha\cap D^4_\alpha=\emptyset\ \ \ \Rightarrow\quad
(W_0\cup_{\de}(-W_\alpha))\cap\Sigma^4=\emptyset\ \ \ \Rightarrow\quad
Y_4\cap_{M_\varphi}[\Sigma^4]=0\ \ \ \Rightarrow\quad Y_4^2\cap_{M_\varphi}[\Sigma^4]=0,
$$
$$
\text{and}\quad p_1^*(M_\varphi)\cap_{M_\varphi}[\Sigma^4]=p_1^*(\tau_{M_\varphi}|_{\Sigma^4})=
p_1^*(\Sigma^4)+p_1^*(\mu)=-4b,$$
where

$\bullet$ $\mu$ is the normal bundle of $\Sigma^4$ in $M_\varphi$;

$\bullet$ the last equality follows because $S^4$ is stably parallelizable, the map $p_1\colon \pi_3(SO_3)\to\Z$ is multiplication by 4 \cite{DNF12}, and $\mu$ corresponds to the preimage of $-b$ under the `action' map $\pi_3(SO_3)\to\pi_3(S^2)$ because the obstruction to existence of a non-zero section of $\mu$ is $-b$.

Let us prove the latter statement.
Take the normal vector field $e_2^0$ on $S^3\subset S^7$ orthogonal to $e_1$ and representing the constant map
$S^3\to S^2$ w.r.t. the standard framing.
The corresponding section of the normal bundle of $\de D^4_0\subset S^7$ is tangent to $\de C_0$.
That section is mapped under $d\varphi$ to the section of the normal bundle of $\de D^4_\alpha\subset S^7$ corresponding to $e_2$ (here $\varphi\colon \de C_0\to \de C_\alpha$ is thought of as a diffeomorphism rather than
a bundle isomorphism).
Clearly, $e_2^0$ extends to a non-zero section of the normal bundle of $D^4_0\subset S^7$.
The obstruction to extension of $e_2$ to a non-zero section of the normal bundle of $D^4_\alpha\subset S^7$ is $b$.
Since $\mu$ is a bundle over $D^4_0\cup_\de(-D^4_\alpha)$,
the obstruction to the existence of a non-zero section of $\mu$ is $-b$.

Thus
$\beta(\tau_0,\tau_\alpha)\cap_{S^1\times S^3}[S^3_1]= (Y_4^2-\frac14p_1^*(M_\varphi))\cap_{M_\varphi}[\Sigma^4]=b$
by Lemmas \ref{l:Bet} and \ref{l:desei}.d.
Hence $\beta(\tau_0,\tau_\alpha)=b[S^1\times1_3]$.
\end{proof}

\begin{proof}[Proof of the parametric additivity of $\beta$ (Lemma \ref{l:Pad})]
We use the notation from \S\ref{s:lemmas-pad}.
Analogously to Lemma \ref{l:piso} one proves that {\it there is a $\pi$-isomorphism
$\psi\colon \de C_\alpha\to \de C_0$ such that $\psi(\de C_\alpha\cap D^7_\pm)\subset(\de C_0\cap D^7_\pm)$.}

For an embedding $g$ denote $C_{g\pm}:=C_g\cap D^7_\pm$.
Observe that
$$
C_{f-}=C_{\alpha-}=C_{0-}\cong C_{0+},\quad C_{h+}=C_{f+} \quad\mbox{and}\quad C_{h-}=C_{\alpha+}.
$$
Then $\id C_{f+}\cup R$ gives diffeomorphisms
$$
C_h\ \cong\ C_{f+}\bigcup\limits_{C_{f+}\cap \de D^7_+=C_{\alpha+}\cap \de D^7_+}C_{\alpha+}\qquad\mbox{and}\qquad
\de C_h \cong
\ \de C_f\cap D^7_+\bigcup\limits_{\de C_f\cap \de D^7_+=\de C_\alpha\cap \de D^7_+}\de C_\alpha\cap D^7_+
$$
(here $C_{\alpha+}$ comes with the positive orientation because $R$ preserves the orientation).
Identify $\de C_h$ with the right-hand side of the latter equality.
Let $\varphi\colon \partial C_h\to\partial C_f$ be $\id(\de C_f\cap D^7_+)\cup\psi|_{\de C_\alpha\cap D^7_+}$.
Then by the String Lemma \ref{l:frbuis} $\varphi$ is a $\pi$-isomorphism.
(This is proved by considering the restrictions of stable tangent framings coming from $S^7$ to $\de C_f\cap D^7_+$,
$\de C_\alpha\cap D^7_+$ and $\de C_0\cap D^7_+$.)

Let
$$
V:=M_{\id\de C_f}\times[-1,0]\bigcup\limits_{R _3}(-M_\psi)\times[0,1],
$$
where $R _3\colon \left(C_{f-}\cup(-C_{f-})\right)\times0\to\left((-C_{\tau-})\cup C_{\tau_0,-}\right)\times0$
is the union of two copies of the reflection with respect to the hyperplane $x_3=0$.
Then $\partial V=M_\varphi\sqcup(-M_\psi)\sqcup(-M_{\id\de C_f})$.
Now the proof is completed analogously to the second paragraph of the proof of $\beta$-transitivity
(Lemma \ref{l:Dif}) using the calculation of $\beta$ (Lemma \ref{l:calc}.a).
\end{proof}




\aronly{

\subsection{Appendix: some remarks to \S\ref{s:prelemmas}}\label{s:prelemmas-rem}

\begin{Remark}[to \S\ref{s:prelemmas-equ}]\label{r:toequ}
(a) The class $\varkappa\in H_2$ measures the linking of $2$-cycles in $N$ and the `top cell' of $N$ under $f$:
if $f=f'$ on $N_0$, then
$(\varkappa(f')-\varkappa(f))\cap N_0=2\widehat A_{f|_{N_0}}^{-1}[f(B^4)\cup f'(B^4)]\in H_2(N_0;\Z)\cong H_2$
(this is proved analogously to \cite[\S2, The Bo\'echat-Haefliger invariant Lemma]{Sk08'}).

(b) Weakly unlinked sections may differ on a 3-skeleton of $N_0$ even up to vertical homotopy, and
{\em a priori} changing $\xi$ on a 3-skeleton could change the integer $\lk_{S^7}(fX,\xi Y)$ in the definition of $\lambda$.
However different choices of $\xi$ do not change $\lk_{S^7}(fX,\xi Y)$.
The formal explanation for this is given in Lemma \ref{l:La}.$\lambda'$.
Informally, the change is trivial because it `factors through' $H_3(S^2)=0$.

(c) If in the definitions of $\varkappa$ and $\lambda$ we would take an arbitrary (i.e. not weakly unlinked) section $\xi$, we would obtain different values.
Note that $2\lambda(x,y)=\lk_{S^7}(fx,\xi y)+\lk_{S^7}(fx,-\xi y)$ for any (i.e., not necessarily weakly unlinked)
section $\xi$ (D. Tonkonog, unpublished, cf. \cite{To10}).

(d) Although a weakly unlinked section is only defined over $N_0$, its construction involves all of the embedding $f$
via the inclusions $\nus^{-1}N_0 \to \de C \to C$, and not only $f|_{N_0}$.
For embeddings $N_0\to\ S^7$ an analogue of $\varkappa$ is not defined and only `a part' of $\lambda$ is defined
(D. Tonkonog, unpublished, cf.~\cite{To10}).

(e) Let $\xi:N_0\to\nus^{-1}N_0$ be a section such that $i_C\xi$ is weakly unlinked.
Then
$\nus^!\overline\lambda=\xi-\partial A$ on $H_3$ and
$\nus^!\overline\varkappa=\xi-\partial A$ on $H_2$.

{\it Proof.} Let $q\in\{2,3\}$.
Since $\nus\partial A=\id H_q=\nus\xi$, for every $y\in H_q$ there is a class $x(y)\in H_{q-2}$ such that
$\partial Ay-\xi y=\nus^!x(y)$.
Then $0=i_C\partial Ay=i_C\xi y+i_C\nus^!x(y)=i_C\xi y+\widehat Ax(y)$.
Now the required equalities  follow by Lemma \ref{l:La}.$\lambda,\varkappa$.
\end{Remark}


\begin{Remark}[to \S\ref{s:prelemmas-hss}]\label{r:tohss}
(a) Lemma \ref{l:jhss} also follows from Lemmas \ref{l:Agr}.c and \ref{l:desei}.a.

(b) {\it An alternative proof of the Agreement Lemma \ref{l:Agr}.a,b for $q\in\{2,3,4\}$.}
We generalize the proof of \cite[Agreement Lemma 2.5]{CS11} which is part (a) for $q=4$.

Let $\xi_0:N_0\to\partial C_0$ be a weakly unlinked section for $f_0$.
Since $\varkappa (f_0)=\varkappa (f_1)$, $\lambda(f_0)=\lambda(f_1)$ and $H_2$ has no 2-torsion,
by Lemma \ref{l:phi-unl} $\xi_1:=\varphi\xi_0$ is a weakly unlinked section for $f_1$.

We have $\nus_1\varphi\de_0A_0=\nus_0\de_0A_0=\id H_q=\nus_1\de_1A_1$.
Also
$$
\xi_1^!\varphi\de_0A_0=\xi_0^!\de_0A_0\overset{(2)}=
\left\{\begin{cases}
\overline\varkappa(f_0) &q=4\\
\xi_0^!\xi_0-\xi_0^!\nus_0^!\overline\lambda(f_0)=\xi_0^!\xi_0-\overline\lambda(f_0) &q=3\\
\xi_0^!\xi_0-\xi_0^!\nus_0^!\overline\varkappa(f_0)=\xi_0^!\xi_0-\overline\varkappa(f_0) &q=2\\
\end{cases}\right\}\overset{(3)}=\xi_1^!\de_1A_1.
$$
Here

$\bullet$ (2) holds because $\xi_0^!\de_0A_0[N]=\varkappa$ \cite[Lemma 2.5.b]{Sk10} and
by Remark \ref{r:toequ}.e.

$\bullet$ (3) holds because $\xi_0^!\xi_0=\xi_0^!\varphi^{-1}\varphi\xi_0=\xi_1^!\xi_1$.

Now (a) follows because the map $\nus_1\oplus\xi_1^!:H_3(\partial C_1)\to H_3\oplus H_1$ is an isomorphism.

Let $i_k:=i_{\de C_k,C_k}$ and consider the diagram from \S\ref{s:prelemmas-hss}.
Part (b) follows from (a) because
$$
(a)\quad\Rightarrow\quad
\varphi\im(\de_0A_0)=\im(\de_1A_1)\quad\Leftrightarrow
\varphi\im\de_0=\im\de_1\quad\Leftrightarrow\quad
\varphi\ker i_0=\ker i_1\quad\Rightarrow$$
$$\Rightarrow\quad\varphi\ker i_0\supset\im\partial_1
\quad\Leftrightarrow\ \ i_0\varphi^{-1}\partial_1=0\quad\Leftrightarrow\quad(b).
$$
(In fact, $(a)~\Leftrightarrow~ \varphi\ker i_0=\ker i_1$, see the proof from \S\ref{s:prelemmas-hss}.)
\end{Remark}

\begin{Remark}[to \S\ref{s:lemmas-sf}]\label{r:tosf}
The existence part of Lemma \ref{l:piso} (under weaker assumption `$\lambda(f_0)(x,x)\equiv\lambda(f_1)(x,x)\mod2$ for every $x$') can be proved as follows.
By the regular homotopy classification of embeddings (Proposition \ref{l:reho})
$f_0$ and $f_1$ are regular homotopic.
(We actually only need the regular homotopy over $N_0$ which is easier to prove.)
A regular homotopy between $f_0$ and $f_1$ extends to a regular homotopy of a tubular neighborhood of $f_0(N)$ in $S^7$.
By the String Lemma \ref{l:frbuis} the bundle isomorphism defined by this regular homotopy is a $\pi$-isomorphism.
(In \S\ref{s:lemmas-sf} we give another proof together with the proof of the uniqueness part of Lemma \ref{l:piso}.)
\end{Remark}

\begin{Remark}[to \S\ref{s:lemmas-jhss}]\label{r:tojhss}
(a) In Lemma \ref{l:desei} we do not assume that $\varkappa(f_0)=\varkappa(f_1)$ or $\lambda(f_0)=\lambda(f_1)$.
However, we apply this lemma when $\varkappa(f_0)=\varkappa(f_1)$, because the existence of a joint Seifert class (Lemma \ref{l:jhss}) requires this assumption.
If $\lambda(f_0)\ne\lambda(f_1)$, then $\overline{\lambda(f_0)}(H_3)\ne\overline{\lambda(f_1)}(H_3)$ is possible,
(but $i_{C_0,M_\varphi}\overline{\lambda(f_0)}(H_3)=i_{C_1,M_\varphi}\overline{\lambda(f_1)}(H_3)$ by (b) and (c)).
This is possible because $i_{C_0,M_\varphi}$ or $i_{C_0,M_\varphi}$ need not be injective when $\lambda(f_0)\ne\lambda(f_1)$.

(b) In applications of Lemma \ref{l:jhssn0} both subsets of $H_3(C_f;\Z_d)$ consist of one element
$\widehat{A_f}b_{\varphi_0,Y_0}=\widehat{A_f}b_{\varphi_1,Y_1}$, where $b_{\varphi,Y}$ is defined in Lemma \ref{l:webeta}.
\end{Remark}

\begin{Remark}[to \S\ref{s:lemmas-calc}]\label{r:tocalc}
The following result (not used in this paper) is proved analogously to Lemma \ref{l:Bet}.
{\it If $\lambda(f_0)=\lambda(f_1)$, \ $\varkappa(f_0)=\varkappa(f_1)$, \ $\varphi\colon \partial C_0\to \partial C_1$ is a bundle isomorphism such that $M_\varphi$ is spin and $Y\in H_5(M_\varphi)$ is a joint Seifert class, then
$\beta(f_0,f_1)=
[\widehat A_0^{-1}i_{C_0,M_\varphi}^{-1}\rho_{\di(\varkappa(f_0))}(Y^2-\frac14p_1^*(M_\varphi))]\in C_{\lambda(f_0),\varkappa(f_0)}$.}
\end{Remark}

}

\section{Proof of the MK Isotopy Classification Theorem \ref{l:isomk}}\label{s:eta-mk}

\subsection{The obstruction $\eta(\varphi,Y)$ and its properties} \label{s:etaobs}

Recall that some notation was introduced in \S\S \ref{s:intro-main}, \ref{s:intro-strat}, \ref{s:defandplan-not} and \ref{s:prelemmas-not}.

Denote $d_0:=\di\varkappa(f_0)$.
In what follows, a statement involving $k$ holds for both $k=0,1$.

A joint Seifert class $Y\in H_5(M_\varphi)$ is called a {\it $d$-class} for an integer $d$ if $\rho_dY^2=0$ (or, equivalently, $Y^2\in dH_3(M_\varphi)$).


\begin{Lemma} \label{l:fc}
If $\lambda(f_0)=\lambda(f_1)$,\ $\varkappa(f_0)=\varkappa(f_1)$,\ $\beta(f_0,f_1)=0$
and $\varphi\colon \de C_0\to\de C_1$ is a $\pi$-isomorphism, then there is a $d_0$-class for $\varphi$.
\end{Lemma}

Lemma \ref{l:fc} follows from Lemmas \ref{l:jhss} and \ref{l:desei}.bc together with the definition of $\beta(f_0,f_1)$.

For a 8-manifold $W$ we consider the intersection products
$$
\capM{\de} \colon H_4(W)\times H_4(W,\de)\to\Z\quad\text{and}\quad
\capM{\de\de}\colon H_6(W,\de)\times H_6(W,\de)\to H_4(W,\de).
$$
As before, for the corresponding squares $H_6(W,\de)\to H_4(W,\de)$ we do not put any subscript.

The following lemma is well-known.
Let $\nu_Q$ be the stable equivalence class of the normal bundle of a manifold $Q$.
A {\it stable normal spin structure} on $Q$ is a framing $\xi$ of $\nu_Q$ over the 2-skeleton of $Q$ compatible with the orientation of $Q$, considered up to the homotopy of the restriction of $\xi$ to the $1$-skeleton.
For brevity we omit `stable' (analogously to stable tangent spin structures).

\begin{Lemma}\label{l:krst}
If a manifold has a spin structure, then it has a normal spin structure.
\end{Lemma}

{\bf Definition of $\eta(\varphi,Y)$ for a $\pi$-isomorphism $\varphi\colon \de C_0\to\de C_1$ and a $d$-class
$Y\in H_5(M_\varphi)$.}
Since $\varphi\colon \de C_0\to\de C_1$ is a $\pi$-isomorphism, $M_\varphi$ is spin.
Take any normal spin structure on $M$ given by Lemma \ref{l:krst}.
Since $M_\varphi$ is simply-connected, a normal spin structure on $M_\varphi$ is unique.
Since $\Omega_7^{Spin}(\C P^\infty)=0$ \cite[Lemma 6.1]{KS91} there is a 8-manifold $W$ with a normal
spin structure and $z\in H_6(W,\partial)$ such that $\partial W\underset{spin}=M_\varphi$ and $\partial z=Y$.
Consider the following fragment of the exact sequence of the pair $(W,\de W)$:
$$H_4(\de W;\Z_d)\xra{~i_W~} H_4(W;\Z_d) \xra{~j_W~} H_4(W,\de;\Z_d)
\xra{~\partial_W~} H_3(\de W;\Z_d).$$
Since $\partial_W\rho_d z^2=\rho_dY^2=0$, there is a class $\overline{z^2}\in H_4(W;\Z_d)$ such that $j_W\overline{z^2}=\rho_dz^2$.
Define
$$\eta(\varphi,Y)=\eta(f_0,f_1,d,\varphi,Y):=
\rho_{\widehat d}(\overline{z^2}\capM{\de}(z^2-p_W^*))\in\Z_{\widehat d},
\quad\text{where}\quad\widehat d:=\gcd(d,24).$$


\begin{proof}[Proof that $\eta(\varphi,Y)$ is well-defined, i.e.~independent of the choice of $W,z$ and
$\overline{z^2}$]
The proof is analogous to \cite[2.3 and footnote (q)]{CS11}.
For the independence from $\overline{z^2}$ instead of \cite[Lemma 2.7]{CS11} we use $\de_Wp_W^*=p_{M_\varphi}^*=0$.
For the independence from $W,z$ instead of the uniqueness of $\de_Wz$ of \cite[Lemma 2.6]{CS11} we use that
$\de_Wz=Y$ is fixed.
\end{proof}

\begin{Lemma}[proved in \S\ref{s:varphin0}]\label{l:etaeven}
Let $\varphi\colon \de C_0\to \de C_1$ be a $\pi$-isomorphism and $Y\in H_5(M_\varphi)$ a $d_0$-class.

(a) (Divisibility of $\eta$ by 2) The residue $\eta(\varphi,Y)\in\Z_{\widehat{d_0}}$ is even.

(b) (Change of $\eta$)
There is an embedding $g\colon S^4\to S^7$, a $\pi$-isomorphism $\varphi'\colon \de C_0\to \de C_{f_1\#g}$ and a $d_0$-class
$Y'\in H_5(M_{\varphi'})$ such that $\eta(\varphi',Y',f_0,f_1\#g,d_0)=\eta(\varphi,Y,f_0,f_1,d_0)+2$.

(c)(Change of $\varphi$)
For every $\pi$-isomorphism $\varphi'\colon \de C_0\to \de C_1$ there is a $d_0$-class $Y'\in H_5(M_{\varphi'})$ such that $\eta(\varphi',Y')=\eta(\varphi,Y)$.
\end{Lemma}

Note that Lemma \ref{l:etaeven}.a is trivial for $\widehat{d_0}$ odd.

Other properties of $\eta$ which are not used in this paper will be discussed in \cite{II}.

\comment

\begin{Lemma}(\cite[Lemma 6.1]{KS91})\label{l:krst}
Assume that $M$ is a closed 7-manifold with a spin structure and $Y\in H_5(M)$.
Then there are an 8-manifold $W$ with a spin structure and
$z\in H_6(W,\partial)$ such that $\partial W\underset{spin}=M_\varphi$ and $\partial z=Y$.
\end{Lemma}

\begin{proof}
This follows from \cite[Lemma 6.1]{KS91} because normal spin structures on oriented manifolds are in natural 1--1 correspondence with tangent spin structures.
Recall well-known construction of this 1--1 correspondence.
If
$s_t$ and $s_n$ are tangent and normal spin structures on $M$,
then the Whitney sum $s_0 \oplus s_1$ defines a spin structure of a trivial bundle over $M$.
If $s$ is a tangent spin structure on $M$, then there is a unique (up to equivalence) normal spin structure $-s$
on $M$ such that $s \oplus (-s)$ is homotopic to the spin structure given by restricting the trivial bundle to the $2$-skeleton.
Then $s\mapsto -s$ defines the 1--1 correspondence.
\end{proof}

\endcomment

\subsection{Proof of Theorem \ref{l:isomk} using Theorem \ref{t:aldi} and
Lemmas \ref{l:etaeven}, \ref{l:welleta'}} \label{s:lemmas-isotmod}

In the definition of $\eta(\varphi,Y)$ in \S\ref{s:etaobs} instead of $C_0,C_1,\varphi,Y$ we can take
any pair of simply-connected parallelizable 7-manifolds $M_0,M_1$, a diffeomorphism
$\varphi\colon \de M_0\to\de M_1$ such that the manifold $M:=M_0\cup_\varphi(-M_1)$ is parallelizable and any class $Y\in H_6(M)$ such that $\rho_dY^2=0$.
Denote by $\eta(\varphi,Y)=\eta(M_0,M_1,d,\varphi,Y)\in\Z_{\widehat d}$ the obtained residue.
Also, in this situation for $d$ even we can define
$$\eta'(\varphi)=\eta'(M_0,M_1,d,\varphi):=\rho_2(\overline{z^2}\capM{\de}z^2)\in\Z_2.$$
The proof that $\eta'(\varphi)$ is independent of the choice of $W,z$ and $\overline{z^2}$
is analogous to the case of $\eta(\varphi,Y)$.

\begin{Lemma}\label{l:welleta'}
Assume that $d_0$ is even and $\varphi\colon \de C_0\to \de C_1$ is a $\pi$-isomorphism.

(a) (proved in \S\ref{s:welleta'})
The residue $\eta'(\varphi)$ is well-defined, i.e.~is independent of the choice of $Y$.

(b) (Change of $\eta'$; proved in \S\ref{s:varphin0})
There is a $\pi$-isomorphism $\varphi'\colon \de C_0\to \de C_1$ such that $\eta'(\varphi')=\eta'(\varphi)+1$.
\end{Lemma}

\begin{Theorem}[Almost Diffeomorphism Theorem; proved in \S\ref{s:lemmas-adtmod}]\label{t:aldi}
Let

$\bullet$ $M_0,M_1$ be oriented simply-connected 7-manifolds whose homology groups are free abelian
and such that $H_5(M_k,\de)\cong\Z$;

$\bullet$ $\varphi:\de M_0\to\de M_1$ be a diffeomorphism such that $M:=M_0\cup_\varphi(-M_1)$ is a
parallelizable oriented manifold for which $H_2(M),H_3(M)$ are free abelian and
$\t j_k:=j_{M_k,M} \colon H_4(M)\to H_4(M,M_k)$, $k=0,1$, are epimorphisms having the same kernel;

$\bullet$  $Y\in H_5(M)$ be a class such that $Y\cap M_k$ is a generator $\alpha_k\in H_5(M_k,\de)$,
\ $\di(Y^2)=\di(\alpha_0^2)=\di(\alpha_1^2)=:d$ and there is a class $Q\in H_4(\de M_0)$ such that $i_MQ\cap_M Y^2=d$.

For some homotopy 7-sphere $\Sigma$ there is an orientation-preserving diffeomorphism
$\overline\varphi\colon M_0\to M_1\#\Sigma$ extending $\varphi$ and such that $\overline\varphi\alpha_0=\alpha_1\#0$ if%
\footnote{We conjecture that the `only if' statement is also true.}
$$\eta(\varphi,Y)=0\quad\text{and, for $d$ even,}\quad \eta'(\varphi)=0.$$
\end{Theorem}

\begin{proof}[Proof of Theorem \ref{l:isomk} using the Almost Diffeomorphism Theorem \ref{t:aldi}]
Since $\lambda(f_0)=\lambda(f_1)$ and $\varkappa(f_0)=\varkappa(f_1)$, by Lemma \ref{l:piso} there is a $\pi$-isomorphism $\varphi\colon \de C_0\to\de C_1$.
Since $\beta(f_0,f_1)=0$, by Lemma \ref{l:fc} there is a $d_0$-class $Y\in H_5(M_\varphi)$.

By the divisibility of $\eta(\varphi, Y)$ by 2 (Lemma \ref{l:etaeven}.a) and by a change of $\eta$ (Lemma \ref{l:etaeven}.b)
we can change $f_1$ (by connected sum with a knot), $\varphi$ and $Y$, and assume that $\eta(\varphi,Y)=0$.

In this paragraph assume that $d_0$ is even.
Then by a change of $\eta'$ (Lemma \ref{l:welleta'}.b) we obtain a new $\pi$-isomorphism $\varphi$ such that $\eta'(\varphi)=0$.
By a change of $\varphi$ (Lemma \ref{l:etaeven}.c) we obtain a new $d_0$-class $Y$ such that $\eta(\varphi,Y)=0$.

Since $7=4+3$, by general position $C_k$ are simply-connected.
The groups $H_3(C_k)\overset{\widehat A_k}\cong H_1$ and $H_4(C_k,\de)\overset{A_k}\cong H_3$ are free abelian.
Hence by Lefschetz duality $H_2(C_k)$ is free abelian.
So all homology groups of $C_k$ are free abelian.
Since $\varkappa(f_0)=\varkappa(f_1)$ and $\lambda(f_0)=\lambda(f_1)$,

$\bullet$ by the exact sequence of the pair $(M_\varphi,C_k)$
and the Agreement Lemma \ref{l:Agr}.b for $q=2$ the group $H_2(M_\varphi)$ is free abelian;

$\bullet$ by the Agreement Lemma \ref{l:Agr}.b for $q=3$ the map $j_k\colon H_4(M_\varphi)\to H_4(M_\varphi,C_k)$ is onto.

Since $\widehat A_k\colon H_2\to H_4(C_k)$ is an isomorphism and
$i_{C_1,M_\varphi} \widehat A_1=i_{\de C_1,M_\varphi}\varphi\nus_0^!=i_{C_0,M_\varphi}\widehat A_0$, we have
$\im i_{C_0,M_\varphi}=\im i_{C_1,M_\varphi}$.
Hence $\ker j_{C_0,M_\varphi}=\ker j_{C_1,M_\varphi}$.
Then from the exact sequence of the pair $(M_\varphi, C_0)$ we obtain that $H_3(M_\varphi)$ and $H_2(M_\varphi)$ are free abelian.

By Alexander duality, $\alpha_k:=A_k[N]$ is a generator of $H_5(C_k,\de)$.
We have $d_0=\di(\varkappa(f_0))=\di(\alpha_k^2)$.
By Lemmas \ref{l:desei}.a and \ref{l:Agr}.c there is a class $Y\in H_5(M)$ such that $Y\cap C_k=\alpha_k$.
Since $Y^2$ is divisible by $d_0$ and $A_0\varkappa(f_0)=A_0(Y\cap C_0)^2$, we have $\di(Y^2)=d_0$.

Take $Q'\in H_2$ such that $Q'\cap_N\varkappa(f_0)=d_0$.
Let $Q:=\nus_0^!Q'\in H_4(\de C_0)$.
Then
$$i_{\de C_0,M_\varphi}Q\cap_{M_\varphi} Y^2=i_{C_0,M_\varphi}\widehat A_0Q'\cap_{M_\varphi} Y^2 =
\widehat A_0Q'\cap_{C_0}(Y\cap C_0)^2= Q'\cap_N \varkappa(f_0)=d_0.$$
Hence by the Almost Diffeomorphism Theorem \ref{t:aldi} for $M_0=C_0$, $M_1=C_1$ and $M=M_\varphi$
there is an orientation preserving diffeomorphism $\overline\varphi\colon C_0\to C_1\#\Sigma$ extending the bundle isomorphism $\varphi$.
The bundle isomorphism $\varphi$ also extends to an orientation-preserving diffeomorphism
$S^7-\Int C_0\to S^7-\Int C_1$.
So $S^7\cong S^7\#\Sigma\cong\Sigma$.
Then $f_0$ and $f_1$ are isotopic by Lemma \ref{l:isored}.
\end{proof}

\subsection{Proofs of Lemmas \ref{l:etaeven} and \ref{l:welleta'}.b}\label{s:varphin0}

\begin{proof}[Proof of Lemma \ref{l:etaeven}.a]
Take any pair $(W, z)$ from the definition of $\eta(\varphi,Y)$ and
take some map $Z\colon W\to \C P^\infty$ corresponding to $z\in H_6(W,\de)\cong H^2(W)\cong [W,\C P^\infty]$.
By spin surgery of $Z$ relative to $\de W$ we may assume that $Z$ is 3-connected.
The residue $\overline{z^2}\capM{\de}(z^2-p_W^*)$ does not change throughout this surgery because it is
`spin $\C P^\infty$-characteristic residue modulo $d$ relative to the boundary'.
Since $Z$ is 3-connected, by the Hurewicz Theorem for the mapping cylinder of $Z$
we have $H_3(W)=\pi_3(W)=\pi_3(\C P^\infty)=0$.
Hence $\Tors H_4(W)\cong\Tors H_3(W)=0$.
So there is a class $\widehat{z^2}\in H_4(W)$ such that $\rho_{d_0}\widehat{z^2}=\overline{z^2}$.
Then
$$\overline{z^2}\capM{\de}(z^2-p_W^*)=\rho_{d_0}(\widehat{z^2}\capM{\de} z^2-\widehat{z^2}\capM{\de} p_W^*) =\rho_{d_0}(\widehat{z^2}\capM{\de}\widehat{z^2}-\widehat{z^2}\capM{\de} p_W^*).$$
The latter residue is divisible by 2 by \cite[Lemma 2.11]{CS11}.
\end{proof}

Lemma \ref{l:etaeven}.b is proved analogously to \cite[\S3, the second equality of Addendum 1.3]{CS11}.

For the proofs of Lemmas \ref{l:etaeven}.c and \ref{l:welleta'}.b we need the following result.

\begin{Lemma}[proved below in \S\ref{s:varphin0}]\label{l:cobordeta}
Assume that $\pi$-isomorphisms $\varphi',\varphi\colon \de C_0\to\de C_1$ coincide over $N_0$ and that
over $\Cl(N-N_0)$ they differ by the generator of   $\pi_4(SO_3)\cong\Z_2$.
Then for every integer $d$ and $d$-class $Y\in H_5(M_\varphi)$ there is a $d$-class $Y'\in H_5(M_{\varphi'})$
such that the pair
$$(M_{\varphi'},Y')\quad\text{is cobordant to}\quad(M_\varphi,Y)\sqcup(S^2\tilde\times S^5, A),$$
where $S^2\tilde\times S^5$ is the total space of the non-trivial $S^2$-bundle over $S^5$ (i.e.~the bundle corresponding to the non-trivial element of $\pi_4(SO_3)\cong\Z_2$) and
$A\in H_5(S^2\tilde\times S^5)\cong\Z$ is a generator.
\end{Lemma}

\begin{proof}[Proof of Lemma \ref{l:etaeven}.c]
By Lemma \ref{l:piso} we may assume that $\varphi'=\varphi$ over $N_0$.
We may also assume that over $\Cl(N-N_0)$ isomorphism $\varphi'$ obtained from $\varphi$ by twisting with   $d(\varphi',\varphi)\in\pi_4(SO_3)\cong\Z_2$.
If $d(\varphi',\varphi)=0$, then we may assume that   $\varphi'=\varphi$ and take $Y'=Y$.
If $d(\varphi',\varphi)\ne0$, then by Lemma \ref{l:cobordeta} and a calculation in
\cite[Proof of the Framing Theorem 2.9]{CS11} $\eta(\varphi',Y')=\eta(\varphi,Y)$.
\end{proof}

\begin{proof}[Proof of Lemma \ref{l:welleta'}.b]
We do not assume Lemma \ref{l:welleta'}.a.
and so we write $\eta'(\varphi,Y)$ instead of $\eta'(\varphi)$ and prove the lemma in the following form.

{\it For every $\pi$-isomorphism $\varphi\colon \de C_0\to \de C_1$, even integer $d$ and $d$-class $Y\in H_5(M_\varphi)$ there is a $\pi$-isomorphism $\varphi'\colon \de C_0\to \de C_1$
and a $d$-class $Y'\in H_5(M_{\varphi'})$ such that $\eta'(\varphi',Y')=\eta'(\varphi,Y)+1$.}

Take a bundle isomorphism $\varphi'\colon \de C_0\to\de C_1$ coinciding with $\varphi$ over $N_0$ and over $\Cl(N-N_0)$ obtained from $\varphi$ by twisting with the non-trivial element of $\pi_4(SO_3)\cong\Z_2$.
Then by the String Lemma \ref{l:frbuis} $\varphi'$ is a $\pi$-isomorphism.
By Lemma \ref{l:cobordeta} and a calculation in \cite[Proof of the Framing Theorem 2.9]{CS11} $\eta'(\varphi',Y')=\eta'(\varphi,Y)+1$.
\end{proof}

\begin{proof}[Proof of Lemma \ref{l:cobordeta}]
Take a smooth map $\alpha\colon S^4\to SO_3$ representing the non-trivial element of $\pi_4(SO_3)\cong\Z_2$ and such that $\alpha|_{D^4_-}=\id S^2$.
For $k=0,1$ identify
$$
\nus_k^{-1}\Cl(N-N_0)\times\left[\frac13,\frac23\right]\quad\text{with}\quad
\Sigma_k:=S^2\times D^4\times\left[\frac13,\frac23\right] \quad(\text{so }\Sigma_0=\Sigma_1).
$$
Let $U_k:=\de C_k\times I-\Int\Sigma_k$.
Define
$$
\overline\alpha\colon U_0\to U_1\quad\text{by}\quad
\overline\alpha(s,t)\colon =\begin{cases}
(\varphi(s),t) & s\in\nus_0^{-1}(N-N_0),\ t\in [\frac23,1]\\
(\varphi'(s),t) & \text{otherwise}\end{cases}
$$
$$
\text{and}\quad V:=C_0\times I\bigcup\limits_{\overline\alpha}\ C_1\times I.
$$
Hence $V$ is a cobordism between $M_{\varphi'}$ and $M_\varphi\sqcup E_\alpha$, where
$$
\widehat\alpha\colon \de\Sigma_0\to\de\Sigma_1\quad\text{maps $(a,b,t)$ to}\quad
\begin{cases}(a,b,t) &t<\frac23 \\
(\alpha(b)a,b,\frac23)& t=\frac23\end{cases}
$$
$$
\text{and}\quad E_\alpha:=\Sigma_0\cup_{\widehat\alpha}\Sigma_1 \overset{(1)}\cong
\frac{S^2\times D^5}{\{(s,b)\sim(\alpha(b)s,R (b))\}_{(s,b)\in S^2\times D^4_+}}\overset{(2)}\cong
S^2\tilde\times S^5.
$$
Here $R :D^4_+\to D^4_-$ is the reflection with respect to $0\times\R^4$.

Consider the following commutative diagram:
$$
\xymatrix{ H^q(V,C_0\times I) \ar[r]^{\ex\cong} \ar[d]^i & H^q(C_1\times I,U_1) \ar[d]&  \\
H^q(M_\varphi,C_0) \ar[r]^{\ex\cong} & H^q(C_1,\de) \ar[rr]^(.4){i_{C_1,C_1\times I}\cong} & &
H^q(C_1\times I,\de C_1\times I) \ar[llu]_{i_{C_1\times I}}}
$$
We have
$$
H^q(\de C_1\times I,U_1)\overset\ex\cong H^q(\Sigma_1,\de)\cong H_{7-q}(\Sigma_1)=0\quad\text{for}\quad q=1,2,3,4.
$$
Hence from the exact sequence of the triple $U_1\subset \de C_1\times I\subset C_1\times I$ we see that
$i_{C_1\times I}$ is injective for $q=2,3,4,5$.
Hence $i$ is
an isomorphism for $q=2,3,4,5$.
Look at the inclusion-induced mapping of the exact sequences of pairs
$(M_\varphi,C_0)$ and $(V,C_0\times I)$.
By the 5-lemma we see that the inclusion $M_\varphi\to V$ induces an isomorphism in $H^q(\cdot)$ for $q=2,4$.
Or, in Poincar\'e dual form, $r_{M_\varphi}\de_V\colon H_q(V,\de)\to H_{q-1}(M_\varphi)$ is an isomorphism for $q=4,6$.
The same holds for $\varphi$ replaced by $\varphi'$.

Let $\overline Y:=(r_{M_\varphi}\de_V)^{-1}Y\in H_6(V,\de)$.
Then by Lemma \ref{l:desei}.a
$$
\overline Y\cap (C_k\times I)=A_k[N]\times I\in H_6(C_k\times I,\partial)\quad
\text{for each $k = 0, 1$.}
$$
So by Lemma \ref{l:desei}.a $Y':=r_{M_{\varphi'}}\de_V\overline Y\in H_5(M_{\varphi'})$ is a joint Seifert class.
We have the equation
$\rho_d(Y')^2=\rho_dr_{M_{\varphi'}}\de_V (r_{M_\varphi}\de_V)^{-1}Y^2=0$, i.e.~$Y'$ is a $d$-class.
Let $Y_\alpha\colon =\partial\overline Y\cap E_\alpha\in H_5(E_\alpha)$.
Since
$$\overline Y\cap (C_0\times I)=A_0[N]\times I,\quad\text{we have}\quad
Y_\alpha\cap\Sigma_1=\left[*\times D^4\times\left[\frac13,\frac23\right]\right]\in H_5(\Sigma_1,\partial).$$
Hence under (1) $Y_\alpha$ goes to a class whose intersection with $[S^2\times 0]$ is $+1$.
Therefore under the composition of (1) and (2) $Y_\alpha$ goes to
a class whose intersection with the fiber $S^2$ is $+1$, i.e.~to $A$.

Therefore $(V,\overline Y)$ is the required cobordism.
\end{proof}

\subsection{Proof of Lemma \ref{l:welleta'}.a}\label{s:welleta'}

{\bf Definition of $M_f$ and $Y_{f,y}$.}
Identify $C_f$ and $C_f\times0$.
Denote
$$
M_f:=\partial(C_f\times I)=M_{\id \de C_f}\quad
\text{and, \quad for }\quad y\in H_3,\quad
Y_{f,y}:=\partial(A_f[N]\times I)+i_{C_f,M_f}\widehat A_fy\in H_5(M_f).
$$

\begin{Lemma}[Description of $d$-classes for $M_f$] \label{l:simpleY0}
A class $Y\in H_5(M_f)$ is a $d$-class if and only if $Y=Y_{f,y}$ for some
$y\in \ker(2\rho_{\di(\varkappa(f))}\overline{\lambda(f)})$.
\end{Lemma}

This follows by Lemma \ref{l:desei}.b,c.

\begin{Lemma}[proved below in \S\ref{s:welleta'}]\label{l:prdiffbor}
For every $y\in H_3$ there is a spin null-bordism $(W,z)$ of $(M_f,Y_{f,y})$ such that $p_W^*$ is even.%
\footnote{We cannot take $W=C_f\times I$ because $\de H_5(C_f\times I,\de)\not\ni Y_{f,y}$.
So we note that the following equality holds
$\de(A_f[N]\times I+\widehat A_fy\times I)=Y_{f,y}+\widehat A_fy\times 1\ne Y_{f,y}$
and `surger out' $\widehat A_fy\times1$ shifted into the interior.}
\end{Lemma}

\begin{proof}[Proof of Lemma \ref{l:welleta'}.a]
Before we prove that $\eta'(\varphi)$ is independent of $Y$  we denote it by $\eta'(\varphi,Y)$.
Take any pair of $d_0$-classes $Y',Y''\in H_5(M_\varphi)$.
We have
$$\eta'(\varphi,Y')-\eta'(\varphi,Y'')\overset{(1)}= \eta'(\id\de C_{f_0},Y)\overset{(2)}=
\eta'(\id\de C_{f_0},Y_{f_0,y})\overset{(3)}=0\in\Z_2,$$
where

$\bullet$ equality (2) holds for some $y\in \ker(2\rho_{\di(\varkappa(f_0))}\overline{\lambda(f_0)})$
by the description of $d$-classes for $M_f$ (Lemma \ref{l:simpleY0});

$\bullet$ equality (3) holds by Lemmas \ref{l:etaeven}.a and \ref{l:prdiffbor};

$\bullet$ equality (1) holds for some $d_0$-class $Y\in H_5(M_{f_0})$ by the following result.

{\it Let $f_0,f_1,f_2 \colon N\to S^7$  be embeddings, $\varphi_{01}\colon \de C_0\to\de C_1$ and $\varphi_{12}\colon \de C_1\to\de C_2$
$\pi$-isomorphisms, $Y_{01}\in H_5(M_{\varphi_{01}})$ and $Y_{12}\in H_5(M_{\varphi_{12}})$ \ $d$-classes.
Then $\varphi_{02}\colon =\varphi_{12}\varphi_{01}$ is a $\pi$-isomorphism and there is a $d$-class
$Y_{02}\in H_5(M_{\varphi_{02}})$ such that $\eta'(\varphi_{02},Y_{02})=\eta'(\varphi_{01},Y_{01})+\eta'(\varphi_{02},Y_{12})$.}

This result is proved analogously to \cite[Lemma 2.10]{CS11}, cf.~\cite[\S2, Additivity Lemma]{Sk08'} (the property that $Y_{02}$ is a $d$-class is achieved analogously to the proof of Lemma \ref{l:cobordeta}).
\end{proof}


{\bf Definition of a simplifying 6-bordism $V$ and maps $v=v_0,v_1,v_2,v_3$.}
A {\it simplifying 6-bordism} for $f$ and an oriented 3-submanifold $P\subset N$ is a 6-manifold $V\subset C_f$
with boundary $\de V=\nus_f^{-1}P\sqcup v(S^2\times S^3)$ for some embedding $v=v_0\colon S^2\times S^3\to \Int C_f$
such that $V\cap\de C_f=\nus_f^{-1}P$ and $v(S^2\times1_3)$ is homologous to $S^2_f$ in $C_f$.
(Then $[\im v]=\widehat A_f[P]\in H_5(C_f)$.)

E.g.~for $N=S^1\times S^3$ and $P=1_1\times S^3$ we can take a simplifying 6-bordism
$S^2\times S^3\times I\cong V\subset C_f$.

Let $v_1\colon S^2\times S^3\times D^1\to\Int C_f$ be an embedding such that $v_1|_{S^2\times S^3\times1}=v$,
\ $\im v_1\cap V=\im v$ and $v_1(c\times D^1)$ is tangent to $V$ for every $c\in S^2\times S^3$.

Extend $v_1$ to an orientation-preserving embedding $v_2 \colon S^2\times S^3\times D^2\to\Int C_f=\Int C_f\times\frac12$   transversal to $V$ and such that $\im v_2\cap V=\im v$.

Extend $v_2$ to an orientation-preserving embedding $v_3\colon S^2\times S^3\times D^3\to \Int(C_f\times I)$.

\begin{Lemma}\label{l:real}
For every oriented 3-submanifold $P\subset N$ there is a simplifying 6-bordism.
\end{Lemma}

\begin{proof}
Equip $\nus_f^{-1}P$ with the spin structure induced from $C_f$.
(This spin structure is compatible with the orientation of $\nus_f^{-1}P$.)
Since $C_f$ is simply connected, we can perform spin surgeries on 1-spheres in
the $5$-manifold $\nus_f^{-1}P$ to obtain a spin bordism between the inclusion
$\nus_f^{-1}P\to C_f$ and a map $\mu\colon X\to C_f$ of some closed simply connected 5-manifold $X$.
Since the induced map $i_{C_f}\colon H_2(\nus_f^{-1}P) \to H_2(C_f) \cong \Z$ is surjective,
$\mu \colon H_2(X) \to H_2(C_f)$ is surjective.
By Smale's classification of simply connected spin $5$-manifolds \cite[Theorem A]{Sm62} (see also \cite[Theorem 4.1]{Cr11}),
there is a closed simply connected spin $5$-manifold $X'$ with $H_2(X')=\ker\mu$.
Choose any isomorphism $H_2(X)\to\ker\mu \oplus \Z$.
Then by Barden's classification of simply connected $5$-manifolds
\cite[Theorem 2.2]{Ba65} (see also \cite[Theorem 5.1]{Cr11}), we may identify $X$ with $(S^2\times S^3)\#X'$
so that $\mu H_2(X')=\{0\}\subset H_2(C_f)$.
Also by Smale's classification \cite[Theorem 1.1]{Sm62}, $X'$ is spin diffeomorphic to the boundary of
a handlebody obtained by attaching $3$-handles to $D^6$ (for some spin structures on these manifolds).
So the co-cores of these handles give framed embeddings of $2$-spheres such that spin surgery on these 2-spheres
gives $S^5$.
Applying this to 2-spheres in $X'$ and using $\mu H_2(X')=0\in H_2(C_f)$ we obtain a spin bordism over $C_f$,
$g\colon V\to C_f$, between $\mu$ and a map $S^2\times S^3\to C_f$ inducing an isomorphism on $H_2$.
Then

$\bullet$ $V$ is a spin 6-manifold obtained from $\nus_f^{-1}P\times I$ by attaching 2-handles $D^2 \times D^4$ and 3-handles $D^3 \times D^3$ to $\nus_f^{-1}P\times 1$;

$\bullet$ $\partial V\underset{spin}=\nus_f^{-1}P\times 0\sqcup S^2 \times S^3$;

$\bullet$ $g|_{\nus_f^{-1}P\times0}$ is the identity and $g|_{S^2 \times S^3}$ induces an isomorphism on $H_2$.

Now the lemma follows by (the second part of) the following Semiproper Embedding Theorem \ref{l:emb}.b for $\de_+V:=\nus_f^{-1}P$ and $\de_-V:=S^2\times S^3$.
\end{proof}

\begin{Theorem}[Semiproper Embedding] \label{l:emb}
Let $V$ and $X$ be $v$- and $x$-manifolds such that $\de V=\de_+V\sqcup \de_-V$.
Then every map $g\colon V\to X$ such that $g|_{\de_+V}$ is an embedding into $\de X$ is homotopic $\rel\de_+V$
to an embedding $V\to X$, provided either

(a) $x<2v$ and $(V,\de_-V)$ is $(2v-x-1)$-connected, or

(b) $x=7=v+1$ and $(V,\de_-V)$ is 2-connected, $X$ and $V$ have spin structures $s_X$ and $s_V$ such that  $s_V|_{\de_+V}=g^* s_X|_{\de X}$.
\end{Theorem}

\begin{proof}
First assume that $(V,\de_-V)$ is $(2v-x-1)$-connected.
Then there is a handle decomposition of $V$ relative to $\de_+V$ without handles of index more than
$v-(2v-x-1)-1=x-v$.
In particular, $V$ is a regular neighborhood (in itself) of an $(x-v)$-polyhedron.

Use induction on the number of handles.
The base case is $V=\de_+V\times I$  (when there are no handles).
Then define an embedding $V\to X$ as
$\de_+V\times I\overset{g|_{\de_+V}\times \id I}\to \de X\times I\overset{i_X}\to X$, where $i_X$ is the collar inclusion.

Let us prove the inductive step for (a).
We may assume that $V':=V\cup D^k\times D^{v-k}$, \ $k\le x-v$, \ $g\colon V'\to X$ is a map such that
$g|_{\de_+V}$ is an embedding into $\de X$ and $g|_V$ is an embedding and $g^*s_X|_V=s_V$.
Since $x<2v$, we have $x\ge2(x-v)+1\ge 2k+1$.
Since $V$ is a regular neighborhood (in itself) of an $(x-v)$-polyhedron, we may assume that
$g|_{D^k\times0}$ is an embedding and $g(D^k\times0)\cap g(V)=g(\de D^k\times0)$.
Since $k\le x-v$, we have $\pi_{k-1}(V_{x-k,v-k})=0$.
Hence the normal $(v-k)$-framing of $g(\de D^k\times0)$ in $g(\de_-V)$ extends to a normal $v-k$ framing of
$g(D^k)$ in $X$.
Thus $g|_{D^k\times0}$ extends to an embedding $D^k\times D^{v-k}\to X$ whose image intersects $g(V)$ at
$g(\de D^k\times D^{v-k})$.
This extension defines an embedding $V'\to X$ extending $g|_V$ and homotopic to $g$.

Now we prove (b).
By hypothesis, there is a handle decomposition of $V$ relative to $\de_+V$ without handles of index more than $6-2-1=3$.
The proof is the same as the proof of (a) above except that $x\ge2k+1$ is verified directly and $k>x-v=1$ is possible.
The required extension of the $(v{-}k)$-framing exists

$\bullet$ for $k=3$ because $\pi_{k-1}(V_{x-k,v-k})=\pi_2(\SO_4)=0$;

$\bullet$ for $k=2$ because $g^*s_X|_V=s_V$ (in spite of $\pi_{k-1}(V_{x-k,v-k})=\pi_1(\SO_5)\ne0$).
\end{proof}

\begin{proof}[Proof of Lemma \ref{l:prdiffbor}]
Take any $y\in H_3$.
Since $H_3 \cong H^1(N) \cong [N, S^1]$, the class $y$ is represented by an oriented 3-submanifold $P\subset N$
that is  the preimage of a regular value of a map $N \to S^1$ representing $y$;
orientations on $N$ and $S^1$ give an orientation on the preimage.
Take a simplifying 6-bordism $V\subset C_f$ given by Lemma \ref{l:real}.
Take the corresponding maps $v,v_1,v_2,v_3$.
Let
$$
W_-:=(C_f\times I)-\Int\im v_3\quad\text{and}\quad W:=
W_-\cup_{v_3|_{S^2\times S^3\times S^2}}(S^2\times D^4\times S^2).
$$
(The manifold $W$ may be called the result of an $S^2$-parametric surgery along $v_3$.)

Denote
$$t:=v_3(S^2\times 0\times S^2)\quad\text{and}\quad \Delta:=1_2\times D^4\times1_1.$$
Identify $S^2\times D^4\times S^2$ with $t\times\Delta$.

Consider the cohomology exact sequence of the pair $(W,W_-)$ in the following Poincar\'e dual form:
$$(*)\qquad\xymatrix{
H_6(t\times\Delta) \ar[r]  & H_6(W,\de) \ar[r]^{r_{W_-}} & H_6(W_-,\de) \ar[r] & H_5(t\times\Delta) \\
H^2(W,W_-)\ar[u]_{\cong}^{PD\circ\ex}  & & & H^3(W,W_-)\ar[u]_{\cong}^{PD\circ\ex}
}$$
Since $H_5(t\times\Delta)=0$, the map $r_{W_-}$ is an epimorphism.
Take any
$$Z\in r_{W_-}^{-1}(A_f[N]\times I\cap W_-)\subset H_6(W,\partial).$$
Denote
$$\widehat V:=V\cup (S^2\times D^4\times1)\subset W
\quad\text{and}\quad z:=Z+[\widehat V]\in H_6(W,\partial).$$
Objects constructed above depend upon $y,f$ and the choices in the construction.
We do not indicate this in their notation.

Since $H_6(t\times\Delta)=0$, the spin structure on $W_-$ coming from $S^7\times I$ extends to $W$.
Clearly, $\partial W\underset{spin}=\partial(C_f\times I)\underset{spin}=M_f$
(for the `boundary' spin structure on $M_f$ coming from $C_f\times I$).
Since
$$\de_WZ=\de_{C_f\times I}(A_f[N]\times I)=Y_{f,0}\quad\text{and}\quad
\de_W[\widehat V]=[\nus_f^{-1}P\times\frac12]=i_{M_f}\widehat A_fy,\quad\text{we have}\quad \de_W z=Y_{f,y}.$$

Consider the first line of diagram (*)
with subscripts 6,5 changed to 4,3, respectively.
Since $p_W^*\cap W_-=p_{W_-}^*=0$, by exactness $p_W^*=n[t]$ for some $n\in\Z$.
Denote
$$W'_+:=(S^7-\Int\im v_2)\cup_{v_2|_{S^2\times S^3\times S^1}}S^2\times D^4\times S^1.$$
Then
$$n=n[t]\cap_{t\times\Delta}[\Delta]=
(p_W^*\cap t\times\Delta)\cap_{t\times\Delta}[\Delta]\overset{(3)}=
(p_{W'_+}^*\cap S^2\times D^4\times S^1)\cap_{S^2\times D^4\times S^1}[\Delta]\equiv0\bmod2.$$
Here

$\bullet$ the homology classes $[t]$ and $[\Delta]$ are taken in the space indicated under `$\cap$'
(so $[\Delta]$ has different meanings in different parts of the formula);

$\bullet$ the equality (3) holds because
$r_{S^2\times D^4\times S^1}\colon H_4(t\times\Delta,\de)\to H_3(S^2\times D^4\times S^1,\de)$ is an isomorphism;

$\bullet$ the congruence holds because $H_5(S^2\times D^4\times S^1)=0$, so the spin structure on $S^7-\Int\im v_2$ coming from $S^7$ extends to $W'_+$, hence by Lemma \ref{l:sp} $p_{W'_+}^*$ is even.
\end{proof}

\subsection{Proof of Theorem \ref{t:aldi} using Lemmas \ref{l:modif} and \ref{l:easur}}\label{s:lemmas-adtmod}


{\bf Definition of an elementary pair.}
Suppose that $U,V_0$ and $V_1$ are abelian groups and that $\capM{01}\colon V_0\times V_1\to\Z$ a unimodular pairing.
(Then $V_k$ has to be free abelian.)
An {\it elementary} pair is a pair $v_k\colon U\to V_k$, $k=0,1$, of monomorphisms such that $v_0U\capM{01} v_1U=0$ and
$v_kU$ is a half-rank direct summand in $V_k$ for each $k=0,1$.
(Then $\rk V_k$ has to be even.)


\smallskip
The following theorem is an easy corollary of a theorem of Kreck.
In it and in \S\ref{s:lemmas-ele} we consider the intersection product
$$
\capM{01}\colon H_4(W,M_0)\times H_4(W,M_1)\to\Z.
$$

\begin{Theorem}[Modified surgery theorem]\label{mst}
For $l \geq 2$ let

$\bullet$ $M_0,M_1\subset\R^{8l}$ be $(4l{-}1)$-manifolds with common boundary;

$\bullet$ $p\colon B\to BO$ be a fibration such that $\pi_1(B)=0$ and $\pi_i(p)=0$ for every $i\ge 2l$;

$\bullet$ $\overline{S\nu_k}\colon M_k\to B$, $k=0,1$, be $(2l{-}1)$-connected maps coinciding on the boundary and such that $p\overline{S\nu_k}$ is the classifying map of the normal bundle of $M_k$.

A diffeomorphism $M_0\to M_1$
commuting with $\overline{S\nu_k}$ and identical on $\de M_0$ exists if
there is

$\bullet$ a $4l$-manifold $W$ such that $\partial W=M_0\cup(-M_1)$,

$\bullet$ a $2l$-connected
map $\overline{S\nu}\colon W\to B$ extending $\overline{S\nu_0}\cup \overline{S\nu_1}$,

$\bullet$ a subgroup $U\subset \ker\overline{S\nu}\subset H_{2l}(W)$ such that the pair $j_{M_k,W}|_U$, $k=0,1$, is elementary.
\end{Theorem}

\begin{proof}
For an elementary pair $v_k\colon U\to V_k$, $k=0,1$, the quotient $v_0U\times V_1/v_1U\to\Z$ of $\capM{01}$ is unimodular.
So by \cite[the Kreck Theorem 4.1]{CS11}, cf.~\cite[Theorem 4]{Kr99}, $\overline{S\nu}$ is bordant
(relative to the boundary) to an $h$-cobordism.
Hence the theorem holds by the relative $h$-cobordism theorem \cite{Mi65}.
\end{proof}

{\bf Definitions of $i_W,j_W,\de_W$, convenient manifold and pre-elementary class.}
Let $W$ be an 8-manifold.

Denote by $i_W,j_W,\de_W$ the homomorphisms from the exact sequence of the pair $(W,\de W)$.

The manifold $W$ is called {\it convenient} if $H_3(\de W)$ is free abelian,
$H_5(W,\de)=H_3(W)=0$ and $\de W$ is parallelizable.

A class $z\in H_6(W,\de)$ is called {\it pre-elementary} if there is a homomorphism $s\colon H_4(W,\de)\to H_4(W)$ such that

(1) $H_4(W)=\im i_W\oplus\im s$,

(2) $su\cap_W sv=su\capM{\de} v$ for every $u,v\in H_4(W,\de)$, and

(3) $\sigma(W)=sp_W^*\capM{\de} p_W^*=sz^2\capM{\de} z^2=sz^2\capM{\de} p_W^*=0$.

\begin{Lemma}[Pre-elementary class; proved in \S\ref{s:lemmas-mod}]\label{l:modif}
Let $W$ be a convenient 8-manifold
and $z\in H_6(W,\de)$ a class such that for $d:=\di(\de_Wz^2)$ and some $\overline{z^2}\in H_6(W,\Z_d)$
$$j_W\overline{z^2}=\rho_dz^2,\quad\overline{z^2}\capM{\de}(z^2-p_W^*)\equiv0\bmod\widehat d \quad
\text{and, if $d$ is even,}\quad\overline{z^2}\capM{\de} z^2\equiv 0\bmod2.$$
Then there is a spin 8-manifold $W'$ such that $\de W'$ is a homotopy 7-sphere and
$z\sharp0\in H_6(W\sharp W',\de)$ is pre-elementary.
\end{Lemma}

\begin{Lemma}[Elementary pair; proved in \S\ref{s:lemmas-ele}]\label{l:easur}
Let $W$ be a convenient 8-manifold such that

(*) $\de W=M_0\cup_{\de M_0=\de M_1}(-M_1)$ for some 7-manifolds $M_0,M_1$
without torsion in their homology and having a common boundary, and

(**) $j_{M_k,\de W}\colon H_4(\de W)\to H_4(\de W,M_k)$, $k=0,1$, are epimorphisms having the same kernel.

Let $z\in H_6(W,\de)$ be a pre-elementary class for which there is a class
$q\in\ker j_{M_0,\de W}$ such that $q\cap_{\de W}\de_Wz^2=\di(\de_Wz^2)$.

Then
there is a subgroup $U\subset H_4(W)$ such that $U\capM{\de} z^2=U\capM{\de} p_W^*=0$ and
the pair of homomorphisms $j_{M_k,W}|_U$, $k=0,1$, is elementary.
\end{Lemma}

We remark that the proofs of the Elementary pair Lemma \ref{l:easur} modulo Lemma \ref{l:sinter}
(found in \S\ref{s:lemmas-ele}) and of the Pre-elementary class Lemma \ref{l:modif} modulo Lemma \ref{l:modules} (found in \S\ref{s:lemmas-mod})
are similar to \cite[Proof of Bordism Theorem 4.3 and of Lemma 4.5]{CS11}.
However, these proofs are different in details from those in \cite{CS11}, even when $H_1=0$.

\begin{proof}[Proof of the Almost Diffeomorphism Theorem \ref{t:aldi} modulo Lemmas \ref{l:modif} and \ref{l:easur}]
Take the spin structure on $M$ corresponding to a tangent framing on $M$.
Take any normal spin structure on $M$ given by Lemma \ref{l:krst}.
Since $\Omega_7^{Spin}(\C P^\infty)=0$ \cite[Lemma 6.1]{KS91} there is an 8-manifold $W$ with a normal spin structure and $z\in H_6(W,\partial)$ such that $\partial W\underset{spin}=M$ and $\partial z=Y$.

Recall that $B\Spin=BO\!\left<4\right>$ is the (unique up to homotopy) 3-connected
space for which there exists a fibration $\gamma\colon B\Spin\to BO$ inducing an isomorphism on $\pi_i$ for every $i\ge4$.
Let $B := B\Spin\times\C P^\infty$, $p:=\gamma\pr_2$ and
$\overline{S\nu}\colon W\to B$ be the map corresponding to the given normal spin structure on $W$ and to
$z\in H_6(W,\partial)\cong H^2(W)\cong [W,\C P^\infty]$.

For each $k=0,1$ since $M_k$ is torsion free, $H_2(M_k)\cong H_5(M_k,\de)\cong\Z$.
Then the homomorphism $\overline{S\nu}|_{M_k}\colon H_2(M_k)\to H_2(\C P^\infty)$ is an isomorphism.
This and the fact that $\pi_1(M_k)=0$ imply that the map $\overline{S\nu}|_{M_k}$ is 3-connected.

Performing $B$-surgery below the middle dimension we can change $\overline{S\nu}$ relative to the boundary and
assume that $\overline{S\nu}$ is 4-connected \cite[Proposition 4]{Kr99}.
Then
$$
H_5(W,\partial)\cong H^3(W)\cong H^3(B)=0,\quad H_3(W)\cong H_3(B)=0\quad\text{and}\quad H_2(W)\cong H_2(B)\cong\Z.
$$
From Poincar\'e-Lefschetz duality it follows that $\im j_W$ is a direct summand in  $H_4(W,\de)$.
The manifold $W$ is now convenient.  By the Pre-elementary class Lemma \ref{l:modif}
we can change $W$ to obtain a new manifold, again denoted $W$, with $\partial W = M_0 \# \Sigma$
and $z \in H_6(W, \partial)$ elementary.



Let $q:=i_MQ$.
Take a subgroup $U$ given by the Elementary pair Lemma \ref{l:easur}.
Recall that there is an isomorphism $H_4(B)\to \Z\oplus\Z$ mapping
$\overline{S\nu}(x)$ to $(x\capM{\de} z^2,x\capM{\de} p_W^*)$ for every $x\in H_4(W)$.
Then $U\subset \ker\overline{S\nu}$.
Apply the Modified Surgery Theorem \ref{mst} for $l=2$ and $\overline{S\nu_k}:=\overline{S\nu}|_{M_k}$.
The obtained diffeomorphism commutes with $\overline{S\nu_k}$ and so is orientation-preserving.
\end{proof}

\comment

 {\bf Remark.}
 There is a natural forgetful map $\gamma \colon B\Spin \to B\SO$.
For an oriented manifold $P$, a {\em normal smoothing} of $P$ in $B\Spin \times \C P^\infty$ is a map $\overline{S\nu}\colon P\to B\Spin\times\C P^\infty$ such that $\gamma\circ\pr_1\circ\overline{S\nu}\colon P\to B\SO$ is the normal Gauss map for some embedding of $P$ into $\R^\infty$
(i.e.~is a classifying map of the stable normal bundle of $P$).
A {\em normal $k$-smoothing} is a  normal smoothing which is $(k{+}1)$-connected.

When necessary in the surgery arguments of Proof of the Almost Diffeomorphism Theorem \ref{t:aldi} (\S\ref{s:lemmas-adtmod}), we pass to the language of normal smoothings.
If $P$ is $p$-dimensional then up to natural equivalence, the normal 2?-smoothings
inducing an isomorphism on $\pi_1$ (i.e.~$\pi_1(P)=0$?), an epimorphism on $\pi_2$ and such that
$\overline{S\nu}\colon P \to B\Spin \times \C P^\infty$ are in 1--1 correspondence to pairs of a spin structure on
$P$ and a homology class $z \in H_{p-2}(P, \de)$.
This 1--1 correspondence assigns to $\overline{S\nu}$ the spin structure  $\pr_1\circ\overline{S\nu}$ and
the class homology class $z \in H_{p-2}(P, \de)\cong H^2(P)$ is determined by the homotopy class of the map $\pr_2\circ\overline{S\nu}$.

\endcomment

\subsection{Proof of the Pre-elementary class Lemma \ref{l:modif}}\label{s:lemmas-mod}

We first construct a homomorphism $s$ satisfying (1) from the definition of a pre-elementary class
(and some additional properties).
Then we show how to achieve (2) keeping (1), and finally we show how to achieve (3) keeping (1) and (2).

\begin{Lemma}\label{l:modules}
Let $V,V'$ be free abelian groups, $\cdot\colon V\times V'\to\Z$ a unimodular form, $j\colon V\to V'$ a homomorphism whose image
is a direct summand and $(\im j)^\perp=\ker j$.

A homomorphism $s\colon V'\to V$ is called {\it 1-homomorphism} if
$$V=\ker j\oplus\im s,\quad jsj=j\quad\text{and}\quad sjs=s.$$
\ \quad
A homomorphism $s \colon V'\to V$ is called a {\it 2-homomorphism}  if
$$V=\ker j\oplus\im s\quad\text{and}\quad x\cdot jsv=x\cdot v\quad\text{for every}\quad x\in\im s,v\in V'.$$

(a) There is a 1-homomorphism.

(b) For a 1-homomorphism $t\colon V'\to V$
define a homomorphism $t^*\colon \im t\to V$ by the property
$$t^*x\cdot u=x\cdot jtu\quad\text{for every}\quad u\in V'.$$
Then $t^*t$ is a 2-homomorphism.

(c) If $s$ is a 2-homomorphism, then

\quad (c1) $jsj=j$;

\quad (c2) $V'=\im j\oplus\ker s$;

\quad (c3) $sjs=s$;

\quad (c4) $t$ is a 1-homomorphism for every homomorphism $t\colon V'\to\im s$ such that $tj=sj$.
\end{Lemma}

\begin{proof}[Proof of (a)]
Since $V$ and $V'$ are free abelian, there is a subgroup $T\subset V$ such that $V=\ker j\oplus T$.
Then $j|_T$ is injective and $j(T)=\im j$.
Since $j(T)=\im j$ is a direct summand in $V'$, the inverse of the abbreviation $j\colon T\to j(T)$ extends
to an epimorphism $t\colon V'\to T$.
We have $tjt=t$ and $jtj=j$.
\end{proof}

\begin{proof}[Proof of (b)]
Denote $s:=t^*t$.
Take any $x\in\im t$.
Since $jtj=j$, we have $t^*x\cdot jtu=x\cdot jtjtu=x\cdot jtu$ for every $u\in V'$.
Hence $t^*x-x\perp \im(jt)$.
Since $V=\ker j\oplus\im t$, we have $\im(jt)=\im j$.
Then $t^*x-x\in \ker j$, i.e.~$jt^*x=jx$.
Since $tjt=t$, we have $tjt^*x=tjx=x$.

Since $t^*x-x\in \ker j$, we have $V=\ker j+\im t^*$.
If $jt^*x=0$, then $jx=0$, and consequently $x\in\ker j\cap\im t=\{0\}$.
Hence $V=\ker j\oplus\im t^*$.
Since $\im s=t^*\im t=\im t^*$, we obtain $V=\ker j\oplus\im s.$

Since $$t^*x\cdot jt^*y=x\cdot jtjt^*y=x\cdot jy\quad\text{for all}\quad x,y\in \im t,$$
$$\text{we have}\quad su\cdot jsv=t^*tu\cdot jt^*tv = tu\cdot jtv=t^*tu\cdot v=su\cdot v\quad\text{for all}\quad u,v\in V'.$$
\end{proof}

\begin{proof}[Proof of (c1)]
Take any $y\in V$.
Since $\ker j\perp\im j$, we have $x_1\cdot jy=0=x_1\cdot jsjy$ for every $x_1\in\ker j$.
Also $x_2\cdot jy=x_2\cdot jsjy$ for every $x_2\in\im s$.
Since $V=\ker j\oplus\im s$, we have $x\cdot jy=x\cdot jsjy$ for every $x\in V$.
Then by the unimodularity of $\cdot \colon V \times V' \to \Z$ we have $j=jsj$.
\end{proof}

\begin{proof}[Proof of (c2)] If $sjx=0$, then $jsjx=0$.
So by (c1) $jx=0$.
Therefore $\im j\cap\ker s=0$.
Since $\im j$ is a direct summand, by rank considerations $V'=\im j\oplus\ker s$.
\end{proof}

\begin{proof}[Proof of (c3)] By (c2) $\im s=s\im j$. Also $\im j=j\im s$.
So the abbreviations $j\colon \im s\to\im j$ and $s\colon \im j\to\im s$ are surjective.
Hence
$$jsj=j\quad\Leftrightarrow\quad js|_{\im j}=\id(\im j)\quad\Leftrightarrow\quad  sj|_{\im s}=\id(\im s)\quad\Leftrightarrow\quad sjs=s.$$
\end{proof}

\begin{proof}[Proof of (c4)] We have $jtj=jsj=j$ by (c1).
We have $\im t\supset tj(V)=sj(V)=\im s$ by (c2).
Hence $V=\ker j\oplus\im t$ and $tj|_{\im s}=sj|_{\im s}=\id(\im s)$ by (c3).
\end{proof}

\begin{proof}[Proof of the Pre-elementary class Lemma \ref{l:modif}]
Since $H_3(\de W)$ is free abelian, $\im j$ is a direct summand in $H_4(W,\de)$.
Apply Lemma \ref{l:modules}.ab to $V=H_4(W)$, $V'=H_4(W,\de)$, $\cdot=\capM{\de}$ and $j=j_W$.
We obtain a 2-homomorphism $s$.
Let us show how to modify $(W,z,s)$ to achieve property (3) from the definition of a pre-elementary class.

Since $\rho_d\de_Wz^2=0$, $\rho_dz^2\in\rho_d\im j_W$.
Hence there are $a\in H_4(W)$ and $b\in H_4(W,\de)$ such that $z^2=j_Wa+db$.
So $\rho_dj_Wsz^2=\rho_dj_Wsj_Wa\overset{(2)}=\rho_dj_Wa=\rho_dz^2$.
Here (2) holds by Lemma \ref{l:modules}.c1.
Since the residues in the Pre-elementary class Lemma \ref{l:modif} are independent of the choice of
$\overline{z^2}\in j_W^{-1}\rho_dz^2$, we may take $\overline{z^2}:=\rho_dsz^2$.

Below we prove that

(a) we can change $\eta(z,s):=sz^2\capM{\de}(z^2-p_W^*)$ by $2d$ without changing $sz^2\capM{\de} z^2$;

(b) we can simultaneously change $\eta(W,z,s)$ by $d^2-d$ and $sz^2\capM{\de} z^2$ by $d^2$.

If $d$ is odd, applying (b) we make $sz^2\capM{\de} z^2$ even keeping $\eta(z,s)$ divisible by $\widehat d=\gcd(d,3)$.
Then applying (a) we can change $\eta(z,s)$ by $2d$ keeping $sz^2\capM{\de} z^2$ even.

If $d$ is even, $sz^2\capM{\de} z^2$ is even by the hypothesis.
Applying (a,b) we can change $\eta(z,s)$ by $\gcd(2d,d^2-d)=d$ keeping $sz^2\capM{\de} z^2$ even.

Take $(S^2)^4$ and the class $z_S$ which is the sum of four summands, each represented by a product of three 2-spheres and a point.
Then $z_S^4=24$.
Since $(S^2)^4$ is almost parallelizable, we have $p_{(S^2)^4}^*=0$.
Taking connected sums with copies of $((S^2)^4, z_S)$ we can change $\eta(z,s)$ by
any multiple of 24 while keeping $sz^2\capM{\de} z^2$ even.

So we obtain that $\eta(z,s)=0$ and $sz^2\capM{\de} z^2$ is even.

By \cite[spin case of (2.4) and Proposition 2.5]{KS91} there is a closed spin 8-manifold $W_0$ and $z_0\in H_6(W_0)$ such that
$z^2_0\cap_{W_0}(z^2_0-p_{W_0}^*)=0$ and $z^2_0\cap_{W_0}z^2_0=2$.
Taking connected sums with copies of $(W_0, z_0)$ we can change $sz^2\capM{\de} z^2$ by any multiple of 2  without changing $\eta(W,z,s)$.
So we can obtain $sz^2\capM{\de} z^2=0$ while keeping $\eta(z,s)=0$.

Let $\Ha P^2$ be quaternionic projective space oriented so that its signature
is given by $\sigma(\Ha P^2)=1$.
Recall that $\Ha P^2$ is 3-connected and $(p_{\Ha P^2}^*)^2=1$
\cite[Lemmas 3 and 4]{Mi56}.
There is a 3-connected parallelizable 8-manifold $\overline E_8$ whose boundary is
a homotopy sphere and whose signature is 8.
Then $p_{\overline E_8}^*=0$.

Since $\de W$ is parallelizable, $\de_Wp_W^*=0$.
By \cite[Lemma 2.11.b]{CS11} $sp_W^*$ is a characteristic element for $\cap_W|_{\im s}$.
Hence by Lemma \ref{l:sinter}.a $\sigma(W)=\sigma(\cap_W|_{\im s})\underset{\mod8}\equiv sp_W^*\capM{W} sp_W^*$.
Therefore taking connected sums with copies of $\Ha P^2$  and $\overline E_8$ we can achieve
$\sigma(W)=sp_W^*\capM{\de} p_W^*=0$ while keeping $\eta(z,s)=sz^2\capM{\de} z^2=0$.%
\footnote{This covers a minor gap in \cite[\S4]{CS11}: there we needed additionally to take connected sums with the
$\overline E_8$-manifold to kill $\alpha_W$, and so $\de W$ will in general be changed by connected sum with a homotopy sphere.}
\end{proof}

\begin{proof}[Proof of (b)]
Denote $W_1:=W\#\Ha P^2\#(-\Ha P^2)$.
We have $H_4(\Ha P^2\#(-\Ha P^2))\cong\Z^2$ with evident basis.
In this basis the intersection form is $\diag(1,-1)$ and $p_{\Ha P^2\#(-\Ha P^2)}=(1,1)$.
Let $z_1$ be the preimage of $z$ under the `connected sum' isomorphism $H_6(W_1,\de)\to H_6(W,\de)$.
In order to construct the new $s$ (this is $t^*t$ not $s_1$, both defined below) let us define the lower two lines of the following diagram:
$$\xymatrix{
{z_1^2\choose p_{W_1}^*} \ar[d]^\in \ar@{|->}[r]^{c_\de} \ar@(u,u)[rrr]^{ct}
& {z^2,0,0\choose p_W^*,1,1} \ar[d]^\in \ar@{|->}[r]^(0.45){s':=(s\oplus\de)\oplus\id}
& {sz^2,\de_Wz^2,0,0\choose sp_W^*,0,1,1} \ar[d]^\in \ar@{|->}[r]^(0.525){t':=i\oplus(t''\oplus\id)}
& {sz^2,0,d\choose sp_W^*,1,1} \ar[d]^\in \\
H_4(W_1,\de) \ar[r]^(.425){c_\de,\cong} \ar[dr]^{t} \ar@(d,l)@{-->}[dr]_{s_1,t^*t}
& H_4(W,\de)\oplus\Z^2 \ar[r]^(.425){s',\cong} \ar@(d,d)@{-->}[rr]^{s\oplus\id\ (\ne t's')}
& \im s\oplus H_3(\partial W)\oplus\Z^2 \ar[r]^(.6){t'} & H_4(W)\oplus\Z^2 \\
& H_4(W_1) \ar@(r,d)[rru]_{c,\cong}
 }$$
Let $c_\de$ and $c$ be the `connected sum' orthogonal isomorphisms (for the form $\diag(1,-1)$ on $\Z^2$).
Let $\id:=\id\Z^2$.
Let $s'(u,a,b):=s(u)\oplus\de u\oplus(a,b)$.
Since $H_3(W)=0$, by Lemma \ref{l:modules}.c2 $s\oplus\de\colon H_4(W,\de)\to \im s\oplus H_3(\partial W)$ is an isomorphism.
Hence $s'$ is an isomorphism.

Since $H_3(\de W)$ is free abelian and $d=\di(\de_Wz^2)$, there is  a map
$$t''\colon H_3(\de W)\to
\Z^2\quad\text{such that}\quad t''(\de_Wz^2)=(0,d).$$
Let $t'(u,v,a,b):=u\oplus(t''(v)+(a,b))$.
Let
$$V:=H_4(W_1),\quad V':=H_4(W_1,\de),\quad \cdot:=\cap,\quad j:=j_{W_1},\quad s_1:=c^{-1}(s\oplus\id)c_\de,
\quad t:=c^{-1}t's'c_\de,$$
so that the undashed lines of the diagram commute.
Then $c\im s_1=\im s\oplus\Z^2=\im t'=c\im t$.
Clearly, $s_1$ is a 2-homomorphism.
 Also
$$tj=c^{-1}t's'c_\de j=c^{-1}t's'(j_W\oplus\id)c=c^{-1}t'((sj_W\oplus0)\oplus\id)c=c^{-1}s(j_W\oplus\id)c=
c^{-1}(s\oplus\id)c_\de j=s_1j.$$
Hence by Lemma \ref{l:modules}.b,c4 for $s_1$ we obtain that $t^*t$ is a 2-homomorphism.

For every $u_1,u_2\in H_4(W_1,\de)$ by definition of $t^*$ we have
$$t^*tu_1\cap_{W_1} u_2=tu_1\cap_{W_1} tu_2=v_1\cap_Wv_2+a_1a_2-b_1b_2,\quad\text{where}\quad (v_k,a_k,b_k)=ctu_k.$$
Clearly, the images of $z_1^2$ are as shown in the first line of the diagram.
Since $\de W$ is parallelizable, the images of $p_{W_1}^*$ are as shown in the first line of the diagram.
Hence
$$t^*tz_1^2\cap_{W_1}z_1^2=sz^2\cap_W sz^2-d^2 \quad\text{and}\quad \eta(z_1,t^*t)-\eta(z,s)=0\cdot(0-1)-d\cdot(d-1)=-d^2+d.$$
\end{proof}

\begin{proof}[Proof of (a)]
By \cite{Mi56} there is a $D^4$-bundle over $S^4$ whose Euler class is 0 and whose first Pontryagin class is 4.
The double of this bundle is an $S^4$-bundle $S^4\t \times S^4$ over $S^4$ whose first Pontryagin class is 4.
We have $H_4(S^4\t \times S^4)\cong\Z^2$ with evident basis.
In this basis $p_{S^4\t \times S^4}=(2,0)$ and the intersection form is $\left(\begin{matrix}0&1\\1&0\end{matrix}\right)$.
Analogously to the proof of (b) with $\Ha P^2\#(-\Ha P^2)$ replaced by $S^4\t \times S^4$ we construct $W_1,z_1$ and $t^*t$.
Then for every $u_1,u_2\in H_4(W_1,\de)$ we have
$$t^*tu_1\cap_{W_1} u_2=tu_1\cap_{W_1} tu_2=v_1\cap_Wv_2+a_1b_2+a_2b_1,\quad\text{where}\quad (v_k,a_k,b_k)=ctu_k.$$
Also $\displaystyle ct{z_1^2\choose p_{W_1}^*}={sz^2,0,d\choose sp_W^*,2,0}$.
Then
$t^*tz_1^2\cap_{W_1}z_1^2=sz^2\capM{W} sz^2\quad\text{and}\quad \eta(z_1,t^*t)=\eta(z,s)-2d.$
\end{proof}

\subsection{Proof of the Elementary pair Lemma \ref{l:easur}}\label{s:lemmas-ele}

\begin{Lemma}\label{l:sinter}
Let $W$ be an 8-manifold satisfying the assumptions (*) and (**) of the Elementary pair  Lemma \ref{l:easur}.
Let $s\colon H_4(W,\de)\to H_4(W)$ be a homomorphism such that
$H_4(W)=\im i_W\oplus\im s$ (additively which implies orthogonally w.r.t. $\cap_W$).
Denote $j_k:=j_{M_k,W}$ and $S:=\im s$.
Denote by the superscript $\perp$ the orthogonal complement with respect to $\capM{01}$, unless another intersection product is indicated as subscript.
Then

(a)  $S$ is free abelian and the form $\cap_W|_S$ is unimodular;

(b) $j_0|_S,j_1|_S$ are injective,
$H_4(W,M_k)=(j_{1-k}S)^\perp\oplus j_kS$, $k=0,1$, and the restrictions of $\capM{01}$ both to
$j_0S\times j_1S$ and to $(j_1S)^\perp\times (j_0S)^\perp$ are unimodular;

(c) $j_k\im i_W$ is a half-rank direct summand in the free abelian group $(j_{1-k}S)^\perp$;

(d) if
$$a\in H_4(W,\de),\quad q\in \ker j_{M_0,\de W}\subset H_4(\de W)\quad\text{and}
\quad q\cap_{\de W} \de_Wa=\di(\de_Wa),$$
then there is a subgroup $U\subset\im i_W$ such that $U\capM{\de} a=0$ and the pair $j_k|_U\colon U\to(j_{1-k}S)^\perp$,
$k=0,1$, is elementary.
\end{Lemma}

\begin{proof}[Proof of (a)]
Since the torsion of $H_4(W)$ is contained in $\im i_W=H_4(W)^\perp_{\cap_W}$, the group $S$ is free abelian.
Since  $\capM{\de}\colon H_4(W)\times H_4(W,\de)\to\Z$ is unimodular, $x\cap_W y=x\capM{\de} j_Wy$
for all $x,y\in H_4(W)$
and $H_3(\de W)$ is free abelian,
it follows that the form $\cap_W|_S$ is unimodular.
\end{proof}

\begin{proof}[Proof of (b)]
Since $j_0x\capM{01} j_1y=x\cap_W y$ for all $x,y\in H_4(W)$, it follows that $\capM{01}|_{j_0S\times j_1S}$ is unimodular.
Then $j_0|_S$ and $j_1|_S$ are injective.
So if $x\in S$ and $j_kx\capM{01} j_{1-k}S=0$, then $x=0$.
Also, for every  $y\in H_4(W,M_k)$ the $\capM{01}$-intersection with $y$ defines a linear map $j_{1-k}S\to\Z$.
Hence there is a class $y_S\in j_kS$ such that $y\capM{01} x=y_S\capM{01} x$ for every $x\in j_{1-k}S$.
Then $y=y_S+(y-y_S)$ and $(y-y_S)\capM{01} j_{1-k}S=0$.
Thus $H_4(W,M_k)=(j_{1-k}S)^\perp\oplus j_kS$, $k=0,1$.
Since $\capM{01}$ and $\capM{01}|_{j_0S\times j_1S}$ are unimodular,  $\capM{01}|_{(j_1S)^\perp\times (j_0S)^\perp}$ is unimodular.
\end{proof}

{\it Some notation for the proofs of Lemmas \ref{l:sinter}.c,d and \ref{l:easur}.}
Denote by $i_{W,k},\partial_{W,k}$ homomorphisms from the exact sequence of the triple $(W,\partial W,M_k)$.
Denote by $i_k,j_k,\de_k$ and $\t  i_k,\t  j_k,\t \partial_k$ the homomorphisms from the exact sequences of
the pairs $(W,M_k)$ and $(\partial W,M_k)$, respectively.
Recall that $H_4(W)^\perp_{\cap_W}=\im i_W$.
Consider the following diagram:
$$\xymatrix{&
&H_5(W,\de)=0 \ar[d]^{\de_W} \ar[dr]^{\de_{W,k}} &
& & H_3(W)=0  \\
& H_4(M_k) \ar[r]^{\t  i_k} \ar[rd]^{i_k} & H_4(\de W) \ar[d]^{i_W} \ar[r]^{\t  j_k}
&H_4(\de W,M_k) \ar[d]^{i_{W,k}} \ar[r]^{\t \de_k=0}
& H_3(M_k) \ar[ru]^{i_k}
\\
& &H_4(W) \ar[r]^{j_k} \ar@{=}[d]  &H_4(W,M_k) \ar@{=}[d] \ar[ru]^{\de_k} & \\
& &\im i_W\oplus S \ar[r]^{j_k}    & (j_{1-k}S)^\perp\oplus j_kS &
}$$

\begin{proof}[Proof of (c)]
We have
$$\frac{(j_{1-k}S)^\perp}{j_k\im i_W}\overset{(1)}\cong
\coker j_k
\overset{(2)}\cong H_3(M_k)\overset{(3)}\cong H_4(\partial W,M_k) \overset{(4)}
\cong\im i_{W,k}\overset{(5)}=j_k\im i_W.$$
Here

$\bullet$ (1) is obtained by adding $j_kS$ both to nominator and denominator and using (b);

$\bullet$ (2) holds because $H_3(W)=0$, hence $\partial_k$ is an epimorphism;

$\bullet$ (3) holds by Poincar\'{e}-Lefschetz duality because both $H_3(M_k)$ and
$H_4(\de W,M_k)\overset{\ex}\cong H_4(M_{1-k},\de)$ are free abelian;

$\bullet$ (4) holds because $H_5(W,\de)=0$, hence $i_{W,k}$ is injective;

$\bullet$ (5) holds because $\t  j_k$ is surjective.

Since $H_3(M_k)$
is free abelian, $j_k\im i_W\cong \dfrac{(j_{1-k}S)^\perp}{j_k\im i_W}$ is free abelian.
This implies (c).
\end{proof}

\begin{proof}[Proof of (d)]
Since $M_k$ is torsion free, by Poincar\'{e}-Lefschetz duality
$H_4(\de W,M_0)\overset{\ex}\cong H_4(M_1,\de)$ is free abelian.
Since $\t  j_0$ is surjective, it follows that there is a subgroup
$$U''\subset H_4(\de W)\quad\text{such that}\quad \t  j_0|_{U''}\colon U''\to H_4(\de W,M_0)\quad\text{is an isomorphism.}$$
Since $H_3(\partial W)$ is free abelian, there is a class
$$a_0\in H_3(\partial W)\quad\text{such that}\quad\de_Wa=a_0\di(\de_Wa).$$
$$\text{Define}\quad U':=\{u-(a_0\cap_{\de W} u)q\ :\ u\in U''\}\quad\text{and}\quad U:=i_WU'.$$
Since $q\cap_{\de W} \de_Wa=\di(\de_Wa)$, we have $\de_Wa\cap_{\de W} U'=0$.
Thus $U\capM{\de} a=0$.
So by (c) it remains to prove that
{\it $j_k|_U$ is an isomorphism onto $j_k\im i_W$.}

Since $\ker\t  j_0 = \ker \t  j_1$, the map $\t  j_1|_{U''}$ is injective.
Then $U''\cap_{\de W}\im\t  i_k=0$.
This and the fact that $q\in \im\t  i_0=\im\t  i_1=\ker\t  j_0 = \ker \t  j_1$ imply that $U'\cap_{\de W} \im\t  i_0=0$.
Since $\ker\t  j_0 = \ker \t  j_1$, we have $\im\t  i_0=\im\t  i_1$.
 Therefore $U'\cap_{\de W}\im\t  i_1=0$.
Thus  $\t  j_k|_{U'}$ is injective.
Since $H_5(W,\partial)=0$, the map $i_W$ is injective.
Hence $U\cap_W i_W\ker\t  j_k=0$.
We have $i_W\ker\t  j_k=i_W\im\t  i_k=\im i_k=\ker j_k$.
Thus  $j_k|_U$ is injective.

Since $\ker\t  j_0 = \ker \t  j_1$, we have
$H_4(\partial W)=U''+\ker\t  j_0=U''+\ker\t  j_1$.
Since $q\in \ker\t  j_0 = \ker \t  j_1$, it follows that
$H_4(\partial W)=U'+\ker\t  j_0=U'+\ker\t  j_1$.
So $\t  j_1U'=\t  j_1H_4(\partial W)$.
Therefore we have
$j_kU=j_ki_WU'=i_{W,k}\t  j_kU'=i_{W,k}\t  j_kH_4(\partial W)=j_k\im i_W$.
\end{proof}

\begin{proof}[Proof of the Elementary pair Lemma \ref{l:easur}]
The group $H_4(W,M_k)$ is torsion free for $k=0,1$.
(Indeed, consider the Poincar\'e dual of the exact sequence of the pair $(W,M_{3-k})$:
$$H_5(W,\de)\to H_4(M_{3-k},\de) \to H_4(W,M_k) \to H_4(W,\de).$$
By the assumptions
$$H_5(W,\de)=0,\quad \Tors H_4(M_{3-k},\de)=\Tors H_2(M_{3-k})=0\quad\text{and}\quad
\Tors H_4(W,\de)=\Tors H_3(W)=0.$$
Hence $H_4(W,M_k)$ is torsion free.)

Since $z$ is pre-elementary, there is a homomorphism $s$ from the definition of a pre-elementary class.
Denote $S:=\im s$.
Let
$$\widehat U:=\{u\in S\ |\ lu=msz^2+nsp_W^*\text{ for some integers }l,m,n\}.$$
Since $z$ is pre-elementary, $\widehat U\cap_W\widehat U=0$.
By Lemma \ref{l:sinter}.a $S$ is free abelian and the form $\cap_W|_S$ is unimodular.
Then there is a subgroup
$$T\subset S\quad \text{such that}\quad \widehat U\subset T,
\quad \rk T=2\rk \widehat U \quad \text{and \quad $\cap_W|_T$ is unimodular}.$$
Hence $\sigma(T)=0$.
Since both $\cap_W|_S$ and $\cap_W|_T$ are unimodular, $T\cap T^\perp_{\cap_W}=0$ and
$\rk T^\perp_{\cap_W}=\rk S-\rk T$, we have $S=T\oplus T^\perp_{\cap_W}$.
So $\sigma(T^\perp_{\cap_W})=\sigma(W)-\sigma(T)=0$.
Hence there is a half-rank direct summand
$$\t  U\subset T^\perp_{\cap_W}\quad\text{such that}\quad\t  U\cap_W \t  U=0.$$
Let $U_S:=\widehat U\oplus \t  U$.

We have {\it $U_S\capM{\de} z^2=U_S\capM{\de} p_W^*=0$ and the pair $j_k|_{U_S} \colon U_S\to j_kS$, $k=0,1$,
is elementary.}

(Indeed, since $z$ is pre-elementary, $\widehat U\capM{\de} z^2=\widehat U\capM{\de} p_W^*=0$.
Also $\t  U\cap_W\widehat U=0$.
Hence by the properties (2) and (3) of $s$ we obtain $U_S\capM{\de} z^2=U_S\capM{\de} p_W^*=0$.
Since $\widehat U\cap_W\widehat U=0=\t U\cap_W\t U$, we have $j_0U_S\capM{01} j_1U_S=0$.
By Lemma \ref{l:sinter}.b $j_k|_S$ is injective.
Since $\widehat U$ and $\t  U$ are half-rank direct summands in $T$ and in $T^\perp_{\cap_W}$, respectively,
$U_S$ is a half-rank direct summand in $S$.
So $j_kU_S$ is a half-rank direct summand in $j_kS$.)

Applying Lemma \ref{l:sinter}.d to $a=z^2$ we obtain a subgroup $U_{\de}\subset\im i_W$.
Since $\partial W$ is parallelizable, $p_1(\partial W)=0$.
Hence $\im i_W\capM{\de} p_W^*=0$.
Therefore $U:=U_S\oplus U_{\de}$ is as required.
\end{proof}


\noindent{\it Diarmuid Crowley, Institute of Mathematics, University of Aberdeen,
United Kingdom. University of Melbourne, Australia.}

Email: \texttt{dcrowley@unimelb.edu.au}.
Web: \url{www.dcrowley.net}

\noindent{\it Arkadiy Skopenkov,
Moscow Institute of Physics and Technology,
and Independent University of Moscow, Russia.}

Email: \texttt{skopenko@mccme.ru}.
Web: \url{www.mccme.ru/~skopenko}

\end{document}